\documentclass[12pt]{amsart}
\usepackage{amsfonts, amsbsy, amsmath, amsthm, amssymb, latexsym, enumerate,verbatim,mathrsfs}
\usepackage{color,hyperref,enumerate}
\usepackage[pagewise]{lineno}%\linenumbers
\hypersetup{pdfauthor={Name}}

\usepackage[top=30mm,right=25mm,bottom=30mm,left=25mm]{geometry}

\numberwithin{equation}{section}

\newtheorem{theorema}{Theorem}

\newtheorem{proposition}[theorema]{Proposition}
\newtheorem*{hh-thm}{Hall-Higman Theorem}
\newtheorem{propo}{Proposition}[section]

\newtheorem{lemma}[propo]{Lemma}
\newtheorem{corol}[propo]{Corollary}

\newtheorem{theo}[propo]{Theorem}
\newtheorem{examp}[propo]{Example}
\newtheorem{remar}[propo]{Remark}

\newcommand{\ld}{,\ldots ,}

\newcommand{\ra}{ \rightarrow }

\newcommand{\lan}{ \langle }
\newcommand{\ran}{ \rangle }

\newcommand{\diag}{\mathop{\rm diag}\nolimits}

\newcommand{\Id}{\mathop{\rm Id}\nolimits}
\newcommand{\Irr}{\mathop{\rm Irr}\nolimits}

\newcommand{\spec}{\mathop{\rm Spec}\nolimits}
\newcommand{\Syl}{\mathop{\rm Syl}}
\newcommand{\EC}{\mathcal{E}}
\newcommand{\CC}{\mathbb{C}}
\newcommand{\FF}{\mathbb{F}}
\newcommand{\GG}{\mathbf{G}}
\newcommand{\SB}{\mathbf{S}}

\newcommand{\QQ}{\mathbb{Q}}

\newcommand{\ZZ}{\mathbb{Z}}

\newcommand{\al}{\alpha}
\newcommand{\be}{\beta}

\newcommand{\ep}{\varepsilon}
\newcommand{\eps}{{\varepsilon}}
\newcommand{\lam}{\lambda }
\newcommand{\om}{\omega}

\newcommand{\up}{^{-1}}

\newcommand{\si}{\sigma }

\def\d12{{_{12}}}

\def\acf{{algebraically closed field }}

\def\au{{automorphism }}

\def\ccc{{constituent }}

\def\ei{{eigenvalue }}
\def\eis{{eigenvalues }}

\def\f{{following }}

\def\ii{{if and only if }}

\def\ir{{irreducible }}

\def\irt{{irreducible. }}
\def\irr{{irreducible representation }}

\def\itf{{It follows that }}

\def\mult{{multiplicity }}

\def\po{{polynomial }}

\def\rep{{representation }}
\def\reps{{representations }}

\def\syls{{Sylow $p$-subgroups }}
\def\syl{{Sylow $p$-subgroup }}

\newcommand{\SC}{\mathbf{S}}
\newcommand{\Fr}{\mathsf{F}}
\renewcommand{\setminus}{\smallsetminus}

\newcommand{\edit}[1]{{\color{red} #1}}

\def\PSL{{\rm PSL}}%\nolimits}
\def\SL{{\rm SL}}%\nolimits}
\def\SO{{\rm SO}}%\nolimits}
\def\Spin{{\rm Spin}}%\nolimits}
\def\PSU{{\rm PSU}}%\nolimits}
\def\SU{{\rm SU}}%\nolimits}
\def\PSp{{\rm PSp}}%\nolimits}

\def\GL{{\rm GL}}%\nolimits}
\def\Sp{{\rm Sp}}%\nolimits}
\def\F{{\mathsf F}}%\nolimits}
\def\U{{\rm U}}%\nolimits}

\def\2k{2^k}
\def\st{Suppose that }

\newcommand{\bp}{\begin{proof} }
\newcommand{\enp}{\end{proof}}
\newcommand{\bl}{\begin{lemma}\label}
\newcommand{\el}{\end{lemma}}

\newcommand{\med}{\medskip}

\newcommand{\TC}{{\mathbf T}}

\newcommand{\StC}{\mathsf{St}}

\begin{document}

 \title[Minimal polynomials of $p$-elements of finite groups of Lie type]{Minimal polynomials of $p$-elements  of finite\\ groups of Lie type with cyclic Sylow $p$-subgroups} %in cross-characteristic representations}

\author[Pham Huu Tiep]{Pham Huu Tiep}
\address{Department of Mathematics, Rutgers University, Piscataway, NJ 08854}
\email{tiep@math.rutgers.edu}

\author[A.E.  Zalesski]{Alexandre Zalesski}
\address{Department of Mathematics,  University of Brasilia, Brasilia-DF, Brazil}
\email{alexandre.zalesski@gmail.com}

%\thanks{e-mail:
%Address correspondence to
 %E-mail:   alexandre.zalesski@gmail.com (Alexandre Zalesski)}


%file zt14(3)-2025-03-14

\dedicatory{Dedicated to the memory of Gary Seitz}

\thanks{The first author gratefully acknowledges the support of the NSF (grant
DMS-2200850) and the Joshua Barlaz Chair in Mathematics.}

\thanks{This work was started during the authors' visit to the Isaacs Newton Institute for Mathematical Sciences, Cambridge, for their participation in the programme ``Groups, representations and applications: new perspectives''. It is a pleasure to thank the Isaac Newton Institute  for the generous hospitality and stimulating 
environment during the programme supported by EPSRC grant no EP/R014604/1.}

\thanks{The authors are grateful to Gerhard Hiss for his comments on results of \cite{GH} from which he deduced the explicit shape of the 3-decomposition matrix for $G=D_4^-(q)$ for $3|(q-1)$ reproduced in Section 4.}

%\thanks{The authors are grateful to the referee for careful reading and many comments that helped greatly improve the exposition of the paper.}

\begin{abstract}
Extending earlier results of the authors on minimal polynomials of $p$-elements of finite groups of Lie type in cross-characteristic representations, this paper focuses on the case where Sylow $p$-subgroups are cyclic and $p$ is distinct from the representation field characteristic.  %The main result states that if $g$ is a semisimple $p$-element of a quasisimple group $G$ of Lie type with cyclic Sylow $p$-subgroup and $\phi$ is a non-trivial cross-characteristic \irr  of $G$ then the minimal polynomial degree of $\phi(g)$ is at most $|g|-1$.
\end{abstract}
\maketitle  

\tableofcontents

\section{Introduction}

This paper   extends   and refines  some results in our earlier papers \cite{Z3}, \cite{TZ8} and  \cite{TZ20}. We focus on \ir $\ell$-modular \reps of  quasi-simple finite groups of Lie type $G=G(q)$ whose Sylow $p$-subgroups are cyclic for some prime $p\neq \ell,p\nmid q$.

The main result of the paper can be described as follows. 
 
\begin{theorema}\label{mm1} Let  $G\neq A_1(q)$ be a universal quasisimple group of Lie type in 
characteristic $r$, and let $\phi$ be a non-trivial absolutely irreducible representation of $G$ in characteristic $\ell \neq r$. Let $g\in G$ be a $p$-element with $p\neq r,\ell$. Suppose that Sylow   $p$-subgroups of $G$ are cyclic. Then the degree $\deg \phi(g)$ of the minimal \po  of $\phi(g)$ is equal to $|g|$ or $|g|-1$. 
\end{theorema}

Here, by a {\it universal quasisimple group of Lie type in characteristic $r$} we mean the group $G=\mathbf{G}^\F :=\{g\in \mathbf{G} \mid \F(g)=g\}$, where $\mathbf{G}$ is a simple, simply connected algebraic group in characteristic $r>0$ and 
$\F:\mathbf{G} \to \mathbf{G}$ is a Steinberg endomorphism (in the sense of \cite[Definition 21.3]{MaT}).
In the special case where $|g|$ is a prime  this result was obtained in our previous work \cite{TZ8}, so here we extend it to $|g|$  an arbitrary prime power. Note that \cite{TZ8}
contains a number of partial results for $|g|$ being a prime power, in particular, the case $G$ of type $A_1$ has been settled in \cite[Lemma 3.3]{TZ8}. 

 The problem on the minimal \po degree of $p$-elements in finite linear groups and group \reps goes back to the famous Hall--Higman theorem  \cite{HH}. We state it here for $p>2$:

\begin{hh-thm}  Let G be a finite p-solvable  \ir linear group over a field of characteristic $p>2$, and $g\in G$ an element of order $p^k$. Then $\deg g$, the degree of the minimal \po of g, is at least $(p-1)p^{k-1}$.
\end{hh-thm}

This is an important result which was used in an essential way in  the classification of finite simple groups. It has many refinements and is largely discussed in the literature (see for instance \cite[Ch. IX, \S 2]{HB} and \cite[Ch. VII, \S 10]{Fe}). In particular, it was observed
that a similar result holds for an arbitrary  ground field provided the $p$-element
in question does not lie in any abelian normal subgroup of $G$ and $G$ is $p$-solvable
(see \cite[Ch. VII, Theorem 10.2]{Fe} for a slightly weaker result). The  problem of obtaining a sharp lower bound  for the minimal \po degree of $p$-elements in arbitrary groups attracted a lot of attention too.
The bulk of the problem is to study quasisimple $G$, where  groups of Lie type in %defining
 characteristic $p$ are excluded so that one can expect a
lower bound similar to that of Hall and Higman.  

The project of extending  the Hall--Higman theorem 
 to arbitrary finite groups was started in \cite{Z3} and significantly developed further  in \cite{DZ1} and \cite{TZ8}.
(If $G$ is of Lie type then $\ell$ is assumed to differ from the defining characteristic $r$ of $G$.)  Both  \cite{DZ1} and \cite{TZ8} deal with quasi-simple classical groups $G$  and semisimple $p$-elements $g\in G$, where we obtain the Hall--Higman lower bound $|g|(p-1)/p$ in almost all situations.
 On many occasions the degree in question is equal to the order of $g$ in $G/Z(G)$.  Moreover,  Sylow $p$-subgroups of $G$ are cyclic in most cases where the degree remains undetermined. The quasisimple groups of Lie type of small rank are considered in \cite{TZ20}.

In the current paper we focus on the case where Sylow $p$-subgroups of $G$ are cyclic.  
Note that for $p=\ell$, and for all quasi-simple groups $G$ with cyclic Sylow $p$-subgroup, all cases where the degree in question is less than $|g|$ have been determined in \cite{Z3}.The methods used in cite{Z3} do not work in the case with $p\neq \ell\neq 0$, and the problem  remained open.  In this paper  we assume that  $p\neq \ell$.

Although obtaining Theorem \ref{mm1} is the main target of this work, we additionally  identify  many cases of the \reps $\phi$ in question with $\phi(g)=|g|$ for all $p$-elements   $g\in G$. We show that if $ \deg\phi(g)=|g|-1$ then 1 is the only $|g|$-root of unity that is not an \ei of $\phi(g)$. The  absence of \ei 1 has geometric meaning as this is equivalent to saying that $\phi(g)$ acts fixed-point-freely on the %$FG$-
module afforded by $\phi$.  Therefore,   we try  to describe $g\in G$ such that $\deg\phi(g)=|g|$. 
Details are in the following theorem, where $\Phi_m$ is the $m$-th cyclotomic polynomial.

\begin{theorema}\label{mm2} Let  $G$ be a universal quasisimple group of Lie type in characteristic $r >0$,  and 
let $\phi$ be a non-trivial absolutely irreducible representation of $G$ in characteristic $\ell \neq r$. Let $g\in G$ be a $p$-element for $p\neq r,\ell$, and let $\deg\phi(g)$ denote the degree of 
the minimal polynomial of $\phi(g)$. Suppose that Sylow $p$-subgroups of $G$ are cyclic. Then one of the following statements holds.
\begin{enumerate}[\rm(I)]
\item $\deg\phi(g)=|g|$. 
\item $\deg\phi(g)=|g|-1$, $g$ is a generator of a Sylow $p$-subgroup of $G$, and all $|g|$-roots of unity but $1$ are eigenvalues of $\phi(g)$. Moreover, 
one of the \f statements holds.

\begin{enumerate}[\rm(a)]
\item $G= \SL_n(q),$ $n=d^a$ for some odd prime $d$, $p > n$,  
$|g|=(q^n-1)/(q^{n/d}-1)$, and $d \nmid (q^{n/d}-1)$.

\item $G=\SU_n(q)$ and $p > n$. Furthermore, there exist an odd prime $d$ and some $j \in \{0,1\}$ such that 
$n-j=d^a>1$ and $|g|=(q^{n-j}+1)/((q^{(n-j)/d}+1)\cdot\gcd(d,\ell,q^{n/d}+1))$. In addition,
either $d \nmid (q^{(n-j)/d}+1)$, or $d=\ell$ and $\ell \mid (q^{(n-j)d}+1)$.
  
\item
$G= \Spin^-_{2n}(q)$,  $\Spin^+_{2n+2}(q)$, or $\Spin_{2n+1}(q)$,  $n=d^a>3$ for an odd prime d, 
$|g|=(q^n+1)/((q^{n/d}+1)\cdot\gcd(d,\ell,q^{n/d}+1))$, and either 
$d \nmid (q^{n/d}+1)$, or $d=\ell \mid (q^{n/d}+1)$.
 
\item
$G= \Sp_{2n}(q)$,  $q$  odd, n is a $2$-power. Furthermore, either $|g|=(q^n+1)/2$ and $\dim\phi=(q^n-1)/2$;  
or $\ell=2$ and $(q,n,|g|)=(3,4,41)$ or $n>2, $ $q-1=2^a \geq 4$. 
 
\item
$G=E_6(q)$ and $|g|=q^6+q^3+1$ and $3\nmid  (q-1).$ 

\item
$G\in \{^2E_6(q), E_7(q)\}$, $|g|=q^6-q^3+1$ and $3\nmid  (q+1).$

\item
$G=E_8(q)$ and $|g|\in\{\Phi_{15}(q),\Phi_{20}(q), \Phi_{30}(q)\}.$ 
\end{enumerate}
\end{enumerate}
\end{theorema}

Throughout the paper, we use the convention that $\gcd(d,\ell,m)=1$ if $\ell=0$.

 Note that the question which cases listed in Theorem B(II) actually yield the equality $\deg \phi(g)=|g|-1$ remains open. 
Also, the problem of extending the results of our paper \cite{TZ8}  
to quasi-simple groups of exceptional Lie type is a subject of another paper in preparation.
Another  open case of the minimal polynomial problem
is that of alternating groups $G$.

For completeness, in Proposition \ref{m9t} below we handle the case of quasisimple groups of Lie type with exceptional Schur multiplier, still with cyclic
Sylow $p$-subgroups.

\begin{proposition}\label{m9t} Let $G$ be finite quasisimple group such that $G/Z(G)$ is a simple group of Lie type in   characteristic $r>0$, $r $ divides $ |Z(G)|$, and
$G$ has cyclic Sylow $p$-subgroups for some $p \neq r$. Let $\phi$ be a faithful, absolutely irreducible representation of $G$ in characteristic $\ell \neq p,r$. 
Suppose that $\deg\phi(g)<|g|$ for some $p$-element $g \in G$. Then one of the \f holds:

\begin{enumerate}[\rm(i)]
\item $G=2 \cdot\SL_2(4)$, $\ell\neq 2$,  $|g|=3,5$ and $\dim\phi=\deg\phi(g)=2,$ or $\ell\neq 2$, $|g|=5$ and  $\dim\phi=\deg\phi(g)=4;$

\item $G=2\cdot\SL_3(2)$, $\ell=7$, $|g|=3$ and $\dim\phi=\deg\phi(g)=2$, or 
$\ell\neq 2$,  $|g|=7$ and $\dim\phi=  \deg\phi(g)=4;$

\item $G=3\cdot\PSL_2(9)$, $\ell\neq 3$,  $|g|=4,5$, and $\dim\phi=\deg\phi(g)=3;$ 

\item $G=3\cdot\PSU_4(3)$, $\ell=2$,  $|g|=7$, and $\dim\phi=\deg\phi(g)=6;$
 
\item $G=6\cdot\PSU_4(3)$, $\ell\neq 2,3$,  $|g|=7$, and $\dim\phi=\deg\phi(g)=6;$

\item $G=2\cdot\Sp_6(2)$, $\ell\neq 2$,  $|g|=5$, $\dim\phi=8$, and $\deg\phi(g)=4;$
 
\item $G=2\cdot G_2(4)$, $\dim\phi=12$, $|g|=7,13$, and $\deg\phi(g)=|g|-1;$

\item $G=2 \cdot Sz(8)$, $\ell=5$, $\dim\phi=8$, $|g|=13$, and $\deg\phi(g)=8.$
\end{enumerate}
\end{proposition}

\bigskip

{\it Notation.} 
We write $\ZZ$ for the set of integers. 
For integers $a,b>0$ we write $(a,b)$ for the greatest common divisor of $a,b$. 

Let $G$ be a finite group. Then $|G|$ is the order of $G$, $Z(G)$ is the center of $G$, and $O_p(G)$ the maximal normal $p$-subgroup of $G$ for a prime $p$. 
We often use $|G|_p$ to denote the $p$-part of $|G|$. For $g\in G$ the order of $g$ is denoted by $|g|$; furthermore, $G'$ denotes the derived group of a group $G$. A $p'$-element is one of order coprime to $p$.

Let $F$ be an \acf of characteristic $\ell\geq 0$. All \reps are over $F$. An \ir $\ell$-modular \rep $\phi$ of a group $G$ is called {\it liftable} if $\ell\neq 0$ and there exists a \rep of $G$ over the complex numbers whose character coincides with the Brauer character of $\phi$ on the $\ell$-regular elements.  If $\phi$ is an $F$-\rep of $G$ or the Brauer character of it, we write
$\phi\in\Irr_\ell (G)$ to state that $\phi$ is irreducible. If $\ell=0$, we drop the subscript $\ell$. We denote by $1_G$ the trivial \ir character or \rep of $G$ (both ordinary and $\ell$-modular), and by $\rho_G^{reg}$ the regular \rep of $G$ or the (Brauer) character of it.

If $H$ is a subgroup of $G$ and $\eta$ is a character or \rep of $H$ then $\eta^G$ denotes the induced character. If $\phi$ is a character or a \rep of $G$ then $\phi|_H$ stands for the restriction of $\phi$ to $H$. We also write $(\chi,\phi)$ for the inner product of complex characters $\chi,\phi$ of a group $G$.

If $\mathbf{G}$ is an algebraic group then $\mathbf{G}^\circ$ denotes the connected component of $\GG$.

Let $\mathbf{G}$ be a connected reductive algebraic group in characteristic $r > 0$. Then  
$\mathbf{G}^\F:=\{g\in \mathbf{G}: \F(g)=g\}$, where $\F:\mathbf{G} \to \mathbf{G}$ is a Steinberg endomorphism, 
is called a {\it finite group of Lie type}, and the characteristic $r$ is called the {\it defining characteristic of} $\mathbf{G}^\F$.  Thus, in our paper a finite group of Lie type always comes together with a pair $(\mathbf{G},\F)$. In particular, a {\it regular} element of $\mathbf{G}^\F$
is one that is regular in $\mathbf{G}$. A {\it maximal torus} of $\mathbf{G}^\F$ is defined as
$\mathbf{T}^\F$ for an $\F$-stable  maximal torus $\mathbf{T}$ of $\mathbf{G}$, 
we say maximal tori $\mathbf{T}^\F$, $\mathbf{S}^\F$ of $\mathbf{G}^\F$ are {\it conjugate} if 
$\mathbf{T} $ and $\mathbf{S}$ are $\mathbf{G^\F}$-conjugate. A {\it simple group of Lie type} means a non-abelian composition factor of a quasi-simple group of Lie type.

Throughout the paper we use standard terminology used  
for  finite groups of Lie type and their representations and characters;
for this we usually follow \cite{DM}. 

It is customary to write $\GL^\ep_n(q)$ and $\SL^\ep_n(q)$ for the uniform treatment of groups
$\GL _n(q)$, $\SL_n(q)$ and $\U_n(q)$, $\SU_n(q)$, respectively. Thus, $\GL^\ep _n(q)$ is $\GL_n(q)$ or $\U_n(q)$ when $\ep=+$ or $-$, respectively.   
The symbol $\ep \in \{+,-\}$ is frequently  
interpreted as $\ep 1$ in various formulas as well. For $G$ a finite classical group, with natural module $V = \FF_q^n$, $n \geq 2$, we frequently write {\it an irreducible element 
$g \in G$} to indicate that $g$ acts irreducibly on $V$ (and hence $g$ is semisimple).

For the original notion of a {\it Weil representation} over the complex numbers  we refer to \cite{Ger} and \cite{Howe}, 
where this applies to the groups $\Sp_{2n}(q)$ with $q$ odd, and $\GL_n(q)$, $\U_n(q)$ for any $q$.  
In this paper, we also use this term for $\SL_n(q)$, respectively $\SU_n(q)$, to denote any \ir constituents of the restriction of Weil representations of 
$\GL_n(q)$, respectively $\U_n(q)$. 
Furthermore, by {\it Weil representations} in any prime characteristic $\ell$ we mean the non-trivial \ir constituents of reductions modulo $\ell$ of complex Weil representations. One can consult \cite[Section 5.1]{DM18} and \cite[Sections 5 and 11]{GMST} for details.       

\section{Preliminaries}

\begin{lemma}\label{zgm} {\rm \cite[Lemma IX.2.7]{HB}}
Let $p,r$ be primes and $a,b$ positive integers such that $p^{a}=r^{b}+1$. Then either
\begin{enumerate}[\rm(i)]
\item $p = 2$, $b = 1$, and $r$ is a Mersenne prime, or

\item $r = 2, a = 1$, and $p$ is a Fermat prime, or

\item $p^{a} = 9$.
\hfill $\Box$
\end{enumerate}
\end{lemma}

\bl{pp7} Suppose that n is odd and $(q^n+1)/(q+1)$ is a prime power. Then
$n$ is a prime.  \el

\bp Let $p^l=(q^n+1)/(q+1)$.   
If $n=km$ with $k,m>1$ then  
$$(q^n+1)/(q^m+1)=\frac{q^n+1}{q+1}:\frac{q^m+1}{q+1}\in \ZZ,$$  
and hence
$(q^m+1)/(q+1)>1$ is a $p$-power. Then 
$(n,q) \neq (3,2)$; so, by Zsigmondy's theorem \cite{Zs}, there is a prime $t$ such that $t|(q^{2n}-1)$ and $(t,q^i-1)=1$ for $1\leq i<2n$, in particular, $(t,q^{2m}-1)=1=(t,q^m+1)$ and $t|\frac{q^n+1}{q^m+1}$. So $t\neq p$, a contradiction.
\enp

\bl{22d} Let $A$ be a finite abelian group with a   proper subgroup B.
Let $\chi$ be a complex character of A constant on $A\setminus\{1\}$ with $\ker(\chi) \neq A$. Then $\chi=m\cdot \rho_A^{reg}+c\cdot 1_A$ for some integer $m>0$ and $c\in \ZZ$, and  $\chi|_B=\rho_B^{reg}+\chi_1$, where $\chi_1$ is a true character of B. \el

\bp  The trivial and the regular characters of $A$ lie in the subspace of the 2-dimensional space of class functions constant on $A\setminus 1$, so $\chi=m\cdot \rho_A^{reg}+c\cdot 1_A$, where $m>0$. Here $m,m+c\in\ZZ$ as $\chi$ is a character,
and $m>0$ as $\ker(\chi) \neq A$. Since $\rho_A^{reg}|_B=|A:B|\cdot \rho_B^{reg}$,
we have $\chi|_B=|A:B|\cdot m\cdot \rho_B^{reg}+c\cdot 1_B$. As $|A:B|\cdot m > m \geq -c$, the result follows.\enp

\bl{um8} %{\rm \cite[Theorem 1.1]{Z5}}
 Let $G= \SL_n^\ep(q)$, $n>2$. Let p be a prime such that Sylow p-subgroups are cyclic, and let $\tau$ be a complex irreducible character of $G$. Suppose
$\deg\tau(g)<|g|$ for a p-element $g\in G$. Then $n$ is an odd prime, $|g|=(q^n-\ep 1)/(q- \ep 1)$, $\dim \tau=(q^n-   q)/(q-\ep 1)=|g|-1$,  and $\tau$ is a unipotent character of G. \el

\bp Let $P$ be a \syl of $G$ such that $g\in P$.  The assumption $\deg\tau(g) < |g|$ implies by \cite[Theorem 1.1 and Corollary 1.3(2)]{Z3} that 
$\tau(1)=(q^d-   q)/(q-\ep 1)$ (and $2 \nmid d$ if $\eps=-$). There is a single \ir character of this degree,  and  $\tau$ is unipotent (see for instance \cite[\S 3]{TZ96}), and moreover
$\tau|_P=\rho^{reg}_P-1_P$.  In particular, the statement follows if $P=\langle g \rangle$. If $P \neq   R:=\lan g\ran$, then  $\tau|_{R}
=\rho^{reg}_R+\chi$ for some proper character of $R$, and so $\deg \tau(g)=|g|$. \enp

\begin{theo}\label{hhs}   {\rm (Hall–Higman–Shult, cf.  \cite[Ch. IX, Theorem 3.2]{HB})} Let B be a finite group containing
a normal extraspecial p-subgroup $E_n$, $W=E_n/Z(E_n)$, and let $b \in B$ be such that $[b,Z(E_n)] = 1$. Suppose that $|b| = r^a$, where $r\neq p$ is a prime, and  $b^{r^{a-1}}$ does not centralize $E_n$. Let $\theta$ be a faithful irreducible representation of B  in characteristic $\ell\neq p$. Then either
\begin{enumerate}[\rm(i)]
\item the degree $d$ of the minimal polynomial of $\theta(b)$ equals $|b|$, or
\item $|b| = p^{n-t} + 1$, $|C_W(b)| = p^{2t}$ and b acts irreducibly on $W/C_W(b)$, in which case $d =|b| - 1$.
\end{enumerate}
\end{theo}

\section{Cyclic Sylow $p$-subgroups}

Throughout this section,  $ \GG$ is a connected reductive algebraic group  in characteristic $r>0$, $\F$ a Steinberg endomorphism of $\GG$ and $  G=\GG^\F$. 
%In addition, $\GG^u$
%and $G^u$ is the subgroups of the respective group generated by unipotent elements. 
%It is well known that $\GG^u=\GG'$ is semisimple, $Z(\GG^u)$ is finite,   $\GG=\GG' Z(\GG)^\circ=\GG^u Z(\GG)^\circ$  and $Z(\GG)^\circ$ is contained in every maximal torus $\mathbf{T}$ of $\GG$ \cite[p. 88]{C}.   
The  purpose of this section is to prove for certain finite groups $G=\GG^F$ of Lie type in characteristic $r > 0$ that if $\phi\in\Irr_\ell(G)$, $\dim\phi>1$, $\ell \neq r$,
and $Y$ is a cyclic \syl of $G$ with $p \neq \ell$, such that all elements of $Y\setminus \{1\}$ are regular semisimple, then  $\deg\phi(g)\geq |g|-1$. 
Recall a semisimple element of $G=\mathbf{G}^\F$ is {\it regular} if it is regular in $\mathbf{G}$, equivalently, that
 its centralizer in $\GG$ contains no non-identity unipotent element. By a maximal torus in $\GG^F$ we mean the subgroup
 $\TC^{\F}$ for an $\F$-stable maximal torus $\TC$ of $\GG$. Furthermore, in this and the next sections, where we discuss 
 $G = \GL^\eps_n(q)$, then we also choose $\GG = \GL_n$ with a Steinberg endomorphism $\F:\GG\to\GG$ such that $G = \GG^{\F}$. 

We begin with some general facts.

\bl{ct1}  Let $G=\GL^\eps_n(q)$ and $H=[G,G] = \SL^\eps_n(q)$.  For $\phi\in\Irr_\ell(G)$, let $\phi'$ be an \ir constituent of $\phi|_{H}$ and let $g\in H$ be a semisimple element.  Then $\deg\phi(g)=\deg\phi'(g)$. 
\el

\bp   
We may assume $g\in T \cap H$ for some maximal torus $T$ of $G$. By Clifford's theorem,  $\phi|_{H}$  is a direct sum of  $(\phi')^h$, where $\phi'$ is an \ir constituent of $\phi|_{H}$ and $h$ ranges over representatives 
of the cosets in $G/H$. As $G=HT$, the  representatives can be chosen in $T$. Therefore, $(\phi')^h|_{T\cap H}=\phi'|_{T\cap H}$, and hence $\phi|_{T\cap H}$ is a multiple of $\phi'|_{T\cap H}$. Whence the result. \enp

\bl{tr1} Let G be a finite group of Lie type, $T$  a maximal torus of $G$, and let $ Y\subset  T$ be a subgroup, such that every element  of prime order in Y is regular. Then $Y$ is cyclic. \el

\bp It suffices to prove the lemma for $Y$ a (nontrivial abelian) $p$-group for some prime $p$. Suppose the contrary; then there are  elements $x,y\in Y$ of order $p$ such that   the subgroup $\lan x,y\ran$ is not cyclic.  We can view  $x,y$ as elements of a maximal torus ${\mathbf T}$ containing $T$. 
Let $\al$ be a ${\mathbf T}$-root of ${\mathbf G}$. As $x,y$ are regular, we have
$\alpha(x) \neq 1$ and $\alpha(y) \neq 1.$ As $|x|=|y|=p$, there is some $1 \leq i < p$ such that $\alpha(x^i\cdot y) = 1$, and so the element $x^i\cdot y\in Y$ is not regular. \enp

\bl{t45} Let $\mathbf{G}$ be a connected reductive   
algebraic group in characteristic r and $Y$ a finite p-subgroup for some prime $p\neq r$. Suppose that $\mathbf{G}'$ is simply connected and  every element of order p in Y is regular. Then Y is cyclic.  \el

\bp We may assume $Y$, and hence $Z(Y)$, is nontrivial. Let $y\in Z(Y)$ be of order $p$ and let $\mathbf{T}$ be a maximal torus of $\mathbf{G}$
containing $y$. As $y$ is regular semisimple and $\mathbf{G}'$ is simply connected, $C_{\mathbf{G}}(y)=\mathbf{T}$, see \cite[Theorem 3.5.6]{C}. 
So $Y \subset \mathbf{T}$, and hence $Y$ is abelian. The lemma
now follows from Lemma \ref{tr1}.\enp

\bl{bb1} Let $T$ be a maximal torus of a finite reductive group $G=\mathbf{G}^\F$, and let $t_1,t_2\in T$. Suppose that    $C_{\GG}(t_1)$ is connected.
If $t_1,t_2$ are conjugate in G then they are conjugate in $N_G(T)$.\el

\bp Let $\mathbf{T}$ be a maximal torus of $\mathbf{G}$ containing $T$.
By \cite[Ch. II, \S 3.1]{SS},  $nt_1n\up=t_2$ for some $n\in N_{\mathbf{G}}(\mathbf{T})$. Then $\Fr(n)t_1\Fr (n\up)=t_2$, whence $n\up \Fr (n)t_1n\Fr (n\up)=t_1$,
that is, $n\up \Fr (n)\in C_{\mathbf{G}}(t_1) $. By the Lang theorem, $n\up \Fr (n)=t\up \Fr (t)$ for some $t\in C_{\mathbf{G}}(t_1)$. So $tn\up=\Fr (t)\Fr (n\up)=\Fr (tn\up)$, so $tn\up\in G$ and $x:=nt\up\in G$. Clearly, $xt_1x\up=t_2$ and $x\in
N_{\mathbf{G}}(\mathbf{T})\cap G=N_{G}(\mathbf{T})$. As $T=\mathbf{T}\cap G$, we have $xTx\up=T$, as required.\enp

\begin{remar} 
{\em Lemma \ref{bb1} is a corrected version of \cite[Lemma 3.1]{TZ20}. The proof of   Lemma 3.1 in \cite{TZ20} requires $C_{\mathbf{G}}(h)$ to be connected.
This is missed in the statement of \cite[Lemma 3.1]{TZ20}. However, 
\cite[Lemma 3.1]{TZ20} is used only in the proof of \cite[Lemma 3.2]{TZ20}, where this property is shown to hold. So all other results of \cite{TZ20} are not affected by the omission.}\end{remar}

\med
Let $G^*=(\mathbf{G^*})^{\F^*}$ be the group dual to $G=\mathbf{G}^\F$, see \cite[\S 4.3]{C} and \cite[\S 13]{DM}. The  
 character theory of groups of Lie type establishes a  correspondence $\mathbf{T}\ra \mathbf{T^*}$ between $G$-classes of $\F$-stable maximal tori in $\mathbf{G}$ and $G^*$-classes of $\F^*$-stable maximal tori in $\mathbf{G^*}$, and $|\mathbf{T}^\F|=|(\mathbf{T^*})^{\F^*}|$ \cite[Propositions 4.3.4 and 4.4.1]{C}.
Furthermore, in the case $G=\GG^\F = \GL^\eps_n(q)$ (with $\GG=\GL_n$) we have $G^* \cong G$, and every semisimple element in $G$ has connected
centralizer in $\GG$ (see e.g. \cite[Theorem 3.5.6]{C}).

For  a semisimple $\ell'$-element   $s\in G^*$, denote by $\mathcal{E}_{\ell,s}$ the union of the rational series $ \mathcal{E}_{ys}= \mathcal{E}(G,ys)$, where  $y\in G^*$, $ys=sy$ and $|y|$ is an $\ell$-power. 
%By a theorem by Brou\'e and Michel \cite{BM}, $\mathcal{E}_{\ell,s}$ is a union of $\ell$-blocks,  so 
For every  $\phi\in \Irr_{\ell} (G)$  there exists a semisimple $\ell'$-element $s\in G^*$ such that $\phi$
is a constituent of $\chi\pmod \ell$ for some $\chi\in\mathcal{E}_{\ell,s}$. Therefore, it is meaningful to write $\phi\in\mathcal{E}_{\ell,s}$ in such a case. In fact, $\chi$ can be chosen in $\mathcal{E}_{s}$ \cite[Theorem 3.1]{H4}. In particular, 
if $\phi\in\mathcal{E}_{\ell,1}$ then $\phi$ is a   constituent of $\chi\pmod \ell$ for some unipotent character $\chi$, in which case $\phi$ 
is called {\it unipotent}.
 
\bl{dd2} Let $\phi\in\mathcal{E}_{\ell,s}$ be a Brauer character. Then $\be$ in an integral linear combination of ordinary characters of $ \mathcal{E}_{\ell,s}$ restricted to $\ell'$-elements.  In addition, $\phi(1)$ is a multiple of $|G^*:C_{G^*}(s)|_{r'}$.
\el

\bp By a theorem by Brou\'e and Michel \cite{BM}, $\mathcal{E}_{\ell,s}$ is a union of $\ell$-blocks. So $\phi$ lies in one of those blocks. Then $\phi$ is a $\ZZ$-linear combination of the ordinary \ir characters from the block it belongs to. For the additional statement see
\cite[Proposition 1]{HM}. \enp

\bl{gh6}  \cite[Theorem 5.1]{GeH} Let $G = \mathbf{G}^\F$ be a group of Lie type with $\mathbf{G}$ having connected center, and let
$\ell$ be a good prime for $\mathbf{G}$.  
Let $s\in G^*$ be a semisimple $\ell'$-element and $\be\in \mathcal{E}_{\ell,s}$ be a Brauer character. Then $\be$ is an integral linear combination of the ordinary  
characters from $\mathcal{E}_{s}$ restricted to $\ell'$-elements. In addition, the number of    \ir Brauer characters in $\mathcal{E}_{\ell,s}$ equals $|\mathcal{E}_{s}|$.\el

\bl{t45a} Let $\mathbf{G}$ be a  connected reductive  algebraic group in characteristic $r>0$, $\F$ a Steinberg  endomorphism of $\mathbf{G}$ and $G=\mathbf{G}^\F$. 
Suppose that $\mathbf{G}'$ is simply connected and every element of order $p\neq r$ in G is regular. Then every unipotent character of G  of non-zero p-defect lies  in the principal $p$-block.  \el

\bp  By Lemma \ref{t45}, a \syl of $G$ is cyclic, and hence is contained in a maximal torus $T=\mathbf{T}^\F$, say, of $G$ \cite[Corollary 3.16]{DM}.  Let $y\in T$ be of order $p$. By hypothesis, $y$ is regular and $\mathbf{G}'$ is simply connected, so $C_{\mathbf{G}}(y)=\mathbf{T}$, and  hence   $C_G(y)= T$. Now if $g \in C_G(y)$ is $p$-singular, then $g \in T$ and, if $h$ is a power of $g$ of order $p$, then $h$ is regular, whence $\mathbf{T} \subset  C_{\mathbf G}(g) \subset  C_{\mathbf G}(h) = \mathbf{T}$ and thus $g$ is regular. We have shown that every $p$-singular element of $G$ is conjugate to a regular element in $T$.

Let $\chi$ be a unipotent character of $G$ of non-zero $p$-defect.  By \cite[Remark 2.3.10]{GM}, $\chi$  is constant on the regular elements of $T$, and hence on the $p$-singular elements of $G$. Then the result follows by \cite[Theorem 1.3(1)]{PZ}. \enp
 
\def\stc{Suppose the contrary.} 
\def\sft{Suppose first that }

\bl{nm6} Let $\GG$ be a connective reductive algebraic group with F  endomorphism $\F$,  $G=\GG^\F$, ${\mathbf T}$ is an $\F$-stable maximal torus of $\GG$.  Suppose that  $\GG'$ is simply connected, $T=\TC^\F$ has an element of prime order $p$, and that every element of order $p$ in T is regular. Then T contains a \syl of G, which is cyclic, and every p-singular element of $G$ is regular.\el

\bp Let $Y$ be a \syl of $T$. By Lemma \ref{tr1}, $Y$ is cyclic. 
%Also, the regularity of every element of order $p$ in $T$ implies that $p \nmid |Z(G)|$. 
Let $Y_1$ be a \syl of $G$ containing $Y$, and assume that $Y_1\neq Y$. Then $N_{Y_1}(Y)\neq Y$, so let $h\in N_{Y_1}(Y) \smallsetminus Y$, whence $h\notin T$.  Let $y\in Y$ be of order $p$. As $Y$ is cyclic, the $p$-element $h$ normalizes the subgroup $\lan y \ran$ of order $p$,  and so we have $[h,y]=1$. 
Thus $h\in C_{\GG}(y)$.   As $y$ is regular and $\GG'$ is simply connected, 
$C_{\GG}(y)={\mathbf T}$. It follows that $h\in C_G(y)=C_{\GG}(y)^\F={\mathbf T}^F=T $, a contradiction. 

We have shown that $Y \in \Syl_p(G)$. Suppose now that $g \in G$ is $p$-singular. Then some power $h$ of $g$ has order $p$, and without loss we may assume
$h \in T$. By assumption $h$ is regular, so again we have $g \leq C_\GG(h)= \TC$. But in this case $\TC \leq C_\GG(g) \leq C_\GG(h)=\TC$, so 
$C_\GG(g)=\TC$ and $g$ is regular.
\enp

%\subsection{Groups with connected center}
%Now we focus on the connected reductive groups with connected center.
%The \f lemma refines \cite[Lemma 3.2]{TZ8}. 
Recall that for every $\F$-stable maximal torus $\mathbf{S}$ of $\GG$ there exists an $\F$-stable maximal torus 
$\mathbf{S^*}$ of $\GG^*$ in duality with $\mathbf{S}$, such that $|S|=|S^*|$, where $S^*:=(\mathbf{G}^*)^{\F^*}$, see \cite[Proposition 4.4.1 and Corollary 4.4.2]{C}.

\bl{nn1} Let $\GG$ be a connected reductive algebraic group  in characteristic $r>0$ with 
%connected center such that 
$\GG'$ being  
%simple and 
simply connected.  Let $\F$ be a Steinberg  endomorphism of $\GG$,  $G=\GG^\F$ and 
let $\chi\in\mathcal{E}_{s}$ for a semisimple element $s\in G^*$.
Let $\SB^*$ be an $\F^*$-stable maximal torus of $\GG^*$ and $S^*=({\mathbf S}^*)^{\F^*}$. Let $\mathbf{S}$ be in duality with $\mathbf{S^*}$, and let 
$Y \neq \{1\}$ be a  Hall $\pi$-subgroup of $S:=\mathbf{S}^\F$ for some set $\pi$ of primes 
such that every element of $Y\setminus \{1\}$ is regular.
Then the following statements hold.

\begin{enumerate}[\rm(i)]
\item \st  $s$ is not conjugate to an element of $S^*$. Then $\chi(y)=0$
for all $1\neq y\in Y$. 

\item \st  $s\in S^*$ and   $(|s|,|Y|)=1$. 
Then $\chi$ is constant on the elements of $Y \setminus \{1\}$.  

\item Suppose that $s=1$, so that $\chi$ is unipotent.  Then $\chi$ is constant on the regular elements $h$ of any fixed maximal torus of $G$.  

\item Assume in addition that $G \cong G^*$ and $Z(\GG)$ is connected. If $(|s|,|Y|)>1$ and $\ell \nmid |s|$,  
then all ordinary characters $\psi \in \EC_{\ell,s}$ have the same degree.
\end{enumerate}
\el
  
\bp (i) For a function $f$ on $G$ denote by $f^{\#}$ the restriction of $f$ to the set of semisimple elements of $G$. 
Then $\chi^{\#}$ is a $\QQ$-linear combination of $R^{\#}_{T_i,\theta_i}$,
where  $R_{T_i,\theta_i}$ are Deligne-Lusztig characters in $\mathcal{E}_{s}$,  $T_i$ is a maximal torus of $G$ and $\theta_i$ is a linear character of $T_i$ 
(see for instance \cite[Lemma 4.1]{TZ8}). In addition,   $|\theta_i|=|s|$, where $|\theta_i|$
means the order of $\theta_i$ in the group of linear characters of $T_i$,  see for instance \cite{H4}. The  values $R_{T_i,\theta_i}(h)$ 
at the semisimple elements $h\in G$ are given by the formula
$$R_{T_i,\theta_i}(h)=\ep(\mathbf{T}_i)\ep(\GG)\theta_{i}^G(h)/\StC(h),$$
where $\StC$ is the Steinberg character of $G$ and $\ep(\mathbf{T}_i),\ep(\GG)\in\{\pm 1\}$, see for instance \cite[7.5.4]{C}.
(In particular,  $R_{T_i,\theta_i}(h)=0$ if $h$ is not conjugate to  an element of $ T_i$.)

Now let $h$ be a $p$-singular element of $S$ for some $p \in \pi$. Then $h$ is regular. % 
It is well known that any regular semisimple element
of $G$ lies in a unique maximal torus, so either $T_i$ is conjugate to $ S$ or
$R_{T_i,\theta_i}(h)=0$.  Therefore,  we conclude that
either $\chi(h)=0$ for all  $p$-singular elements $h \in S$    
 or $ \chi(h)=\sum_i a_iR_{S,\theta_i}(h)$, with some nonzero coefficients  $a_i$.  \itf if $ \chi(h)\neq 0$ then $s$ is conjugate in $G^*$ to an element of $S^*$.
 % the maximal torus of $G^*$  dual to $S$.   This implies (i).

Furthermore, $\StC (h)=\ep(\GG)\ep(\SB)$ as $h$
is regular and hence 
%$C_\GG(h)^\circ=\SB$ \cite[Corollary 9.3]{DM}.
%As some power of $h$ is  a semisimple element of order $p$, 
%Also, by \cite[Theorem 3.5.6]{C}, $C_{\mathbf{G}}(h)$ is connected, so  
$C_\GG(h)=\SB$.
Then  $R_{S,\theta_i}(h)=\theta_{i}^G(h)$, and hence 
\begin{equation}\label{eq220}
  \chi(h)=\sum_i a_i\theta_{i}^G(h).
\end{equation}   By Lemma \ref{bb1}, if $h,h'\in S$ are 
 conjugate in $G$  then $h,h'$ are conjugate in $N:=N_G(S)$. 
  It follows that \begin{equation}\label{eq221}
  \theta_{i}^G(h)=\theta^{N}_{i}(h)\,\,\,\,{\rm and }\,\,\,\,\chi(h)=\sum_i a_i\theta_{i}^N(h).\end{equation}  
for every $p$-singular element $h\in S$.  

\smallskip
(ii) Let $1 \neq y \in Y$  and consider the characters $\theta_i$ of $S$   
 in \eqref{eq220}. As $(|s|,|Y|)=1$ and so $|\theta_i|=|s|$ is coprime to $|y|$, we have 
$\theta_{i}(y)=1$. In fact, if $n \in N_G(S)$ then $|nyn^{-1}| = |y|$ and hence $\theta_i(nyn^{-1})=1$ as well. Now using \eqref{eq221} we have
$$\theta_i^G(y) = \theta_i^N(y)=|N_G(S)/S|.$$ 
This implies by \eqref{eq220} that $\chi$  constant on $Y\setminus \{1\}$.

\smallskip
(iii)   See   \cite[Remark 2.3.10]{GM} where it is observed that $\chi(h)=(\chi,R_{T,1})$.  

\smallskip
(iv)  
%By Lemma \ref{ca1}, $|G/G'|$ divides $r^2|Z(G)|$. As the elements of $Y\setminus \{1\}$ are regular semisimple, we have  $(|Y|,r^2|Z(G)|)=1$ and hence $(|Y|,|G/G'|)=1$.  
By Lemma \ref{nm6}, the elements $g\in G$ with $(|g|,|Y|)>1$ are regular, and hence $r \nmid |C_G(g)|$ for all such $g$.
%Then  it follows from  Lemma \ref{nm6}     
Since $G \cong G^*$ and $(|s|,|Y|) > 1$, it follows that $r \nmid |C_{G^*}(s)|$, and so $s$ is regular. The same holds for $xs$, for every $\ell$-element $x\in G^*$ with $xs=sx$.   
By definition,  $ \EC_{\ell, s}$ is the union of $\EC_{xs}$ for $\ell$-elements $x \in C:=C_{G^*}(s)$,  so $\psi$ lies in  $\EC_{xs}$ for some $\ell$-element $x$ with $xs=sx$. Since $Z(\GG)$ is connected, $C_{\GG^*}(s)$ is connected   \cite[Theorem 4.5.9]{C}, so $C_{\GG^*}(s)$ is a maximal torus; similarly, 
$xs$ has connected centralizer in $\GG^*$ which is the same as the centralizer of $s$. Hence  $\mathcal{E}_{xs}$ consists of a single character $\chi_{xs}$ \cite[Theorem 13.23]{DM}, which has degree $[G^*:C]_{r'}$, see \cite[13.24]{DM}.
% where $r$ is the defining characteristic of $G$,    
\enp

\bl{kk8} Let G be a finite group, $\ell$ a prime divisor of $|G|$ and let $\be$ be an \ir Brauer character of G in a block B of non-zero defect. Then 
\begin{enumerate}[\rm(i)]
\item $\be$ is an
integral  linear combination of ordinary \ir characters  of B  (restricted to the set of $\ell'$-elements of G);  
\item if $\be= \sum a_j\chi^\circ_j$, where $\chi_j$ are distinct \ir characters of $G$, $0\neq a_j\in \CC$  and $\chi^\circ_j$ stands for the restriction of $\chi_j$ to the set of $\ell'$-elements of G, then none of $\chi_j$'s is of defect $0$.  
\end{enumerate}
\el

\bp The first statement is well known, see \cite[Lemma 3.16]{N}. The additional statement is known too, but we provide a proof for the reader's convenience.  

Set $\be_1= \sum a_j\chi_j$. Suppose the contrary, let $\chi_1$, say, is of defect 0. Then  $\chi_1$ is a unique character in the block it belong to.  Therefore, $(\chi_j,\chi_1)=0$ for $j>1$, and    $(\be_1,\chi_1)=0$ by (i). On the other hand,  $(\be_1,\chi_1)=a_1$, a contradiction.    \enp

\bl{rm1} Let $\GG$ be a connected  reductive algebraic group  in characteristic $r$ 
%with connected center 
such that $\GG'$   is simply connected. Let $G=\GG^\F$,  and let $Y \neq \{1\}$ be an abelian Hall $\pi$-subgroup of $G$, for a set $\pi$ of primes different from $r$ 
such that every element of $Y\setminus \{1\}$ is regular. 
%$(|Y|, |G/G'|)=1$.  
Let   $\phi\in\mathcal{E}_{\ell,  s}$ be an irreducible $\ell$-Brauer character, where $(|s|,\ell)=1$ and $(\ell,r|Y|)=1$. Then the \f holds. 
\begin{enumerate}[\rm(i)]
\item \st $(|C_{G^*}(s)|,|Y|)=1$.   Then $\phi(y)=0$ for every $1\neq y\in Y$. 
\item \st   $(|s|,|Y|)=1$. Then $\phi$ is constant on $Y \setminus \{1\}$.  
\item Suppose that $G \cong G^*$, $Z(\GG)$ is connected, and $(|s|,|Y|)\neq 1$. Then $\phi$ lifts to characteristic $0$.
\end{enumerate}
\el

\bp   Note that $Y$ is contained in a maximal torus $S = \SC^\F$ of $G$. Indeed, $C_G(y)=S$  for $1\neq y\in Y$  as $\GG'$ is simply connected. 
Let $S^*$ be the maximal torus of $G^*$ dual to $S$. 

\smallskip
Recall that $\mathcal{E}_{\ell,s}$ is a union of $\ell$-blocks, so, by Lemma \ref{kk8}, $\phi$ is a $\ZZ$-linear combination (restricted  to 
$\ell'$-elements) of \ir characters $\chi\in \mathcal{E}_{\ell,s}$. Each such $\chi$ belongs to $\EC_{xs}$ for some $\ell$-element $x \in C_{G^*}(s)$.
If $\chi$ is not completely vanishing on $Y \setminus \{1\}$, then $xs$ is conjugate to an element in $S^*$ by Lemma \ref{nn1}(i), whence we may assume that 
the $\ell'$-part $s$ of $xs$ belongs to $S^*$, and so $C_{G^*}(s) \geq S^*$ has order divisible by $|S^*|=|S|$, which is divisible by $|Y|$. Thus (i) follows.
Next, if $(|s|,|Y|)=1$, then $|xs|$ is coprime to $|Y|$, whence $\chi$ is constant on $Y\setminus \{1\}$
by Lemma  \ref{nn1}(i), (ii). This implies (ii).

\smallskip
(iii)  As mentioned in the proof of Lemma \ref{nn1}(iv),   $xs$ is regular in $G^*$ for every $\ell$-element $x\in C_{G^*}(s)$.  By Lemma  \ref{nn1}(iv), all ordinary characters in  $\mathcal{E}_{\ell,s}$ are of the same degree $d:=|G^*:C_{G^*}(s)|_{r'}$. Using Lemma \ref{kk8} we conclude that $d|\phi(1)$. In addition, $\phi$ is a constituent of $\chi\pmod\ell$ for some $\chi\in \mathcal{E}_{\ell,s}$, see \cite[Proposition 1]{HM}. Therefore, $\phi(1)=d$, and the claim follows.\enp

The \f lemma will be used only for $G=E_8(q)$.
% with $|S|\in\{\Phi_{15}(q), \Phi_{30}(q)\}$.

\bl{20tz} Under assumptions of Lemma {\rm \ref{nn1}(iv)}, suppose that every element  $1\neq y\in S $ is regular.  If $\ell$ divides $|S|$, assume in addition that $G$ is simple. Let $\phi\in\Irr_\ell(G)$ be nontrivial with Brauer character $\be$,
and suppose that  $\deg\phi(g)<|g|$ for some $g\in S$ with $(|g|,\ell)=1$. Then either $\be$ is liftable to characteristic $0$, or $|g|=|S|$, $\be$ is in a unipotent block,  and $\deg\phi(g)= |g|-1$. \el

\bp (a) Since $\GG'$ is simply connected, $C_\GG(y) = \SC$ and hence $C_G(y)=S$ whenever $1\neq y\in S$.  Let $s\in G^*$ be a semisimple $\ell'$-element such that $\phi\in{\mathcal E}_{\ell,s}$, and recall $\SC^*$ is dual to $\SC$.

By assumption, $g \in S$ is a nontrivial $\ell'$-element.
Hence the Hall $\ell'$-subgroup $Y$ of $S$ contains $g$ and so is nontrivial. Taking $\pi$ to be the set of prime divisors of $|Y|$ (so that
$Y$ is the Hall $\pi$-subgroup of $S$), we observe that $Y$ is a Hall $\pi$-subgroup of $G$.
Indeed, otherwise there exist a prime $p \in \pi$ such that the Sylow $p$-subgroup $Y_p$ of $Y$ is not a \syl of $G$.
Then $N_G(Y_p)\setminus Y_p$ contains a  $p$-element $h$, say. As $h$ acts on $Y_p \neq \{1\}$, we have  $C_{Y_p}(h)\neq \{1\}$; let $1\neq y\in Y_p$ be such that  $h\in C_{G}(y)$.  As $y$ is regular, we have $C_{G}(y)=S$ by the above.  
Hence $h\in S$ and so $h \in Y$, a contradiction.  

Next suppose that $(|C_{G^*}(s)|,|Y|)=1$. Then $\be$ vanishes on $Y \setminus \{1\}$ by Lemma \ref{rm1}(i), and hence $\deg \phi(g)= |g|$, contrary to the 
assumption. We have shown that
%\begin{equation}\label{gcd10}
$$(|C_{G^*}(s)|,|Y|) >1.$$
%\end{equation}  
Thus, there is $h\in C_{G^*}(s)$ of prime order $p$ for some 
$p \in \pi$.  Recall that $|G|=|G^*|$, $|S|=|S^*|$, and $Y$ is a Hall $\pi$-subgroup of $G$. It follows that a \syl of $S^*$ is a \syl of $G^*$. Therefore, $xhx\up\in S^*$ for some
$x\in G^*. $  Since we assume $G \cong G^*$,  $xhx^{-1}$ is regular (see the proof of Lemma \ref{nn1}(iv)).    But $[s,h]=1$, so $xsx\up\in C_{\GG^*}(xhx\up) =\mathbf{S^*}$,
and hence $xsx\up\in S^*$. In what follows we may therefore assume that $s \in S^*$. 

Since $s$ is an $\ell'$-element and $Y$ is a Hall $\ell'$-subgroup of $G$ with $|G^*|=|G|$, we see that either
$(|s|,|Y|) > 1$, or $s=1$. In the former case, $\be$ lifts to characteristic $0$ by Lemma \ref{rm1}(iii). 
So we  may assume that $s=1$. Then 
$\be$ is constant on $Y\setminus \{1\}$ by Lemma \ref{rm1}(ii). Therefore,  $\deg\phi(g)=|g|-1$ and $|g|=|Y|$ by Lemma \ref{22d}.
As $s=1$, $\be$ is in a unipotent block, and we are done if $\ell \nmid |S|$.

\smallskip
(b) Suppose that  $\ell$ divides $|S|$, and $s=1$. We now show that this contradicts the assumption $\deg\phi(g)<|g|$. 

By Lemma \ref{nm6} applied to the prime $p=\ell$, 
Sylow $\ell$-subgroups of $G$ are cyclic and all $\ell$-singular elements of $G$ are regular. So $\ell$ is a good prime for $G$. By \edit{Lemma \ref{gh6}},  $\be$ is an integral linear combination of ordinary unipotent characters $\chi$ restricted to the $\ell'$-elements; write
$$\be=\sum a_j\chi^\circ_j$$ 
for some distinct unipotent characters $\chi_j$ and integers $a_j\neq 0$. By Lemma \ref{kk8}, none of them is of $\ell$-defect 0.  Therefore, each $\chi_j$ belongs to the principal $\ell$-block of $G$ by Lemma \ref{t45a}, and $\phi$ belongs to the principal block.  

Note that   $N:=N_G(O_\ell(S))=N_G(S)$. 
(Indeed, $N_G(S)\leq N_G(O_\ell(S))$. Conversely, any $x\in N_G(O_\ell(S))$ normalizes  $C_{G}(O_\ell(S))=S$,  as elements of $S\setminus \{1\}$ are regular.) 

Let $F$ be an \acf of characteristic $\ell$ and $V$   an $FG$-module with character $\be$. 
Then $V$ is in the principal $\ell$-block (as so is $\be$). 
Furthermore,  $V|_N=V_1\oplus V_2$ for some  $FN$-module $V_2$, where $V_1$ is the Green correspondent of $V$; in addition,  $V_1$ is indecomposable and lies in the principal $\ell$-block of $N$, see  \cite[Proposition 64.42(ii)]{CR}.
By \cite[Theorem 6.10]{N}, $O_{\ell'}(N)$ lies in the kernel of $V_1$. As $Y\subseteq O_{\ell'}(N)$, we conclude that  $V_1|_{Y }$ is trivial.
So $1_Y$ is a constituent of  $V|_{Y}$. 

By Lemma \ref{rm1}(ii), $\be$ is constant on $Y\setminus \{1\}$.  
As noted above, $C_G(y)=S$ for all $1 \neq y \in S$.
Observe that $N\neq S$ as $G$ is simple (see   for instance     \cite[Theorem 39.1]{Asch}). 
 \itf $N$ fixes no non-identity element of $Y$, 
  and hence no non-trivial \ir character of $Y$. Note that the action of $N$ on 
  characters $\tau$ of $Y$ is defined by sending $y$ to $\tau(nyn\up)$.  We denote this character by $\tau^n$. As $\be^n=\be$, we have $(\be,\tau)=(\be^n,\tau^n)=(\be,\tau^n)$, so an \ir character $\tau$ of $Y$ occurs in $\be|_Y$ with the same \mult as that of $\tau^n$. As $Y=\lan g\ran$, two \ir characters $\tau,\tau'$ coincide \ii $\tau(g)=\tau'(g)$.  
 So $\deg\phi(g)= |g|-1$ implies that $\be|_Y$ has exactly $|g|-1$ distinct \ir constituents. 
Let $\tau_0$ be the \ir character of $Y$ that does not occur as an \ir constituent of $\be|_Y$. If $\tau_0\neq 1_Y$ then $\tau_0^n\neq \tau$ for some $n\in N$ and $\tau_0^n$
is not a constituent of $\be|_Y$  too, which  is a contradiction. So $\tau_0=1_Y$,
so  $\deg\phi(g)= |g|-1$ implies that 1 is not an \ei of $g$ on $V$. But this is impossible since $V_1$ is trivial on $Y$.
\enp

If $\be$ is liftable then the cases where $\deg\phi(g)<|g|$ are listed in \cite{Z3};
in particular we always have $\deg\phi(g)\geq |g|-1$.
 
\section{Linear and unitary groups} 

In this section we deal with groups $\GL_n^\ep(q)$ and $\SL_n^\ep(q)$, where $\ep=+$
or $-$, and $q$ is a power of a prime $r$. We start with the case where $n=d>2$ is a prime and then aim to reduce the general case to this one (when possible).
For $X$ a finite abelian group and $p$ a prime, we will also write $X_p$ instead of $O_p(X)$, and $X_{p'}$ instead of $O_{p'}(X)$, for brevity.

\subsection{The case of $n=d$ an  odd prime}
  
In this section $d>2$ is a  prime, $G= \GL_d^\ep(q)$ and $S$ is a maximal torus of $G$
of order  $q^d-\ep 1 $. Note that $S$ is \ir on the natural module $V$ for $G$. 

\bl{us1} Let $G= \GL_d^\ep(q)$ and $g\in S\setminus   Z(G)$.  
Then  $g$ is \ir on the natural module  of $G$ and hence regular in $\GL_d(\overline{\FF}_q)$.\el

\bp 
Let $K=\FF_q$  if  $\ep=+$ and  $\FF_{q^2}$ if  $\ep=-$. Let $E$ be the enveloping algebra of $S$ over $K$.
Then $E$ is a field by Schur's lemma, $|E:K|=d$ and $g\in E$.  Then   $\lan K,g\ran$
  is a subfield of $E$, and $\lan K,g\ran\neq K$ as $g$ is not scalar. As $d$ is a prime and 
$g\notin Z(G)$, 
we have $\lan K,g\ran=E$, and hence $g$ is \ir on $V$.
\itf  $g$ has $d$ pairwise distinct eigenvalues over $\overline{\FF}_q$, and hence 
$g$ is regular. \enp

\bl{co8} Let $G=\GL_d^\ep(q)$.  If $\chi$ is a unipotent ordinary character of $G$ then $Z(G)$ is in the kernel of $\chi$ and  $\chi$ is constant on $S\setminus Z(G)$,  where $S$ is a maximal torus of order $q^d - \ep $.\el 

\bp This follows from Lemma \ref{us1} and Lemma \ref{nn1}(iii). 
\enp
 
\bl{ue1} Let $G_1= \SL_d^\ep(q)$, and let $S_1$ be a maximal torus  of order $(q^d-\ep)/(q-\ep)$. Then every d-element from $S_1$  is in $Z(G)$. Consequently,  $S_1=Z(G_1) \times Y$, where $Y$ is a Hall $d'$-subgroup of $S_1$, and also of $G_1$. In addition, $Y$ is a Hall subgroup of $G:=\GL^\ep_d(q)$ and every non-identity element of Y is regular.\el

\bp Let $g\in S_1$ be a $d$-element and $g\notin Z(G_1)$ so $g$ is not scalar. Then $d$ divides $q^d-\ep 1$. By Lemma \ref{us1},
$g$ is \ir on the natural module $V$ for $G_1$ and $G$. 
If $d \mid (q-\ep 1)$ then, by \cite[Lemma 3.2]{Z3}, $G_1$ contains no \ir $d$-element.   As $|Z(G_1)|\in \{d, 1\}$, we have $S_1=Z(G_1)\times Y$, where $Y\cong S_1/Z(G_1)$ is a $d'$-group.  

We claim that $(|Z(G)|,|Y|)=(q-\ep 1,|Y|)=1$. Indeed, we have
  \begin{equation}\label{quot1}
 |S_1|=(q^d-\ep 1)/(q-\ep 1)= q^{d-1}+(\ep q)^{d-1}+\cdots +\ep q+1\equiv d\pmod{q-\ep 1}.
\end{equation}  
So  $(q-\ep 1,|Y|)\in\{1,d\}$,
% then there is a prime   $r$, say, that divides both $|Y|$ and $q-\ep 1$, hence $r=d$, 
a contradiction. 

By Lemma \ref{nm6},  for every prime divisor $p$ of $|Y|$,   Sylow $p$-subgroups of $Y$ are   Sylow $p$-subgroups of $G_1$.    This implies that $Y$ is a Hall $d'$-subgroup of $G_1$. Moreover,   $Y$ is Hall in $G$  as $|G/(G_1\cdot Z(G))|=d$. By Lemma \ref{us1},  non-identity elements of $Y$ are regular.\enp

\bl{nr5}  Let $G= \GL_d^\ep(q)$ and $G_1=\SL^\ep_d(q)$.  Let S be a maximal torus of G of order $q^d-\ep 1$ and $S_1=S\cap \SL^\ep_d(q)$. Let  $Y$ be the  Hall  
$d'$-subgroup of $S_1$, and write $Y=Y_\ell\times Y_{\ell'}$, where $Y_\ell$ is the Sylow $\ell$-subgroup of $Y$. Let $\be\in \mathcal{E}_{\ell,s}$ be a Brauer character of G, where $(|s|,\ell q)=1$ and $s \in G^* \setminus Z(G^*)$.  Then either $\be$ is liftable or $\be|_{Y_{\ell'}}=a\cdot \rho_{Y_{\ell'}}^{reg}$ for some integer $a\geq 1$.\el

\bp By Lemma  \ref{ue1}, $Y$ is a Hall subgroup of $G$. 

Let $S^*$ be a maximal torus of $G^*$ dual to $S$.
 If $(|C_{G^*}(s)|,|Y|)=1$ then %by Lemma \ref{}3.15(i)
$s$ is not conjugate to an element of $S^*$ (as otherwise $S^*$ divides $|C_{G^*}(s)|$ and $|S^*|=|S|=\gcd(d,q-\ep)(q-\ep)|Y|$). By Lemma \ref{nn1}(i), in this case every ordinary character $\chi\in \mathcal{E}_{\ell,s}$ vanishes on $Y\setminus \{1\}$. As $\beta|_{Y_{\ell'}}$ is a linear combination of characters $\chi\in \mathcal{E}_{\ell,s}$ restricted to $Y_{\ell'}$, we have  
$\be(y)=0$ for every $y\in Y_{\ell'}\setminus \{1\}$. Then $\beta|_{Y_{\ell'}}$ is a multiple of $\rho_{Y_{\ell'}}^{reg}$. 
   
Suppose that $(|C_{G^*}(s)|,|Y|)\neq 1$. Then $s$ is conjugate to an element of $S^*$.  
Indeed,  $S^*$ is a maximal torus of $G^*\cong \GL^\ep_n(q)$ of order $q^n-\ep 1 $. Then $S^*$ contains a subgroup $Y^*\cong Y$, and $(|G^*:Y^*|,|Y^*|)=1$, so $Y^*$ is an abelian Hall  subgroup of $G^*$.
 Therefore, if $t\in C_{G^*}(s)$ and $|t|$ divides $|Y|$ then we can assume that $t\in S^*$.  By Lemma \ref{us1}, $t$ is regular (as $|t| \nmid(q-\ep))$,
and hence  $s\in C_{G^*}(t)=S^*$ as stated. 

Thus we can assume that $s\in S^*$.
Recall that $s\notin Z(G^*)$ by assumption. It follows that $(|s|,|Y|) \neq 1$, and so $\beta$ is liftable by Lemma \ref{rm1}(iii).\enp

\bl{u23} Let $G= \GL^\ep_d(q)$, $G_1= \SL^\ep_d(q)$,   let $S$ be a maximal torus of G order $q^d-\ep 1$, $S_1=S\cap G_1 = Z(G_1) \times Y$.  Let $\be\in\Irr_\ell(G)$ be a non-liftable unipotent character. Suppose that   $\ell$ divides $|Y|$. Then $1_{Y_{\ell'}}$ is a constituent of  $\beta|_{Y_{\ell'}}$.\el

\bp Let $B$ be the $\ell$-block containing $\be$. As $\be$ is not liftable, $B$ is not of defect 0. By Lemma \ref{gh6}, $\be$ is an integral linear combination of 
ordinary unipotent characters   restricted to the $\ell'$-elements; write
$$\be=\sum a_j\chi^\circ_j$$ 
for some distinct unipotent characters $\chi_j$ and integers $a_j\neq 0$. By Lemma \ref{kk8}, none of $\chi_j$ is of defect 0. Note that each $\chi_j$ is constant at the $\ell$-singular elements of $G$, and hence belongs to the principal $\ell$-block of $G$ by Lemma \ref{t45a}. So  $B$ is the principal block.  

 Set $N=N_G(Y_\ell)$. Then $N=N_G(Y)$ (arguing as in the proof of Lemma \ref{20tz}). 
Let $V$ be  an $FG$-module with Brauer character $\be$.   
Then $V|_N=V_1\oplus V_2$ for some  $FN$-module $V_2$, where $V_1$ is the Green correspondent of $V$; in addition,  $V_1$ is indecomposable and lies in the principal $\ell$-block of $N$, see  \cite[Proposition 62.42(ii)]{CR}.
By \cite[Theorem 6.10]{N}, $O_{\ell'}(N)$ lies in the kernel of $V_1$. As $Y_{\ell'}\subseteq O_{\ell'}(N)$, we conclude that  $V_1|_{Y_{\ell'}}$ is trivial as stated.  \enp

In Lemma \ref{u23} $\be$ lies in $ \mathcal{E}_{\ell,1}$. Next we extend Lemma \ref{u23} to the case where $\be\in \mathcal{E}_{\ell,s}$ and  $1\neq s\in Z(G^*)$. 

\bl{zg8} Under the assumptions of Lemma {\rm \ref{u23}}, let $1\neq s\in Z(G^*)$ and let $\be\in \mathcal{E}_{\ell,s}$ be a  non-liftable \ir  Brauer character of G.
Suppose that   $\ell$ divides $|Y|$. Then $1_{Y_{\ell'}}$ is a constituent of  $\beta|_{Y_{\ell'}}$.\el

\bp  
By Lemma \ref{gh6}, $\be=\sum a_j\chi^\circ_j$ for some distinct  characters $\chi_j\in \mathcal{E}_{s}$ and integers $a_j\neq 0$. By  \cite[13.30]{DM},
there is some $\sigma \in \Irr(G/G_1)$ such that for all $j$ we have 
 $\chi_j=\chi_j'\otimes \si$, where $\chi_j'$ is a unipotent character of $G$.
Set $\be'=\sum a_j\chi'^\circ_j$. Then $\be'$ is a unipotent Brauer character and $\be=\be'\otimes \si$ (where we view $\si$ as a Brauer character). 
As $Y_{\ell'}\subseteq Y \subset G_1$, we have $\si(Y_{\ell'})=1$, and hence 
$(\chi_j)|_{Y_{\ell'}}=(\chi_j')_{Y_{\ell'}}$. Now the result follows from 
Lemma \ref{u23}. \enp

%\bl{us2} Let $G= \GL_d^\ep(q)$. Let Y be as in Lemma {\rm \ref{nr5}}, $g\in Y$, 
%and let $\be \in\Irr_\ell(G)$ be a  Brauer character of G. Suppose that $(\ell,|Y|)=1$ and $\deg\be(g)<|g| $.  Then either $\be$ is liftable, or 
%$Y=\langle g \rangle$ and $\deg\be(g)\geq |g|-1$.\el
%
%\bp 
%By Lemma \ref{nr5} (which is also valid if   $Y_{\ell'}=1$), either $\deg\be(g)=|g| $ or $\be$ is liftable, or $\be \in \mathcal{E}_{\ell,s}$ with
%$s\in Z(G^*)$. Suppose we are in the latter case. As in the proof of Lemma \ref{zg8},
% % by  \cite[13.30]{DM},
%we have $\chi_j=\chi_j'\otimes \si$, where $\chi_j'$ is a unipotent character of $G$ and  $\si(1)=1$. 
%Set $\be'=\sum a_j\chi'^\circ_j$. 
%Then $\be'$ is a unipotent Brauer character and $\be=\be'\otimes \si$.
%% (where we view $\si$ as a Brauer character). 
%Since $\si$ is linear, it suffices to prove the statement for $\beta'$.
%As $\chi_j'$ is unipotent,  $Z(G)$ lies in the kernel of each $\chi_j'$ and hence in the kernel of  $\be'$. By Lemma \ref{co8}, each  $\chi_j'$ is constant on 
%$S\setminus Z(G)$, and hence on $Y\setminus \{1\}$. So the result follows from Lemma \ref{22d}. 
%\enp
 
\bl{us2new} Let $G= \GL_d^\ep(q)$, $S$ a maximal torus of order $q^d-\eps$ in $G$, and let $\be \in\Irr_\ell(G)$ be a  Brauer character of $G$ that belongs to $\EC_{\ell,s}$ with $s \in Z(G^*)$. Then $\beta =\beta' \otimes \sigma$ for some linear character $\sigma$, and 
$\beta'$ a Brauer character of $G/Z(G)$ inflated to $G$. Furthermore, if $R$ is any $\ell'$-subgroup of $S/Z(G)$,  then $\beta'$ is constant on $R \setminus \{1\}$.\el

\bp 
As in the proof of Lemma \ref{zg8},
we have $\chi_j=\chi_j'\otimes \si$, where $\chi_j'$ is a unipotent character of $G$ and  $\si(1)=1$. Set $\be'=\sum a_j\chi'^\circ_j$. 
Then $\be'$ is a unipotent Brauer character and $\be=\be'\otimes \si$. 
Since $\si$ is linear, it suffices to prove the statement for $\beta'$.
As $\chi_j'$ is unipotent,  $Z(G)$ lies in the kernel of each $\chi_j'$ and hence in the kernel of  $\be'$. By Lemma \ref{co8}, each  $\chi_j'$ is constant on 
$S\setminus Z(G)$, and the result follows.\enp

\begin{propo}\label{d-case}
Let $G=\GL^\eps_d(q)$, and assume $2<d=\ell|(q-\eps)$. Let $\chi=\chi^\lambda$ be a unipotent character of $G$, labeled by the partition $\lambda$ of $d$.
Then the following statement hold.
\begin{enumerate}[\rm(i)]
\item $\chi$ belongs to the principal $\ell$-block of $G$.
\item If $1 \neq y \in Y$ with $Y$ as defined in Lemma {\rm \ref{ue1}}, then $\chi(y)=0$ if $\lambda$ is not a hook partition, and $\chi(y)=\pm 1$ otherwise.
Also $\ell\nmid |Y|$.
\item Suppose $\eps=+$ and $\lambda$ is not a hook partition. Then $\chi^\circ$ is an irreducible $\ell$-Brauer character, in fact this is the unipotent $\ell$-Brauer 
character $\psi^\lambda$ labeled by the same partition $\lambda$.
\item Suppose $\eps=+$ and $\lambda$ is a hook partition. If $\lambda \neq (1^d), (d)$ and $\psi^\lambda$ is the Brauer character labeled by
$\lambda$, then $\rho^{reg}_{Y}$ is a subcharacter of $\psi|_{Y}$. If $\lambda=(1^d)$, then $\psi^\lambda$ lifts to characteristic $0$.
\end{enumerate}
\end{propo}

\begin{proof}
(i) By the assumption, $\ell|(q-\eps 1)$, so the integer $e$ defined in \cite[\S3]{FS} is $1$. Clearly, the partitions $\lambda$ and $(d)$ of $d$ have 
same $1$-core, so by the main result of \cite{FS}, $\chi$ belongs to the $\ell$-block of the character $\chi^{(d)}$, the principal character.

\smallskip
(ii) Since $2<d|(q-\eps)$, $(q^d-\eps)/(q-\eps)$ is congruent to $d$ modulo $d^2$. So $|Y|=(q^d-\eps)/d(q-\eps)$ is coprime to $\ell=d$.
Next, by Lemma \ref{ue1}, every $1 \neq y \in Y$ is a regular semisimple element, contained in the maximal torus $S$ of order $q^d-\eps 1$.
Hence the result follows from \cite[Corollary 3.1.2]{LST}.

\smallskip
For (iii) and (iv), we have %$\eps=+$ and
$G = \GL_n(q)$. By \cite[Theorem 8.1(vii)]{J}, every irreducible constituent of the Brauer character $\chi^\circ$ is the unipotent Brauer character $\psi^\mu$ with $\mu \vdash d$ dominating $\lambda$.

First suppose that $\lambda \neq (1^d)$; equivalently, if 
$\lambda = (\lambda_1 \geq \lambda_2 \geq \ldots \lambda_k \geq 1)$( with $\sum^k_{i=1}\lambda_i = d$), then $\lambda_1 > 1$. This is also 
equivalent to the condition that $\lambda$ is $\ell$-regular. Since $\ell|(q-1)$, $d=\ell$ is the smallest positive integer $j$ such that $\ell$ divides $1+q+ \ldots +q^{j-1}$. 
Then $(1^d)$ cannot dominate $\lambda$. Thus every irreducible constituent of $\chi^\circ$  is $\psi^\mu$ with $\mu$ $\ell$-regular. Furthermore, by 
\cite[Theorem 8.2 (iii)]{J}, the submatrix of the decomposition matrix corresponding to unipotent, ordinary and $\ell$-Brauer, characters labeled by 
the $\ell$-regular partitions is just the $\ell$-modular decomposition matrix of the symmetric group $\mathsf{S}_d$. Now if $\lambda \vdash d$ is not a hook,
then $\lambda$ is a $\ell$-core, so the previous assertion implies that $\chi^\circ$ is irreducible, yielding (iii). (For the notions of $\ell$-regular partitions and partition dominance see  \cite[p. 237-238]{J}.)

Now suppose that $\lambda \neq (1^d)$ but $\lambda$ is a hook partition: $\lambda=(d-j,1^{j})$ with $1\leq j \leq d-2$. 
According to the hook length formula (see \cite[(21)]{Ol} or \cite{Ma1}), 
\begin{equation}\label{deg20}
  \chi(1)=q^{j(j+1)/2}\frac{(q^{d-j}-1)(q^{d-j+1}-1) \ldots (q^{d-1}-1)}{(q-1)(q^2-1) \ldots (q^j-1)}.
\end{equation}  
By the  previous assertion, and using 
Peel's theorem for $\mathsf{S}_d$ \cite[Theorem 24.1]{J1}, we see that 
\begin{equation}\label{red20}
  \chi^\circ = \psi^{(d-j,1^{j})}+\psi^{(d-j+1,1^{j-1})}.
\end{equation}  
Note that $\psi^{(d)}(1)=1$, and certainly $\psi^{(d-j,1^j)}(1) \leq \chi^{(d-j,1^j)}(1)$. It follows from \eqref{deg20} that $\psi(1)\geq \chi(1)-\chi^{(d-j+1,j-1)}(1)$, so %checking

%$$q^{j(j+1)/2}\frac{(q^{d-j}-1)(q^{d-j+1}-1) \ldots (q^{d-1}-1)}{(q-1)(q^2-1) \ldots (q^j-1)}-q^{j(j-1)/2}\frac{(q^{d-j+1}-1)(q^{d-j+2}-1) \ldots (q^{d-1}-1)}{(q-1)(q^2-1) \ldots (q^{j-1}-1)}=$$

%$$q^{j(j-1)/2}\frac{[q^j(q^{d-j}-1)(q^{d-j+1}-1) \ldots (q^{d-1}-1)-(q^j-1)(q^{d-j+1}-1)(q^{d-j+2}-1) \ldots (q^{d-1}-1)]}{(q-1)(q^2-1) \ldots (q^j-1)}=$$

%$$=q^{j(j-1)/2}\frac{(q^{d-j+1}-1)(q^{d-j+1}-1) \ldots (q^{d-1}-1)[q^j(q^{d-j}-1)-(q^j-1) ]}{(q-1)(q^2-1) \ldots (q^{j}-1)}$$

$$\psi(1) \geq q^{j(j-1)/2}\frac{(q^{d-j+1}-1) \ldots (q^{d-1}-1)(q^d-2q^j+1)}{(q-1)(q^2-1) \ldots (q^j-1)} \geq q^{(j-1)(d-j/2)}\frac{q^d-2q^j+1}{q^j-1},$$
(with equality when $j=1$), where $\psi:=\psi^{(d-j,1^j)}$. Since $\chi^\nu$ takes values $\pm 1$ on $Y \setminus \{1\}$ by (ii) for all hook partitions $\nu \vdash d$, an induction on $j$ using \eqref{red20} shows that 
$$|\psi(y)| \leq j+1$$
for all $y \in Y \setminus \{1\}$. So to prove that $\psi|_Y$ contains $\rho^{reg}_Y$, it suffices to show that 
$\psi(1) \geq (j+1)|Y|$. Recall that $2 <d|(q-1)$; in particular, $q \geq 4$, and $|Y|=(q^d-1)/d(q-1)$. So it suffices to show  
$$q^{(j-1)(d-j/2)}\frac{q^d-2q^j+1}{q^j-1} \geq \frac{j+1}{d} \cdot \frac{q^d-1}{q-1},$$
and this holds true since $q \geq 4$ and $1 \leq j \leq d-2$.

Finally, we consider the case $\psi{(1^d)}$. Since $d|(q-1)$, we can find a $d$-element $s \in \FF^\times_{q^d} \setminus \FF_q$ that % such an element
has degree $d$ over $\FF_q$ (that is, $\FF_q(s)=\FF_{q^d})$. By \cite[Theorem 8.2]{J}, the complex module $S(s,(1))$ of $G$ is irreducible modulo $d=\ell$, and by the main result of \cite{DJ}, its Brauer character is exactly $\psi^{(1^d)}$, and thus $\psi^{(1^d)}$ is liftable.\end{proof}

 \begin{propo}\label{z00} Let $G_1= \SL_d^\ep(q)$,  $\phi\in\Irr_\ell(G_1)$ 
and $\dim\phi>1$.  Let $g\in G_1$ be an \ir p-element for a prime $p\neq \ell$.
Suppose that $1<\deg\phi(g)<|g|$.  Then $p\neq d$, $o(g)=|g|= (q^d-\ep 1)/((q-\ep 1)\cdot \gcd(d,\ell,q-\ep 1))$, $\deg\phi(g)= |g|-1$, and either 
$d \nmid (q-\eps 1)$, or $\eps=-$ and $d=\ell \mid (q+1)$.
\end{propo}

\bp  We can assume that $g$ lies in a maximal torus $S_1$ of $G_1$ of order $(q^d-\ep 1)/(q-\ep 1)$, see \cite[Lemma 5.1]{TZ8} for $\ep=+$ and \cite[Proposition 5.2]{TZ8} for $\ep=-$, and $p \nmid (q-\eps)$. By Lemma \ref{ue1}, we have $p\neq d$, 
Sylow $p$-subgroups of $G_1$ are cyclic and $g\in Y_{\ell'}$, the $\ell'$ direct factor of the  Hall $d'$-subgroup $Y$ of $S_1$.  

If $\phi$ is liftable then the result is true by  \cite[Theorem 1.1, parts (1) and (6)]{Z3}. 

From now on, we assume that $\phi$ is not liftable. By Lemma \ref{ct1}, $\deg\phi(g)=\deg\tau(g)$,
where $\tau$ is an \irr of $G:=\GL_d^\ep(q)$, such that $\phi$ is a constituent of $\tau|_{G_1}$. 
As $g$ is a $p$-element with $p \nmid d(q-\eps)$, $g$ has the same central order in $G$ and in $G_1$, which is equal to $|g|$.
Furthermore,
note that $G/(Z(G)G_1)$ is cyclic of order $(d,q-\eps)$, and $\phi$ extends to a 
$G$-invariant representation of $Z(G)G_1$. In particular, if $(d,q-\eps)=1$, then $G = Z(G) \times G_1$ and $\tau(1)=\phi(1)$.

\smallskip
(a) Suppose that $\tau$ is liftable to a complex character $\chi$ of $G$. If $\tau(1)=\phi(1)$, then certainly 
$\chi|_{G_1}$ is a lift of $\phi$, contrary to the assumption. Hence we must have that $d|(q-\eps)$ and 
$\tau(1)=d\phi(1)$. Now, if $\chi$ is reducible over $G_1$, then $\chi|_{G_1}$ is a sum of 
$d$ $G$-conjugate irreducible characters of $G_1$, and one of these will be a lift of $\phi$, again a contradiction.
Thus $\chi|_{G_1}$ is irreducible, and $\deg\chi(g) = \deg\tau(g) < |g|=o(g)$. Applying \cite[Theorem 1.1, parts (1) and (6)]{Z3} to $\chi|_{G_1}$,
we obtain that $\chi(1) = (q^d-\eps)/(q-\eps)-1$, which is congruent to $-1$ modulo $d$. It follows that $\chi(1)/d=\tau(1)/d$ cannot be equal to $\phi(1)$, a contradiction. (Alternatively, we can apply \cite[Theorem 1.1]{GT1} to see that $\chi(1)/d$ is too small to be the degree of a nontrivial irreducible Brauer character of $G_1$). 

\smallskip
(b) We have shown that 
$\tau$ is not liftable, and recall that $\deg \tau(g) = \deg \phi(g) < |g|=o(g)$. By Lemma \ref{nr5}, $\tau \in \EC_{\ell,s}$ for some 
$s \in Z(G^*)$, and hence $|s|$ is coprime to $|Y_{\ell'}|$ (indeed, any prime that divides both $q-\eps$ and $|S_1|$ must be $d$). 
Hence $\deg\tau(g) = o(g)-1$ and $Y_{\ell'}=\langle g \rangle$ by Lemma \ref{rm1}(ii) and 
Lemma \ref{22d}; moreover, $1$ is not an eigenvalue of $\tau(g)$. 
The latter however contradicts Lemmas \ref{u23} and \ref{zg8} if $\ell$ divides $|Y|$ (as $g \in Y_{\ell'}$).  Thus $\ell \nmid |Y|$, $Y = Y_{\ell'}$, and  the result follows if $d \nmid (q-\eps)$.

Assume now that $d|(q-\eps)$. By Lemma \ref{us2new}, we may assume that $\tau$ is unipotent. Suppose $\ell \neq d$. Then we embed $Y= \langle g \rangle$,
of order $(q^d-\eps)/d(q-\eps)$ in $A:=S/Z(G)$ of order $(q^d-\eps)/(q-\eps)$. Here $\ell \nmid |A|$, and so $\tau$ is constant on $A \setminus \{1\}$ by Lemma \ref{us2new}. Applying Lemma \ref{22d} to $A$, we see that $1$ is an eigenvalue of $\tau(g)$, a contradiction. In the remaining case $d=\ell$, and $|g|=|Y|=(q^d-\eps)/d(q-\eps)$. Furthermore, $\eps=-$ by Proposition \ref{d-case}.\enp

\subsection{General case}

\bl{kt5}   Let $G= \SL_n^\ep (q)$, $n\geq 3$, and let  $g\in G$ be an \ir  p-element. Suppose that Sylow $p$-subgroups of $G$ are cyclic.   
 Let $\phi\in\Irr_\ell(G)$ with $\dim\phi>1$.
 Then $\deg\phi(g)\geq |g|-1$. Moreover, $p > n$, and either 
 \begin{enumerate}[\rm(i)]
 \item $\deg\phi(g) =|g|$, or 
 \item $n=d^a$ for a prime $d\neq p$. If $d>2$ then $|g|=(q^{n}-\ep 1)/((q^{n/d}-\ep 1) \cdot \gcd(d,\ell,q^{n/d}-\ep 1))$ is a $p$-power,
and either 
$d \nmid (q^{n/d}-\eps 1)$, or $\eps=-$ and $d=\ell \mid (q^{n/d}+1)$. If $d=2$ then $\ep=+$ 
 and  $|g|=(q^{n/2}+1)/(2,q-1)$ is a $p$-power.
 \end{enumerate}
 \el 

\bp  
By \cite[Lemma 5.1]{TZ8} for $\ep=+$ and \cite[Proposition 5.2]{TZ8} for $\ep=-$, the assumptions that $g$ is \ir   and the Sylow $p$-subgroups of $G$ are cyclic
imply that $p$ is a primitive prime divisor of $(\ep q)^n-1$ (in the sense of \cite{Zs}); in particular, $p > n$.  
Assume that $\deg\phi(g) \leq |g|-1$.

\smallskip
(a) Suppose first that $n$ has an odd prime divisor $d$. 
Since $p$ is a primitive prime divisor of $(\ep q)^n-1$, 
$g$ is contained in a subgroup $H_d\cong \SL_d^\ep (q^{n/d})$. 
Applying Proposition \ref{z00} to the non-trivial \ir constituents of $\phi|_{H_d}$, we conclude that $p \neq d$,
$\deg\phi(g)=|g|-1$, and 
$$|g|=\frac{q^{n}-\ep 1}{(q^{n/d}-\ep 1)\cdot \gcd(d,\ell,q^{n/d}-\ep 1)}.$$

Assume in addition that $n$ has another odd prime divisor $e \neq d$.
Since $p$ is a primitive prime divisor of $(\ep q)^n-1$, 
$g$ is also contained in a subgroup $H_e\cong \SL_e^\ep (q^{n/e})$. 
Applying Proposition \ref{z00} to the non-trivial \ir constituents of $\phi|_{H_e}$, we conclude that 
$$|g|=\frac{q^{n}-\ep 1}{(q^{n/e}-\ep 1)\cdot \gcd(e,\ell,q^{n/e}-\ep 1)}.$$  
Without loss we may assume $d \neq \ell$. It follows that
$(q^{n/e}-\ep 1)\cdot \gcd(e,\ell,q^{n/e}-\ep 1)=q^{n/d}-\ep 1$. If $\gcd(e,\ell,q^{n/e}-\ep 1) =1$, then  $d=e$, a contradiction. So $e=\ell$ divides $q^{n/e}-\ep 1$, as well as 
$q^{n/d}-\ep 1$. Since $d$ and $e$ both divide $n$, we can write $s:=q^{n/de} \geq 2$ and get  $(s^d -\ep 1)e=s^e-\eps 1$. Now if $d > e$, then $(s^d-\ep 1)e > s^d-\eps 1 > s^e -\eps 1$, a contradiction. So $e > d$; in particular, $e \geq 5$. By \cite{Zs}, $(\ep s)^e-1$ admits a primitive prime divisor 
$v > e$; in this case, $v$ divides $s^e-\eps 1$ but not $(s^d-\ep 1)e$, a contradiction. 

Thus we have shown that if $n$ is odd, then $n$ is a power of $d$.

\smallskip
(b) Suppose that $n$ is even. Then   $G= \SL_n(q)$ by \cite[Proposition 5.2(ii)]{TZ8}. 
Note that $p>2$ as Sylow $p$-subgroups of $G$ are cyclic; furthermore, 
as $p$ is a primitive prime divisor of $q^n-1$ we have $p > n$ 
and $p|(q^{n/2}+1)$. In particular, $o(g)=|g|$. If $n$ is a $2$-power, then $\deg\phi(g)=|g|-1$ and 
$|g|=(q^{n/2}+1)/\gcd(2,q-1)$ by \cite[Corollary 5.9(i)]{TZ8}, and so we are done. 

Suppose now that $n$ is not a $2$-power.
Then we choose $d$ to be odd prime divisor 
of $n$ and $e=2$. As above, we can embed $g\in H_d= \SL_{d}(q^{n/d})$, and 
 $|g|=(q^n-1)/(q^{n/d}-1)\cdot\gcd(d,\ell,q^{n/d}-1)$ is a $p$-power. 
Similarly, $g\in H_e= \SL_{2}(q^{n/2})$, and recall that
$p|(q^{n/2}+1)$. By \cite[Lemma 3.3]{TZ8} applied to $H_e$,  either $q$ is odd and $|g|=(q^{n/2}+1)/2$, or $q$ is even and $|g|=q^{n/2}+1$. 
The latter implies $n/2$ to be a 2-power  (see Lemma \ref{zgm}, note that $(G,p) \neq (\SL_6(2),3)$ as Sylow $p$-subgroups of
$G$ are cyclic), a contradiction.  
So $q$ is odd. By Proposition \ref{z00}, $\deg\phi(g)=|g|-1$, and   
$$|g|=(q^{n/2}+1)/2= (q^n-1)/(q^{n/d}-1)\gcd(d,\ell,q^{n/d}-1),$$ whence  
$(q^{n/d}-1)\gcd(d,\ell,q^{n/d}-1)=2(q^{n/2}-1)$. As $d \geq 3$ divides $n/2$ and $q$ is odd, by \cite{Zs} we can find a primitive 
prime divisor $v>d$ of $q^{n/2}-1$. In this case, $v$ does not divide $(q^{n/d}-1)\gcd(d,\ell,q^{n/d}-1)$, a contradiction.\enp
  
\begin{examp} 
{\em Let $G= \SL_9(2)$, $d=3, |g|=p=73$. Then $g$ is contained in a subgroup $H \cong \SL_3(8)$, and $(2^9-1)/(2^3-1) =73$ is a prime.

%\med Let $G= \SL_9(3)$, $d=3$.  Then  $q^6+q^3+1=3^6+3^3+1=757$ is a prime.  

%\med Let $G=\SU_9(7)$, $d=3$, $q=7$, $g\in H=\SU_3(7^3)$. Then $q^6-q^3+1=7^6-7^3+1=117649-  343+1=   117307$ is a prime.

Let $G=\SU_9(49)$. Then $49^6 - 49^3+1=13841287201 -117649+1=13841169553$ is a prime. This shows that $q$ in Lemma \ref{kt5} is not necessarily a prime for $n=9$.}
\end{examp}

Below we generalize the result of Lemma \ref{kt5} by omitting the condition of irreducibility of $g$.
%to groups $G= \SL_n^\ep (q)$ 

% \begin{comment}
\begin{propo}\label{zu3} Let $G= \SL_n^\ep (q)$, $n\geq 3$ and let  $g\in G$ be a p-element. Suppose that Sylow p-subgroups of G are cyclic.   Let $\phi\in\Irr_\ell(G)$, where $\dim\phi>1$. Then $\deg\phi(g)\geq |g|-1$. In addition, one of the \f holds.

\begin{enumerate}[\rm(i)]

\item $\deg\phi(g)= |g|.$ 

\item $g$ is  \ir in G, $n=d^a$ for an odd prime $d$, $p>n$ and furthermore we have 
$|g|=(q^n-\ep1)/((q^{n/d}-\ep1)\cdot\gcd(d,\ell,q^{n/d}-\ep 1))$, and either 
$d \nmid (q-\eps 1)$, or $\eps=-$ and $d=\ell \mid (q^{n/d}+1)$.

\item $\ep=+$, $n$ is a $2$-power, q is odd,   $|g|=(q^{n/2}+1)/2$ and $g$ is contained in  subgroups of G isomorphic to $ \Sp_n(q)$ and $ \SO^-_n(q)$.

\item $\ep=-$, $n-1=d^a > 1$ for an odd prime $d$, $p > n-1$,  and furthermore we have
$|g|=(q^{n-1}+1)/((q^{(n-1)/d}+1) \cdot \gcd(d,\ell,q^{n/d}+1))$. Moreover, either 
$d \nmid (q^{(n-1)/d}+ 1)$, or $d=\ell \mid (q^{(n-1)/d}+1)$.
\end{enumerate}
\end{propo}

\bp  Note that if the proposition holds for  $g$ a generator of a \syl of $G$, then $\deg\phi(h)=|h|$ for any element 
$h \in \langle g \rangle$ of order less than $|g|$. Hence we will assume that $g$ is a generator of a Sylow $p$-subgroup of $G$.
 Let $V$ be the natural module for $G$. 
 
 If $g$ lies in a parabolic subgroup, then the result follows from \cite[Theorem 1.1]{DZ1} and \cite{DZ1a}, and in this case $\deg\phi(g)= |g|$. (Note that in the  exceptional cases recorded there Sylow $p$-subgroups are not cyclic.)  So we will assume that $g$ does not lie in any parabolic subgroup of $G$ and 
 that $\deg\phi(g) < |g|$.
 
 We next show that either $g$ is \ir on $V$, or $\ep=-$ and $g$ stabilizes a non-degenerate subspace $V_1$ of $V$ of dimension $n-1$ and acts irreducibly on it. Indeed, suppose that $g$ is reducible and does not lie in any parabolic subgroup of $G$; in particular, $\eps=-$ and $V = \FF_{q^2}^n$.  Let $0 \neq W \neq V$ be any $g$-invariant subspace of $V$. If $W_1:=W \cap W^\perp$ is nonzero,
then $g$ lies in the parabolic subgroup 
$\mathrm{Stab}_G(W_1)$. So $W_1=0$, and thus any $g$-invariant subspace of $V$ is non-degenerate. Replacing $W$ by $W^\perp$ if necessary, we may assume that $g$ acts on $W$ nontrivially. As \syls are cyclic, %  $W$ is a unique subspace with these properties. So 
$p \nmid (q+1)$. So $g \in \SU(W) \times \SU(W^\perp)$. Again using the fact that Sylow $p$-subgroups of $G$ are cyclic, we deduce that
$g$ acts trivially on $W^\perp$. As $g$ fixes no (nonzero) singular vector, the subspace $W^\perp \neq 0$
contains no singular vector, whence $\dim W^\perp=1$. This is true for any nonzero proper $g$-invariant subspace in $V$. Hence $g$ acts irreducibly on $W$ (and $W$ 
has codimension $1$).

\smallskip
Suppose that $g$ is \ir on $V$ and $\deg\phi(g)<|g|$. Then $p$ is coprime to $n$ as recorded in Lemma \ref{kt5}. If $n$ is odd, then the result follows from Lemma \ref{kt5}.
 Suppose  that $n$ is even. Then $\ep=+$, and, by \cite[Corollary 5.9]{TZ8}, either $n$ is a 2-power, $|g|=(q^{n/2}+1)/(2,q+1)$ 
and $\deg\phi(g)\geq |g|-1$ or $(n/2, q, |g|) = (3, 2, 9)$; the latter case is ruled out as Sylow 3-subgroups of $\SL_6(2)$ are not cyclic. 
 In the former case     $g$ is contained in a subgroup
$X\cong \Sp_n(q)$ and $X\cong \SO^-_n(q)$, and $q$ is odd  by \cite[Proposition 5.7]{TZ8}.   

\smallskip
Suppose now that $g$ is reducible in $G$; then $G=\SU_n(q)$ and $g$ is \ir in a subgroup $\SU_{n-1}(q)$ of  $  G$. By the above,  %Lemma \ref{kt5} 
applied to $\SU_{n-1}(q)$ in place of $\SU_n(q)$, we have $p > n-1=d^a$ for some odd prime $d$ and an integer $a>0$,  and the result follows. \enp
 
%\end{comment} 
 
We will show later, see Corollary  \ref{ss2}, that in fact the case with $d=2$ in Proposition %\ref{kt5} %  
\ref{zu3} 
does not  occur when Sylow $p$-subgroups of $G$ are cyclic.
We also conjecture that case (iv) does not actually occur 
but our method does not lead to this conclusion. For $n=4$ this is true, see Lemma \ref{uu4}.  For small values of $n>3$ one can use explicit decomposition matrices of $G$ modulo $\ell$ available in \cite{DM15}, see also \cite{TB}.
 
\bl{uu4} Let $G=\SU_4(q)$ and let $g\in G$ be a $p$-element with $3< p|(q^2-q+1)$. Let $1_G\neq \phi\in\Irr_\ell(G)$, where $(\ell q,p)=1$. Then $\deg\phi(g)=|g|$. \el

\bp Note that Sylow $p$-subgroups of $G$ are cyclic. So for $\ell=0$ the result is available in \cite{Z3}. Therefore, the result is true if $\phi$ is liftable. 
We may assume that $g$ is contained in a standard subgroup $H\cong \SU_3(q)$ of $G$.
By \cite[Prop 6.1(ii)]{TZ8}, every non-trivial \ir constituent $\tau$ of $\phi|_H$ is of degree $q^2-q$, in particular, it is a Weil character. In addition, $|g|=q^2-q+1$ and  all $|g|$-roots of unity but 1 are \eis of $\tau(g)$. Suppose the contrary,  that $\deg\phi(g)<|g|$. Then $\phi|_H$ has no trivial \ir constituent, 
and hence $\dim\phi$ is a multiple of $q^2-q$. By \cite[Theorem 2.5]{GMST}, $\phi$ itself  is a Weil character. But then 
$\dim\phi\in\{1,(q^4-1)/(q+1),(q^4+q)/(q+1)\}$ cannot be a multiple of $q^2-q$, a contradiction.\enp

 \section{The case with $G=\Spin_{2n}^-(q)$}

\subsection{Groups $\Spin_{2n}^-(q)$ for $n>4$ odd}

\begin{propo}\label{no3} Let $G=\Spin_{2n}^-(q)$ for $n$ odd, and let $g\in G$ be a $p$-element such that $|g|$ divides $q^n+1$ and 
Sylow $p$-subgroups of $G$ are cyclic. Let $1_G\neq \phi\in\Irr_\ell(G)$. Then $\deg\phi(g)\geq |g|-1$. Moreover,   $\deg\phi(g)=|g|$ unless possibly $n=d^a$ for an odd prime d, $|g|=(q^n+1)/((q^{n/d}+1)\cdot\gcd(d,\ell,q^{n/d}+1))$, and either 
$d \nmid (q^{n/d}+1)$, or $d=\ell \mid (q^{n/d}+1)$.
\end{propo}

\bp Note that the assumptions that $p|(q^n+1)$ and Sylow $p$-subgroups of $G$ are cyclic imply that $G$ contains a group $H\cong \SU_n(q)$ with $g\in H$ being 
irreducible on $\F_{q^2}^n$. Then the result follows from Lemma  \ref{kt5}. 
\enp

\subsection{Groups $G=\Spin^-_8(q)$ with $q$ odd}

\bl{cs8} Let $G=\SO^-_{8}(q)$  with the standard  $\FF_q G$-module $V$, and embed $K = \Omega^-_4(q^2)$ in $G$.
\begin{enumerate}[\rm(i)]
\item Every cyclic subgroup C of order $(q^{4}-1)/(2,q-1)$ in $K$ stabilizes a totally singular subspace $W$ of dimension $ 2$ in $V$. 

\item  If $W'\neq 0$ is a totally singular subspace of  V and  $CW'=W'$ then $\dim W'=2$.

\item  Let $P:=\mathrm{Stab}_G(W)$ and let U be that that the unipotent radical of P. Then $|C_C(Z(U))|\geq (q^2+1)/2$.
\end{enumerate} 
\el

\bp  By the proof of \cite[Proposition 2.9.1(v)]{KL}, $K \cong \PSL_2(q^4)$ is the image of $\SL_2(q^4)$ acting on the orthogonal space 
$V_1 = \FF_{q^2}^4$ in such a way that a  generator $g_1$ of the split torus $C_{q^4-1}$ has spectrum 
$$\{\xi^{q^2+1},\xi^{-q^2-1},\xi^{q^2-1},\xi^{-q^2+1}\}$$
for some $\xi \in \FF_{q^4}^\times$ of order $q^4-1$. It is easy to see that the $\xi^{q^2+1}$-eigenspace $W_1$ for $g$ is totally singular in $V_1$. 
Identifying $V$ with $V_1$ viewed over $\FF_q$ and $C$ with the image of $\langle g_1 \rangle$ in $H$, $W_1$ becomes a totally singular 
$\FF_q$-space of dimension $2$, the desired subspace $W$ in (i).

Furthermore, the $\langle g_1 \rangle$-module $V$ decomposes into irreducible summands as
$W \oplus W_3 \oplus W_2$, where $W_2^\perp = W \oplus W_3$ is non-degenerate of type $+$, $W_3 \cong W^*$, and $W_2  \cong \FF_q^4$ is of type $-$. This proves (ii).

For (iii), one observes that $Z(U)=\{h\in U: (h-\Id)V\subset W\}$. Choose a basis $(b_1\ld b_8)$ of $V$ such that $b_1,b_2\in W$, $b_3,b_4\in W_3$ and $b_5,b_6,b_7,b_8\in W_2$. Then 

\begin{center}$Z(U)=\begin{pmatrix}\Id_2&y&0\\ 0&\Id_2&0\\0&0&\Id_4 \end{pmatrix}$
and $g_1 =\begin{pmatrix}a&0&0\\ 0&{}^ta\up&0\\0&0&b \end{pmatrix}$,\end{center}
where $a\in \GL_2(q)$ has order $q^2-1$, $b \in \GL_4(q)$ and $y$ runs over some set of $(2\times 2)$-matrices over $\FF_q$. 
Hence $g_1^{q^2-1}$ acts trivially on $Z(U)$, and the statement follows.\enp

A group $X\neq 1$ is called {\it of  extraspecial type} if $X'=Z(X)$ coincides with the Frattini subgroup of $X$ and  $Z(X/Z_0)\cong Z(X)/Z_0$ for every proper subgroup $Z_0$ of $Z(X)$. It is explained in \cite[Lemma 3.3]{DZ4}
that if $Z_0$ is a maximal subgroup of $Z(X)$ then $X/Z_0$ is an extraspecial group. Note that $|X|$ is a prime power.
 
\bl{aa1} Let $G=\SO_8^-(q),P$ be as in Lemma {\rm \ref{cs8}}, and let $U$ be the unipotent radical of $P$. Then $U$  is of extraspecial type. In addition, if $N$ is a normal subgroup of $U$ then either $N\subset Z(U)$ or
 $Z(U)\subseteq N$.\el

\bp These facts are well known, see for instance   \cite[Lemma 3.8 and Lemma 3.12]{DZ4}.\enp

\bl{eg1} Let E be a group of extraspecial type and A an abelian subgroup of E. Let $\lam\in\Irr_{\ell}(E)$. Suppose that   $\dim\lam>1$ and  $A\cap Z(E)=1$. Then $\lam|_A$ is a multiple of $\rho_A^{reg}$, the regular \rep of A.   In particular, $1_A$ is a constituent of $\lam|_A$.\el

\bp   Let $\chi$ be the Brauer character of $\lam$. As $\chi(1)>1$,  we have $(\ell, |E|)=1$. Since $E'\leq Z(E)$, it follows  that
 $\chi(x)=0$ for every $x\in E\setminus Z(E)$. As $A\cap Z(E)=1$, we have $\chi(a)=0$ for every $1\neq a\in A$. This implies the result.  \enp

 \bl{dn8} Let $G=\Spin^-_{8}(q)$, $q$  odd. Let $g\in G$ be a p-element such that $p|(q^4+1)$, $p\neq 2,\ell$.  Let $1_G\neq \phi\in\Irr_\ell(G)$, where $\ell\neq p$. Then $\deg\phi(g)=|g|$. \el

 \bp 
 Suppose the contrary, that $\deg\phi(g)<|g|$. As usual, we may assume that 
$\langle g \rangle  $  is a \syl of $G$.
As mentioned in Lemma \ref{cs8},    $\SO^-_8(q)$  contains a subgroup  isomorphic to $\PSL_2(q^{4})$. Then $G$ contains a subgroup $X$, say, such that $Z(G)\subset X$ and $X/ Z(G)\cong \PSL_2(q^{4})$. Set $H=X'$, the derived subgroup of $X$. 
Then either $H\cong \SL_2(q^{4})$ or $H\cong \PSL_2(q^{4})$. Since $|H|_p=|G|_p$, 
by Sylow's theorem we may assume that $g\in H$.

 Let $\tau$ be a nontrivial  \ir \ccc
of $\phi|_H$. Then $\deg\tau(g)<|g|$. By \cite[Lemma 3.3]{TZ8}, 
 $|g|=(q^4+1)/2$ and $\deg\tau(g)=\dim\tau=(q^4-1)/2$ and
  $1$ is not an \ei of $\tau(g)$. Therefore, $1_H$ is not a constituent of $\phi|_H$. So every \ir \ccc of $\phi|_H$ is of dimension $\dim\tau=(q^4-1)/2$. 

Note that  $H$ has exactly two \ir Brauer characters of degree $(q^4-1)/2$. Let
$A$ be a $\langle g \rangle$-invariant maximal unipotent subgroup of $H$. It is well known that  $1_A$ is not
an \ir \ccc of $\tau|_A$ (one can use the characters of $H$ of degree $(q^4-1)/2$ to see this). \itf $1_A$ is not an \ir \ccc of $\phi|_A.$

We show that this leads to a contradiction.

Let $B=N_H(A)$.    
By the Borel-Tits theorem \cite[Theorem 3.1.3]{GLS}, there is a parabolic subgroup $P$ of $G$ such that $B\subseteq  P$ and $A\subset  U$, where $U$
is the unipotent radical of $P$. As $g \in B$,  
by Lemma  \ref{cs8}(iii), $P$ is the stabilizer in $G$ of 
a 2-dimensional totally singular subspace $W$ of $V$, the natural  $\FF_qG$-module. Therefore, $U$ is a group of extraspecial type by Lemma \ref{aa1}. Then we are left to show that $A\cap Z(U)=1$ in view of Lemma \ref{eg1}.  

For this observe that the action of $g$ (of order $(q^4-1)/2$) on $A$ by conjugation has $2$ orbits of size $(q^4-1)/2$. If
 $A\cap Z(U)\neq 1$ then $A\cap Z(U)$ contains one or  two orbits of size $(q^4-1)/2$,
and hence $g$ acts faithfully on $A$, hence on $Z(U)$. 
This contradicts Lemma \ref{cs8}(iii). \enp

\subsection{Groups $G=\Spin^-_8(q)$ with $q$ even}

As $q$ is even, $G$ has a faithful \ir \rep of degree $8$, and, as a subgroup of 
$\GL_{8}(F)$, it is  realized as the derived subgroup  of  $O_{2n}^-(q)$, 
hence it  has index 2 in the latter group. In fact, $G={\mathbf G}^{\F}$, 
where ${\mathbf G}$ is a connected algebraic group of type 
$D_4$,  and $\F$ is a certain Steinberg  endomorphism of ${\mathbf G}$.

We start with the \f useful observation:

 \bl{44q} Let $G=\Spin_8^-(q)$, q even,  $g\in G$ be an \ir $p$-element  and let $\phi \in\Irr_\ell(G)$ be a non-trivial \irr of $G$, $\ell\neq p$. Suppose that 
$\deg\phi(g)<|g|$. Then $|g|=q^4+1$,  $\dim\phi(1)$ is a multiple of $q^4$ and  $1$ is not an eigenvalue of $\phi(g)$. \el

\bp By \cite[Proposition 5.2(iv)]{TZ8}, $|g|$ divides $q^4+1$ and $g$ is contained in a subgroup  $H\cong \SL_2(q^4)$. By \cite[Lemma 3.3]{TZ8},
 $|g|=q^4+1$ and, if $\tau$ is a non-trivial \ir \ccc of $\phi|_H$, then either $\dim\tau=q^4$ and $1\not\in\spec \phi(g)$ or $\dim\phi=q^4-1$ and $\deg\phi(g)=|g|-2$. In the latter case  $\eta,\eta\up\notin\spec \phi(g)$, where $\eta\neq 1$ is some primitive $|g|$-root of unity. We show that this case does not in fact occur. Suppose the contrary. Then $|N_G(T)/T|=4$ by \cite[Lemma 3.6(3)]{z14}. (In \cite[Lemma 3.6(3)]{z14}  $G=\SO^-_{8}(q)$, however, it is clear from the context there that the lemma is applied to $G={\mathbf G}^{\F}$, where ${\mathbf G}$ is a  simple simply connected algebraic group of type $D_4$.) \itf $g$ is conjugate to 4 distinct elements of shape $g^i$, and hence some $i$ differs from $\pm 1$.
Therefore, $\eta^i\neq \eta^{\pm 1}$ is an \ei of $\phi(g)$. This implies $\deg\phi(g)=|g|$, a contradiction.\enp

\bl{qe4} Let $G=\Spin_8^-(q)$, q even,  $g\in G$ be an \ir $p$-element  and let $1_G\neq \phi\in\Irr_\ell(G)$, $\ell\neq p$. Then  $\deg\phi (g)=|g|$.  \el

\bp By Lemma \ref{44q}, $|g|=q^4+1$, so $g$ is a generator of a cyclic maximal torus $S$ of $G$ of order $q^4+1$. It is well known that $G$ is a simple group and that every element $1\neq t\in S$ is regular. So we can use the results of \cite[\S 3]{TZ20}. By  \cite[Lemma 3.6]{TZ20}, either $\phi$ is unipotent or $\phi(y)=0$ for every $1\neq y\in S$ and the result follows, or $\phi$ is liftable. In the latter case $\deg\phi(g)=|g|$ by \cite{Z3}.

We use \cite{DuM}, where the unipotent characters   
are labelled by the \ir characters of Weyl group of type $D_4$, that is, by unordered pairs of partition of 4. (See \cite[Section 13.2]{C} for the general setting.)

We use the decomposition numbers modulo $\ell$ for unipotent characters of $G$. As $p\neq \ell$ and $|g|=q^4+1$, we have $(\ell,q^4+1)=1$ here. So the cases to be considered are (i) $\ell|(q-1)$, (ii) $\ell|(q+1)$, (iii) $\ell|(q^2+1)$, (iv) $\ell|(q^2-q+1)$, (v) $\ell|(q^2+q+1)$. 
Note that the table of unipotent characters for $G$ is available in Chevie and also at L\"ubeck's home page \cite{Lu}. 

\med (i) {\bf The case $\ell|(q-1).$} By \cite[Satz 6.3.7(3)]{H1}, every unipotent Brauer character is liftable provided $\ell$ is coprime to the order of the Weil group $W$ of the simple algebraic group of type $D_4$. If $\be$ is liftable, the result follows from \cite{Z3}. As $\ell\neq 2$ and $|W|=2^3.4!$, we are left with $\ell=3$.   Gerhard Hiss (a private communication) deduced from   \cite{GH} an explicit shape of the decomposition matrix for $\ell=3$, see Table 1.

%\newpage
%\vskip10pt
%\begin{figure}
 \begin{center}{Table 1: The $3$-decomposition matrix for unipotent characters of $G=D^-_4(q)$  }

 \vspace{10pt}
\begin{tabular}{|c|c|c|c|c|c|c|c| c|}
\hline
character  degree&label& $\be_1$ & $\be_2$   & $\be_3$ & $\be_4$ & $\be_5$ &$\be_6$ &$\be_7$   \\
\hline $ 1$&$3.$& $1$ & $0$   & $0$ & $0$ & $0$& $0$& $0$\\  
 \hline
$q(q^4+1)$&$21.$&$1$   & $1$&$0$ & $0$ & $0$& $0$& $0$\\ 
 \hline
$(q^3/2)(q^2-q+1)(q^4+1)$&$1^3.$& $0$   &$1$ & $1  $   & $0$ & $0$& $0$& $0$  \\
  \hline
$q^2(q^4+ q^2+1)$&$2.1$& $0$   &$0$ &  $1$ & $0$ & $0$& $0$& $0$  \\
  \hline
$(q^3/2)(q^2+q+1)(q^4+1)$&$1^2.1  $& $0$   &$ 0$ &  $0$ &$1$& $0$& $0$& $0$   \\
  \hline
$(q^3/2)(q^2+q+1)(q^4+1)$&$1.2$& $0$   &$0$ &  $0$ &$1$& $0$& $0$& $0$   \\
  \hline
$q^6(q^4+q^2+1)$&$1.1^2  $& $0$   &$0$ &  $0$ &$0$& $1$& $0$& $0$   \\
  \hline
$(q^3/2)(q^2-q+1)(q^4+1)$&$.3  $& $0$   &$0$ &  $0$ &$0$& $1$& $0$& $0$   \\
\hline
$q^7(q^4+1)$&$.21$& $0$   &$0$ &  $0$ &$0$& $1$& $1$& $0$   \\
\hline
$q^{12}$&$ .1^3  $& $0$   &$0$ &  $0$ &$0$& $0$& $1$& $1$   \\
\hline
\end{tabular}
\end{center}

So the non-liftable characters are $\be_2,\be_6$ and $\be_7$. We have 
$$\be_2=\chi_{2.}-1_G,~\be_6=\chi_{.21}-\chi_{.3},~\be_7=\chi_{.1^3}-\be_6=\chi_{.1^3}-\chi_{.21}+\chi_{.3},$$ 
whence
$$\begin{array}{lll}\be_2(1)& =q^5+q-1 & \equiv -1\pmod{q},\\  
   \be_6(1) & =q^{11}+q^7-(q^3/2)(q^2-q+1)(q^4+1)& \equiv (q^4 -q^3)/2\pmod{q^4},\\
   \be_7(1) & =q^{12}-q^{11}+q^7+(q^3/2)(q^2-q+1)(q^4+1)& \equiv (q^4  -q^3)/2\pmod{q^4}.\end{array}$$ 
These contradict  Lemma \ref{44q}. 

This states that  $\deg\phi(g)<|g|$ implies $q^4|\dim\phi$. We show that if $\phi$ is not liftable then $\dim\phi$ is not a multiple of $q^4$.

\med
(ii) {\bf The case $\ell|(q+1).$} The decomposition matrices of $G=D_{4}^-(q)$ for this case are determined in \cite[\S 4, p.36, Table 10]{DuM}.    If $q=2$ then $\ell=3$
and hence $\dim\phi\equiv 0\pmod{16}$. By inspection of the \ir Brauer character degrees in \cite[p. 244]{JLPW}, this condition is satisfied only for $\phi$ of degree 2464. As the Brauer character value of $\phi$ at $g$ equals $-1$, one easily deduces a contradiction. So we assume 
$q>2$ in the subsequent analysis of the case where $\ell|(q+1)$.   

%I ALWAYS ASSUNE THIS

Non-liftable $\ell$-modular \ir characters are 

\med
$\be_1(1)=q^2\Phi_3(q)\Phi_6(q)-q\Phi_8(q)=q^2(q^4+q^2+1)-q(q^4+1)\equiv -q\pmod{q^2};$ 

\med
$\be_2(1)=\frac{1}{2}q^3\Phi_6(q)\Phi_8(q)-1 \equiv -1\pmod{q}$, 

\med
$\be_3(1)=\frac{1}{2}q^3\Phi_6(q)\Phi_8(q)-\be_1(1)
\equiv q\pmod{q^2};$ % if $q=2$ then  $\be_3(1)=120$;

\med
$\be_4(1)=\frac{1}{2}q^3\Phi_3(q)\Phi_8(q)-q^2\Phi_3(q)\Phi_6(q)=q^2(q^2+q+1)[ (q/2)(q^4+1)-(q^2-q+1)]\equiv  -q^2\pmod{q^3/2};$ %if $q=2$ then  $\be_3(1)=280$;

\med
 $\be_5(1)=\frac{1}{2}q^3\Phi_3(q)\Phi_8(q)-1\equiv  -1\pmod{q};$

\med
 $\be_6(1)=q^6\Phi_3(q)\Phi_6(q)-\frac{q^3}{2}\Phi_3(q)\Phi_8(q)+1-\frac{q^3}{2}\Phi_6(q)\Phi_8(q)-\al (\frac{q^3}{2}\Phi_6(q)\Phi_8(q)-q^2\Phi_3(q)\Phi_6(q)\\ \equiv   1\pmod q$ (here $0\leq \al\leq 2$);   

\med
$\be_7(1)=q^7\Phi_8(q)-q^6\Phi_3(q)\Phi_6(q)+\frac{1}{2}q^3\Phi_6(q)\Phi_8(q)=q^7(q^4+1)-q^6(q^4+q^2+1)+(q^3/2)(q^2-q+1)(q^4+1)\equiv (q^3/2)\pmod{q^3}$; 

\med
$\be_8(1)=q^{12}-\al\be_7(1)-\be_4(1)-\be_3(1)-\be_1(1)$, where $0\leq \al\leq 2$. Here $\be_7(1)\equiv 0\pmod{q^3/2} $, $\be_4(1)\equiv -q^2\pmod{q^3/2}$, $\be_3(1)\equiv  -q^2\pmod{q^3/2}$, $\be_1(1)\equiv q^2-q\pmod{q^3/2}$. So $\be_8(1)\equiv  q^2+q\pmod{q^3/2}$. 

\med
 So $\be_i(1)$ is not a multiple of $q^4$ for $i=1\ld 8$  and $q>2$, whence the result in this case.

\med
We are left with the cases where  $(\ell,q^2-1)=1$, hence Sylow $\ell$-subgroups are cyclic.
In all these cases we have a contradiction by Lemma \ref{44q}.

\med
(iii) {\bf The case $\ell|(q^2+1)$.} 
In this case there are 5 non-liftable unipotent Brauer characters which we denote by $\be_1,\be_2,\be_3,\be_4,\be_5.$ By  \cite{TB}, and Fong-Srinivasan \cite[Section 8]{FS2}, we have

 \med
$\be_1(1)=\chi_{.3}(1)-1=q^3\Phi_6(q)\Phi_8(q)/2-1=(q^3/2)(q^2-q+1)(q^4+1)-1$ is odd;

\med
$\be_2(1)=\chi_{1.1^2}(1)-\chi_{1^2.1}(1)= q^6(q^4+q^2+1)-(q^3/2)(q^2+q+1)(q^4+1)=-q^3/2\pmod{q^3};$

%$q^3((q^3/2)-1)(q^2+q+1)(q^4+1)\equiv -q^3\pmod{q^4}$;

\med
$\be_3(1)=\chi_{.21}(1)-\chi_{21.}(1)=q^7(q^4+1)-q(q^4+1)=q(q^4+1)(q^6-1)\equiv -q \pmod{q^4}$; 

\med $\be_4(1)=\chi_{.1^3}(1)-\chi_{1^3.}(1)=q^{12}-(q^3/2)(q^2-q+1)(q^4+1)\equiv -q^3/2\pmod{q^3}$ and 

\med $\be_5(1)=\chi_{1.2}(1)-\chi_{2.1}(1)=(q^3/2)(q^2+q+1)(q^4+1)-q^2(q^4+q^2+1)\equiv -q^2\pmod{q^3/2}$.

\med
(iv) {\bf The case $\ell|(q^2+q+1)$.}   (In this case (and in (E) below)   
we can assume   $\ell\geq 7$ as $q^2\pm q+1\not\equiv 0\pmod{5}$ and $3|(q^2-1).$) 
Here there are 4 non-liftable unipotent Brauer character which we denote by $\be_1,\be_2,\be_3,\be_4.$ By \cite[p. 62, Prop. 5.8]{DuM}, we have

\med $\be_1(1)=\chi_{21.}(1)-1=q(q^4+1)-1$ is odd, as well as %(33 if $q=2$)

\med $\be_2(1)=\chi_{1^3.}(1)-\be_1(1)=\chi_{1^3.}(1)-\chi_{21.}(1)+1=(q^3/2)(q^2-q+1)(q^4+1)-q(q^4+1)+1$; %(2.3.34-34+1=171 if $q=2$)  

\med $\be_3(1)=\chi_{.21}(1)-\chi_{.3}(1)=q^7(q^4+1)-(q^3/2)(q^2-q+1)(q^4+1)\equiv -q^3/2 \pmod{q^3}$ and  

\med
$\be_4(1)=\chi_{.1^3}(1)-\be_3(1)=\chi_{.1^3}(1)-\chi_{.21}(1)+\chi_{.3}(1)=q^{12} -q^7(q^4+1) +(q^3/2)(q^2-q+1)(q^4+1)\equiv q^3/2\pmod{q^3}$. 
 
\med
(v) {\bf The case $\ell|(q^2-q+1)$.}  
In this case there are 4 non-liftable unipotent Brauer character which we denote by $\be_1,\be_2,\be_3,\be_4.$ By \cite[p. 79, Table 10]{DuM}, we have

\med
 $\be_1(1)=\chi_{1.2}(1)-1=q^2(q/2)(q^2+q+1)(q^4+1)-1$ is odd, as well as

\med
$\be_2(1)=\chi_{.21}(1)-\be_1(1)=\chi_{.21}(1)-\chi_{1.2}(1)+1=q^7(q^4+1) -(q^3/2)(q^2+q+1)(q^4+1)+1;$ %\equiv q^2(1-(q/2)(q+1))+1 \pmod{q^4}$,

\med
$\be_3(1)=\chi_{1^2.1}(1)-\chi_{21.}(1)=(q^3/2)(q^2+q+1)(q^4+1)-q(q^4+1)\equiv -q\pmod{q^2};$  

\med
$\be_4(1)=\chi_{.1^3}(1)-\be_3(1)=\chi_{.1^3}(1)-\chi_{1^2.1}(1)+\chi_{21.}(1)=q^{12}-(q^3/2)(q^2+q+1)(q^4+1)+q(q^4+1)\equiv q\pmod{q^2}$. \enp

\subsection{General case}
 
We first recall a result of \cite[Lemma 5.10]{TZ8}:

\bl{rc3}   Let $G=\Spin^-_{4m}(q)$, $m>1$, and let $g\in G$ be an \ir p-element. Let  $\phi\in\Irr _\ell (G)$. Suppose that $1<\deg\phi(g)< |g|$. 
Then  $p>2$, $ |g|=(q^{2m}+1)/(2,q-1)$, m is a $2$-power, $\deg\phi(g)= |g|-1$ and Sylow p-subgroups of $G$ are cyclic.\el

In this section we prove Theorem  \ref{mm2} for $G=\Spin^-_{2n}(q)$, and  
we improve the result of Lemma \ref{rc3} by excluding the option with $\deg\phi(g)= |g|-1$. 

\bl{dne} Let $G=\Spin_{2n}^-(q)$, $n\geq 4$,  and let $g\in G$ be a   
$p$-element acting irreducibly on V, the natural $\FF_qG$-module.
Let  $1_G\neq \phi\in\Irr_\ell G$, $\ell\neq p$. Suppose that  Sylow $p$-subgroups of $G$ are cyclic. Then $\deg\phi(g)\geq |g|-1$, and either $\deg\phi(g)=|g|$ or n is an odd prime power and $n<p$.\el

\bp  
As $g$ is irreducible, by \cite[Proposition 5.2]{TZ8} we have $p>2$ and $g^{q^n+1}=1$. Since  \syls are cyclic, $p$ is a primitive prime divisor of $q^{2n}-1$, whence 
$p>n$ and $p \nmid (q+1)$. Assume the contrary: $1 < \deg\phi(g) < |g|$.
Then we can assume that  $g$ is a generator of a \syl of $G$.

If $n$ is odd  then $g$ is contained in a subgroup $H\cong \SU_n(q)$ and $g$ is \ir on the natural module for $H$, see \cite[Proposition 5.2(iv)]{TZ8}.
So $\deg\phi(g)\geq |g|-1$ by Lemma \ref{kt5}  applied to an \ir \ccc of $\phi|_H$ of dimension greater than 1. 

Suppose that $n$ is even.  Then, by Lemma \ref{rc3},  
$|g|=(q^n+1)/(2,q+1)$ and $n$ is a 2-power. If $n=4$ then the result holds by Lemmas \ref{dn8} (for $q$ odd) and \ref{qe4} (for $q$ even). 

Let $n>4$.  By Huppert \cite[Hilfsatz 1, d) and Bemerkung prior to Satz 4]{Hu}, $Y:=\SO_8^-(q^{n/4})$ is isomorphic to a subgroup of 
$X=\SO^-_{2n}(q)$. Then $G$ contains a subgroup $H$ isomorphic to $\Spin^-_{8}(q^{n/4})/Z$, where $Z$ is a subgroup of $Z(G)$. 
By Sylow's theorem, we can assume that $g\in H$.
The result follows by applying the obtained result  for $H$  
to a non-trivial \ir \ccc of $\phi|_H$.
\enp

\bl{o37}   
Let  $G=\Spin_{7} (q)$, q odd, or $G= \Spin^+_8(q)$, and let $g\in G$ be a semisimple $p$-element whose order divides $q^2-q+1$, $p>3$. Let  $1_G\neq \phi\in\Irr_\ell G$  and $\ell\neq p$. Then Sylow p-subgroups of G are cyclic and $\deg\phi(g)=|g|$.\el

\bp In both the cases a \syl is contained in a subgroup $H$ of $G$ isomorphic to $\Spin^-_6(q)\cong \SU_4(q)$. So the result follows by applying Lemma \ref{uu4} to a non-trivial \ir \ccc of $\phi|_H$.\enp

\begin{propo}\label{zo1} Let $n\geq 3$ and $G\in\{\Spin_{2n}^- (q)$, $\Spin_{2n+1} (q)$, $ \Spin_{2n+2}^+ (q)\}$.  Let $1_G\neq \phi\in\Irr_\ell G$ and let 
$V$ be the natural $\FF_q G$-module.  
Let p be a prime such that a \syl of G is cyclic, $\ell\neq p$, and let $g\in G$ be  a p-element. 
Then either $\deg\phi(g)=|g|$, or $n>3$ is an odd  prime power, $p > n$, and $\deg\phi(g)\geq   |g|-1$.  Moreover,  if $\deg\phi(g)=|g|-1$ then one of the following holds: 

\begin{enumerate}[\rm(i)]
\item  g is \ir on  V and $G=\Spin_{2n}^- (q);$ 

\item $\dim C_V(g)=1$,  $G=\Spin_{2n+1} (q)$ and g is \ir on $C_V(g)^\perp;$

\item $\dim C_V(g)=2$,  $G=\Spin_{2n+2}^+ (q)$ and g is \ir on $C_V(g)^\perp$. 
\end{enumerate}
\end{propo}

\bp   Note that $p>3$ as \syls of $G$ are cyclic. 
Suppose the contrary that $\deg\phi(g) < |g|$. Then, as usual, we can assume that  $g$ is a generator of a \syl of $G$. 

If $g$ lies in a parabolic subgroup of $G$ then the result follows from \cite[Theorem 1.1]{DZ1} (the correction \cite[Theorem]{DZ1a} to \cite{DZ1} does not concern with these groups).  We will therefore assume that $g$ lies in no parabolic subgroup of $G$. Hence $C_V(g)$ contains no singular vector, which implies that $j:=\dim C_V(g)\leq 2$, and if  $j=2$ then $C_V(g)$ is non-degenerate 
of type $-$. 

\smallskip
(a) Suppose first that $C_V(g)=0$. Then, by Maschke's theorem,  $V$ is an orthogonal sum  of 
 $g$-stable subspaces, and   $g$ acts irreducibly on each of them. Each of them, say $W$,  is non-degenerate. (Indeed, the 
 radical of the associated bilinear form restricted to $W$ is $g$-invariant. So if it is zero, then it must be equal to $W$. Now the radical $R$ of the 
 associated quadratic form $Q$ restricted to $W$ has codimension $\leq 1$ and is $g$-invariant. As $g$ is not contained in any parabolic subgroup,
 $R=0$, and $\dim W=1$. As $|g|$ is odd, $W\subseteq C_V(g)$, a contradiction.)
Now, if $W\neq V$ then $V=W\oplus W^\perp$, $gW^\perp=W^\perp$, and $g$ acts nontrivially on both $W$ and $W^\perp$  which  implies that Sylow $p$-subgroups are not cyclic, a contradiction.
So $g$ acts irreducibly on $V$, whence $G=\Spin^-_{2n}(q)$ and $g^{q^n+1}=1$ by \cite[Proposition 5.2]{TZ8}.
Applying Lemma \ref{dne} we arrive at (i). 
 
\smallskip 
(b) Suppose that $C_V(g)\neq 0$, so that $j \in \{1,2\}$. Since $g$ is semisimple, we can write $V=C_V(g)\oplus V_1$ with $V_1 = C_V(g)^\perp$ being $g$-invariant and $C_{V_1}(g)=0$. As shown in (a), $V_1$ is non-degenerate, of type $-$ and even dimension $2m$, $g$ acts irreducibly on $V_1$, and $g^{q^m+1}=1$. In particular, $\dim V = 2m+j$, and so $m \geq 3$. As $C_V(g)$ is either of dimension $1$, or of type $-$ and dimension $2$, we have that $m=n$, and either $G=\Spin_{2n+1}(q)$ or 
$G= \Spin_{2n+2}^+(q)$

If $n=3$ then $G=\Spin_7(q)$ or $\Spin_8^+(q)$, and $\deg\phi(g)=|g|$ by Lemma \ref{o37}. Suppose  that $n \geq 4$, and let 
 $H:=\{x\in G \mid xv=v \mbox{ for every }v\in C_V(g)\}.$ 
 Then  $g\in H'$ as $Z(G)$ is a 2-group as well as  $H/H'$. Moreover, $H'\cong \Spin(V_1)/Z$ for some central subgroup $Z$ of $Z(G)$. 
 As $\deg\phi(g)<|g|$, $\deg\tau(g)<|g|$ for every non-trivial \ccc  $\tau$ of $\phi|_{H'}$. Applying 
Lemma \ref{dne}, we arrive at (ii), respectively (iii).
\enp

\section{Symplectic groups}

We recall the result of \cite[Proposition 5.7]{TZ8}:

\begin{propo}\label{tt4} Let $G=\Sp_{2n}(q)$, $n>1$,  and let $g\in G$ be an irreducible $p$-element. Let $\phi\in  \Irr_\ell G$   with $\dim\phi>1$ and $(\ell,q)=1$.  Suppose that $\deg\phi(g)<|g|$.
Then  $p>2$,  $|g|=o(g)=(q^n+1)/(2,q+1)$ and one of the \f holds:

\begin{enumerate}[\rm(i)]

\item q is odd, n is a $2$-power,  $ \deg\phi(g)=|g|-1$ and  Sylow subgroups of G are cyclic;  

\item $(n,q,|g|)=(3,2,9)$ and either $\dim\phi=\deg\phi(g)=7$ or $\ell=3$, $\dim\phi=21$ and $\deg\phi(g)\geq 7;$

\item  $(n,q,|g|)=(2,2,5)$, $\ell=3$ and $\deg\phi(g)=\dim\phi=4.$  
\end{enumerate}
\end{propo}

 Note that in case (ii) Sylow $p$-subgroups of $G$ are not cyclic. 
Below we  will eliminate case (i) in Proposition \ref{tt4}. To be precise, we prove

\begin{propo}\label{gu1} In  Proposition {\rm \ref{tt4}(i)}, we have $\dim\phi=(q^n-1)/2$ and $\deg\phi(g)=|g|-1$,  unless possibly $(G,|g|,\ell)=(\Sp_8(3),41,2)$,
 or $\ell=2$, $n>2$ and $q-1>2$ is a $2$-power.
\end{propo}

We first examine the \reps of  $\Sp_4(q)$. In this section we use notation of \cite{Sr} for conjugacy classes and ordinary \ir characters of $\Sp_4(q)$ unless otherwise explicitly stated.
 
 \begin{lemma}\label{p48}  Let $G=\Sp_{4}(q)$, $q$  odd,  and let $g\in G$ be an irreducible $p$-element. Let $\phi\in  \Irr_\ell G$   with $\dim\phi>1$ and $(\ell,pq)=1$.  Then $\deg\phi(g)=|g|$ unless $\dim\phi=(q^2-1)/2$ and $|g|=(q^2+1)/2$.\end{lemma}

\bp If $p=2$ then the result is contained in \cite[Prop. 5.7]{TZ8}. Suppose that $p>2$. Then $|g|$ divides $q^2+1$ and 
$g$ is contained in a subgroup $H\cong \SL_2(q^2)$. By \cite[Lemma 3.3(i)]{TZ8}, $|g|=(q^2+1)/2$ (so $(\ell,q^2+1)=1$) and $\deg\phi(g)\geq |g|-1$. Moreover, if $\tau$ is a non-trivial \ir \ccc of $\phi|_H$ then $\dim \tau=(q^2-1)/2$ and $1$ is not an \ei of  $\tau(g)$ as here $p$ is odd and $p\neq \ell$. If $1_H$ is a \ccc of $\phi|_H$ then we are done, otherwise 
$\dim\phi\equiv 0\pmod{(q^2-1)/2}$. Note that Sylow $p$-subgroups of $G$ are cyclic. If $\phi$ is liftable  then $\dim\phi=(q^2-1)/2$ by \cite[Theorem 1.1(5)]{Z3}.
Suppose that $\phi$ is not liftable. 

We use the results of White \cite{Wh1,Wh2,Wh3}, who obtained the decomposition numbers for $\Sp_4(q)$, $q$ odd.  This has been refined in \cite{OW1}.  

Let $\phi\in{\mathcal E}_s$, where $s\in G^*$ is a semisimple $\ell'$-element. By Lemma \ref{rm1}, we can assume that $|s|$ is coprime to $(q^2+1)/2$, so  $|s|$ divides $q^2-1$. (Indeed, the condition $(|s|,p)=p$ in Lemma \ref{rm1}(ii) implies $s^k$ to be of order $p$ for some $k$, and every element of order $p$ in $S^*$ is regular.) 

Suppose first that $\ell>2$. If $\ell$  divides $q-1$ then every \ir Brauer character is liftable,  \cite{Wh2} so we assume $\ell|(q+1)$.   By inspection in \cite{Wh2,Wh3}, 
we conclude that all non-liftable \ir characters are unipotent and lies in the principle block.   %CHECK \cite{Wh3} for blocks with cyclic defect. checked!
 By \cite{Wh2,OW1}, the  principal  $\ell$-block has 5 \ir Brauer characters $\phi_0=1_G$, $\psi,\psi_s,\psi_t,\psi_{st}$, where $\psi$ is liftable of degree $q(q^2-1)/2$, $\psi_s=\chi_s-1_G$, $\psi_t=\chi_t-1_G$,  $\psi_{st}=\chi_{st}-1_G-\psi_s-\psi_t -\al\psi$, where $\al=2$ unless $(q+1)_3=3$ and $\al=1$. Here $\chi_s,\chi_t,\chi_{st}$ are ordinary unipotent characters, their degrees are  
$\chi_{st}(1)=q^4$, $\chi_{s}(1)=q(q^2+1)/2=\chi_{t}(1)$, $\chi(1)=q(q-1)^2/2$. 
Therefore, 
$$\begin{aligned}\psi(1) & =q(q-1)^2,\\ 
    \psi_s(1)& =\psi_t(1)=q(q^2+1)/2-1=\frac{q(q^2-1}{2}+q-1,\\ 
    \psi_{st}(1) & =q^4-1-\al q(q-1)^2- q(q^2+1)\\
   & =q^4-1-q(q^2-1)-2q-\al q(q^2-1)/ 2+2\al q^2\\
   & =q^4-1-q(q^2-1)-\al q(q^2-1)+2\al (q^2-1)+2(\al -q).\end{aligned}$$  
It follows that $\psi_{st}(1)$ is not a multiple of $(q^2-1)/2$, a contradiction.

Let $\ell=2$. Then there are at most 3 \ir Brauer characters that are not liftable; these are denoted by $\psi_1,\psi_2,\psi_3$, see   \cite[p. 711]{OW1}. For them, in notation of \cite{OW1}, we have 
$$\psi_3=\Phi_5-2\cdot 1_G-\psi_5,~\psi_1=\Phi_3-\phi_4-x\phi_6,~\psi_2=\Phi_4-\phi_5-x\phi_6,$$ 
where $0\leq x\leq (q-1)/2$ (and $1\leq x\leq (q-1)/2$
if $4|(q-3)$. In addition, 
$$\psi_2(1)=\psi_1(1)=\frac{(q-1)^2(q^2+q(1-x)+1)}{2}$$ 
\cite[p. 711]{OW1}, and 
$$\psi_3(1)= \frac{(q^2+1)(q+1)}{2}  -2 -\frac{q^2-1}{2}=\frac{(q^2-1)(q+1)}{2}+q-1-\frac{q^2-1}{2}=\frac{q(q^2-1)}{2}+q-1.$$
So $\phi_3(1)$ is not a multiple of $\frac{q^2-1}{2}$. Clearly, $\frac{(q-1)^2(q^2+q(1-x)+1)}{2}$ is a multiple of  $\frac{q^2-1}{2}$ \ii $ (q-1)(q^2+q(1-x)+1)$ is a multiple of $q+1$.
As $(q-1)(q^2+q(1-x)+1)  \equiv -2(1+x)\pmod{q+1}$, 
$\psi_1(1)\equiv 0\pmod {(q^2-1)/2}$ \ii $1+x$ is a multiple of $(q+1)/2$. As $0\leq x\leq(q-1)/2$,
this implies $x=(q-1)/2$. In this case $\psi_{1,2}(g)=- \psi_{1,2}(1)/\tau(1)=-(q-1)(q+2)/2$. Note that $g$ lies in a class $B_1(i)$ for some $i$ in notation of \cite{Wh1} and \cite{Sr}.  On the other hand, we have $\Phi_3(g)=\Phi_4(g)=0$,  $\psi_4(g)=\psi_5(g)=-1$ and $\psi_6(g)=\theta_{10}(g)=1$ (note that $\psi_5(1)=\psi_6(1)=(q^2-1)/2$ and hence $\psi_5(g)=\psi_6(g)=-1$). 
Therefore, 
$$\psi_1(g)=\Phi_3(g)-\phi_4(g)-x\phi_6(g)=0+1-x=1+\frac{q-1}{2}=\frac{q+1}{2},$$ a contradiction. 

Let $q=3.$ Then $\ell=2,5$. As $p=5$ here, we have $\ell\neq 5$ by assumption.  If $\ell=2$ then $\phi(Z(G))=1$ and $\phi $ can be viewed as a \rep of  $X=\PSp_4(3)\cong \SU_4(2)$. Then the result  follows from the Brauer character table \cite{JLPW}. \enp

\bl{ww2}  Let  $G=\Sp_4(q)$, $q$ odd,   $H=\Sp_2(q^2)$ and let $1_G\neq \be\in\Irr_\ell(G)$ be the Brauer character of G, $(\ell,2q)=1$.   \st   
the \ir constituents of $\be|_{H}$ are   of degrees    $(q^2\pm 1)/2$. Then $ \be(1)=(q^2\pm 1)/2$, and so $\beta$ is a Weil character.  \el
 
\bp  Let $Z=Z(G)=Z(H)$ and let $\be|_{Z}=\zeta\cdot\be(1)$ for some $\zeta\in\Irr_\ell(Z)$. 
Then $\tau|_Z=\zeta\cdot\tau(1)$ for every \ir \ccc $\tau$ of $\be|_H$. Note that that \ir $\ell$-Brauer characters of $H$ of degree $(q^2-1)/2$ and of degree $(q^2+1)/2$ have different restriction to $Z(G)$. This implies  $\tau(1)\in\{(q^2\pm 1)/2\}$ is the same for all $\tau$. In addition,  $\zeta=1_Z$ \ii $\tau(1)$ is odd. Therefore,  $\be|_H=d\cdot \tau$, in particular,
$\be(1)=d\cdot (q^2-1)/2$ if $\zeta\neq 1_Z$, otherwise  $\be(1)=d\cdot (q^2+1)/2$ for some integer $d$.  We also observe that, since $\ell \neq 2$,  for each $\eta = \pm 1$, $G$ and $H$ each have 2 distinct \ir $\ell$-Brauer characters of this degree 
$(q^2-\eta)/2$, and they coincide at semisimple $\ell'$-elements.
 
Let $U$ be a maximal unipotent subgroup of $H$. Then
  $U=\{1\}\cup U_1\cup U_2$, where the elements of $U_j$ are conjugate in $H$ and $|U_1|=|U_2|=(q^2-1)/2$. Let $u_1\in U_1$, $u_2\in U_2$, and let $\tau_1,\tau_2$ be distinct irreducible $\ell$-Brauer characters of $H$ of degree $(q^2-\eta)/2$ with $\eta =\pm 1$. Then $\tau_1(u)\neq \tau_2(u)$ for $1\neq u\in U$, however, $\tau_1(u_1)+\tau_1(u_2)=\tau_2(u_1)+\tau_2(u_2)=-\eta$.
This follows from the character table of $H$ (by reordering $U_1,U_2$ we can assume that $\tau_j(u_1)=(-\eta +(-1)^j)q)/2$ and $\tau_j(u_2)=(-\eta -(-1)^j)q)/2$). This allows us to ignore the difference between $\tau_1,\tau_2$ in some computations below.

Observe that $u_1,u_2$ are not conjugate in $G$ and lie in classes $A_{31}$, $A_{32}$, respectively. (Indeed, if $\phi$ is an ordinary character of $G$ of degree $(q^2-\eta)/2$ then  $\phi|_H$ is \ir by dimension reason; as $\tau(u_1)\neq \tau(u_2)$ also for $\ell=0$ then   $\phi(u_1)\neq \phi(u_2)$, whence the claim.) \itf 
 \begin{equation}\label{eq-b10}
   \be(u_1)+\be(u_2)=-d\eta=-2\be(1)\eta/(q^2-\eta).
\end{equation}   

\smallskip
(a) Suppose first that $\be$ is liftable, say to $\chi \in \Irr(G)$.  By inspection of the character table of $G$ in \cite{Sr}, we observe that \eqref{eq-b10} implies 
\begin{equation}\label{eq-b12}
\chi\in\{\chi_6(k),\chi_8(k),\xi_1(k),\xi_3(k),\theta_3,\theta_4,\theta_7,\theta_{8},\theta_{11},\theta_{12},  \Phi_1,\Phi_2,\Phi_5,\Phi_6,\Phi_9\}
\end{equation} 
in the notation of \cite{Sr}. Observe that $\dim \theta_i\in\{(q^2\pm 1)/2\}$ for $i=3,4,7,8$ so we have to rule out only the remaining characters. Again using \eqref{eq-b10},
we see that they correspond to $\eta=-1$. In other words, either $\eta=-1$ or $\be$ is not liftable.

\med
 (b) Suppose that $\eta=1$. Then $\tau(u_1)+\tau(u_2)=-1$,  
so 
\begin{equation}\label{eq-b11}
\be(u_1)+\be(u_2)=d\be(1)/\tau(1) 
\end{equation}
by \eqref{eq-b10}. In addition, $\be(Z(G))\neq \Id$, so   $\be$ is not in the principle $\ell$-block. By (a),   $\be$ is not liftable. A priory, $\ell$ divides either $q^2+1$ or $q+1$ or $q-1$.  If $\ell \mid (q-1)$, then all \ir Brauer characters are liftable \cite{Wh3}.
If $\ell|(q+1)$, then all non-liftable characters are in the principle block \cite{Wh2,Wh3}.
So we may assume that $\ell|(q^2+1)$. By \cite[p.650]{Wh3},
 there are two non-liftable characters that are not in the principle block. These are $\phi_5=\theta_5^\circ-\theta_7^\circ$ and 
$\phi_{6}=\theta_{6}^\circ-\theta_{8}^\circ$, where $\theta_7(1)=\theta_{8}(1)=(q^2-1)/2$ and $\theta_{5}(1)=\theta_6(1)=q^2(q^2-1)/2$, see \cite[p. 652]{Wh3}.
By \cite{Sr}, $\theta_5(u_i)=\theta_6(u_i)=0$ for $i=1,2$, so $\be(u_1)+\be(u_2)=-(\tau(u_1)+\tau(u_2))=1$ and hence \eqref{eq-b11} implies $d=1$, as required.

\smallskip 

\med 
(c) Suppose that $\eta=-1$.  Then $\be(1)=d(q^2+1)/2$  
 and $Z(G)$ is in the kernel of $\be$. Note that $\tau_j(u_1)+\tau_j(u_2)=-1$ and 
$\tau_1(u_j)+\tau_2(u_j)=-1$ for $j=1,2$. 
 
Let $h\in H$ be a 2-element of maximal order. As $2|q$, $|h| \geq 8$.
Observe that $H\cong \SL_2(q^2)$ is completely reducible in $\Sp_{4}(\overline{\FF_q})$, and the elements of $H$ are conjugate to
 $\diag(X,X^{(q)})$, where $X\in \Sp_2(q^2)$, and $X^{(q)}$ is obtained from $X$ by applying the  \au  
$x\mapsto x^q$ to every matrix entry $x$ of $X$. We may embed $h$ into a cyclic maximal torus $T_1$ of $H$ of order $q^2-1$, and hence 
$h$ realizes as $\diag(x,x\up,x^q,x^{-q})$ for $x\in\FF_{q^2}^\times$ of order $(q^2-1)_2$.  In particular, $x \notin \FF_q$ and $x \neq x^{\pm q}$, $x^{-1}$.
Therefore $h$ lies in the class $B_2(i)$, 
and $\tau(h) \in \{\pm1\}$ when $\tau(1)=(q^2+1)/2$. As $\be|_H=d\tau$, we have 
$$\be(h)=d\tau(h)\in\{\pm d\}.$$ 

\med 
 Suppose first that $\be$ is liftable. By (a), we have to inspect the characters listed in \eqref{eq-b12}.
 As $\Phi_i(h)=0$ for $i=1\ld 9$, and similarly $\xi_1(k)$ and $\xi_3(k)$ vanish at $h$,    
these characters are ruled out.  In the notation of \cite{Sr},
$\chi_6(k)(h)=\be_{ik}$,  $\chi_8(k)(h)=\al_{ik}$ and $d=2q-2,2q+2$, respectively. As
the absolute values of $\be_{ik},\al_{ik}$ do not exceed 2 \cite[p. 516]{Sr}, these two characters are ruled out too. Similarly, $\theta_{11}(h)=1, \theta_{12}(h)=-1$ and $d=q$; hence $\chi\neq \theta_{11},\theta_{12} $. 

\smallskip
Next suppose that $\be$  is not liftable. Then $(\ell,q-1)=1$ by \cite{Wh2}. 

Suppose first that $\ell|(q^2+1)$. Then there are two non-liftable characters, both in the principle block,  see \cite[p. 652]{Wh3}. These are
$\phi_{9}=\theta_{9}-1_G$, $\phi_{13}=\theta_{13}-\phi_9=\theta_{13}-\theta_{9}+1_G$. Then  $2\phi_9(1)=q(1+q)^2-2\equiv -4\pmod{q^2+1}$, and 
$2\phi_{13}(1)=2\theta_{13}(1)-2\phi_9(1) \equiv 6 \pmod{q^2+1}$. In either case, the degree is not divisible by $\tau(1)=(q^2+1)/2$.
 
Suppose that $\ell|(q+1)$. Then non-liftable characters are all unipotent \cite{Wh3}. As mentioned in the proof of Lemma \ref{p48}, there are 
5 unipotent Brauer characters $1_G$, $\eta^\circ$, $\psi_s,\psi_t,\psi_{st}$,
whose degrees are  $1$, $\psi(1)=q(q-1)^2$, $\psi_s(1)=\psi_t(1)=q(q^2+1)/2-1 $, and 
$$\psi_{st}(1)  =q^4-1-\al q(q-1)^2- q(q^2+1)\equiv -\al q(-2q)\equiv   2\al \pmod{(q^2+1)/2}.$$ 
As $\al =1$ or $2$, none of these degrees  is  a multiple of $(q^2+1)/2$, again a contradiction. \enp

\bl{2ww2}  Let  $G=\Sp_4(q)$, $q > 3$ odd,   $H=\Sp_2(q^2)$ and let $1_G\neq \be\in\Irr_\ell(G)$ be a Brauer character of G for $\ell=2$.    \st all the \ir constituents of $\be|_{H}$ are   of degrees   $1$ or $(q^2- 1)/2$. Assume in addition that $q-1$ is not a $2$-power. Then $ \be$ is of  degree    $(q^2- 1)/2$, and so is a Weil character.\el

\bp  Let  $c,d$ be the number of trivial and non-trivial \ir constituents of $\phi|_{H}$, respectively. Then $\be(1)=c+d(q^2-1)/2$. We can ignore the cases with $\be(1)=(q^2-1)/2$. 
  
Let $\nu\in\Irr_2(H)$ with $\nu(1)=(q^2-1)/2$. Let $g\in H$ with  
$$|g|=(q^2+1)/2;$$ this is odd as $4|(q^2-1)$.   Since   $\nu(g)=-1$, we observe that $\be(g)=c-d$. This implies $\be(1)+\be(g)(q^2-1)/2=c(q^2+1)/2$.  Note that $g$ lies in class $B_1(i)$ in the notation of \cite{Sr}.

Let $ U$ be a maximal unipotent subgroup of $H$, and let $1\neq u_1,u_2\in U$ be non-conjugate elements of $H$. Then $(\nu|_U,1_U)=0$, and hence  $(\be|_U,1_U)=c$.  
It follows that $$q^2(\beta|_U,1_U)=cq^2= \be(1)+(\be(u_1)+\beta(u_2))(q^2-1)/2=c+(d+\be(u_1)+\be(u_2))(q^2-1)/2,$$ whence  $c=-\be(g)+\be(u_1)+\be(u_2)$.  
 
This implies 
\begin{equation}\label{eq9}\be(1)=(\be(u_1)+\be(u_2) )(q^2+1)/2 -\be(g)q^2\end{equation}
Indeed, $\be(1)=c+d(q^2-1)/2=c+(c-\be(g))(q^2-1)/2=c(q^2+1)/2-\be(g)(q^2-1)/2=\\ \smallskip =(\be(u_1)+\be(u_2)-\be(g))(q^2+1)/2 -\be(g)(q^2-1)/2=(\be(u_1)+\be(u_2) )(q^2+1)/2 -\be(g)q^2.$ \med

\medskip 
(a) Suppose first that  $\be$ is liftable, say to $\chi \in \Irr(G)$. 
By inspection in \cite{Sr}, one observes that $\chi(1)$ is a multiple of $(q^2+1)/2$ \ii $\chi(g)=0$. 

Suppose that $\chi(g)\neq 0$. Then $\chi\in\{\chi_1(j),\theta_i, i=5,6,7,8,9,10,13\}$
 in notation of \cite{Sr}. 

We have seen in the proof of Lemma \ref{ww2} that $u_1,u_2$ are not conjugate in $G$ and lie in the classes $A_{31}$, $A_{32}$ (up to ordering of $u_1,u_2$). Observe that  $q^2|\chi(1)$ \ii $\chi(u_1)=\chi(u_2)=0$. So $q^2|\chi(1)$
implies $\chi(1)=-\chi(g)q^2$ by (\ref{eq9}). Inspection in  \cite{Sr}   rules  out the cases with $\chi\in\{  \theta_5,\theta_6,\theta_{13}\}$. 

 If $(q^2-1)/2|\chi(1)$ then, again by  Lemma \ref{eq9},  $-\chi(g)+(\chi(u_1)+\chi(u_2))$
is a multiple of $(q^2-1)/2$. Inspection in  \cite{Sr} shows that this is false (unless
$\chi(1)=(q^2- 1)/2$). This rules out the cases with $\chi\in\{\chi_2(j),
\chi_5(k,l),\xi_{21}(k), \xi_{41}'(k)\}$. Let  $\chi=\chi_1(j)$. As $\chi_1(j)=(q^2-1)^2$ and $(\chi(u_1)+\chi(u_2))/2=1$ by \cite{Sr}, we get from \ref{eq9} 
$-\chi(g)=((q^2-1)^2-q^2-1))/q^2=q^2-3$. As $\chi(g)$ is a sum of four $(q^2+1)$-roots of unity, it follows that $|-\chi(g)|\leq 4$, a contradiction. 

Next suppose  that $\chi(g)=0$. Then $\chi(u_1)+\chi(u_2)=2\chi(1)/(q^2+1)$.   Inspection in  \cite{Sr} gives a contradiction for    
$\chi\in\{\chi_3(k,l),\chi_4(k,l),\chi_5(k,l),\chi_7(k)$, $\chi_8(k),\chi_9(k),\xi_1'(k), \xi_3'(k), \xi'_{21}(k), \xi'_{22}(k)$, \\ $\xi_{41}(k),  \xi_{42}(k)$, $\Phi_3,\Phi_4,\Phi_7,\Phi_8,   \theta_1,\theta_5,\theta_6\}$. 
 
We are left with $\chi\in\{  \chi_6(k),\xi_1(k),\xi_3(k),\Phi_1,\Phi_2,\Phi_5,\Phi_6,\Phi_9,\theta_j, j=9,10,11,12\}$. However, the decomposition matrix of 
$G$ modulo 2 \cite[p. 710]{Wh1} tells us that the characters $\Phi_1,\Phi_2,\Phi_5,\Phi_6,\Phi_9,\theta_j, j=9,11,12\}$ are reducible modulo 2, so these
can be ignored. This leaves us with $\chi\in \{\chi_6(k), \xi_1(k),\xi_3(k)$, $\theta_{10} \}$.

Next we choose some non-trivial element $h\in H$ of odd  order dividing $q^2-1$. 
If $q+1$ is a $2$-power, then, since $q > 3$, $(q-1)/2 > 1$ is odd, and we choose $h$ of 
order $(q-1)/2$ in the class $B_8(2)$  in the notation of \cite{Sr}.
If $q+1$ is not a $2$-power, and $q-1$ is not a $2$-power, then we can choose $h$ of order 
$j:=(q^2-1)_{2'}$ in the class $B_2((q^2-1)/j)$ in the notation of \cite{Sr}. 
Observe that $\nu(h)=0$ as $\nu$ is liftable and  $|h|$ divides $(q^2-1)/2=\nu(1)$. Therefore, $\chi(h)=c=\chi(u_1)+\chi(u_2)-\chi(g)$, whence 
$$\chi(h)+\chi(g)=\chi(u_1)+\chi(u_2).$$ 

If $\chi=\chi_6(k)$ then $\chi(g)=0$,  $|\chi(u_1)+\chi(u_2)|=|2-2q| =2q-2$ and $|\chi(h)| \leq q-1$, a contradiction.   

If $\chi=\xi_1(k)$  
then again $|\chi(g)+\chi(h)| \leq q-1$ and  $\chi(u_1)+\chi(u_2)=2-2q$, a contradiction.

If $\chi=\xi_3(k)$  
then $\chi(g)=0$, $\chi(h)=0$ or $(1+q)\al_{2k}$, and  $\chi(u_1)+\chi(u_2)=2+2q$. As $1 \leq k \leq (q-3)/2$ and $\al_{2k}=\al^{2k}+\al^{-2k}$, where 
$\alpha \in \CC^\times$ has order $q-1$, we have $|\al_{2k}| < 2$, we have a contradiction.

If $\chi=\theta_{10}$  
then $\chi(g)=1$, $ \chi(h)= 0$, and  $\chi(u_1)+\chi(u_2)=q$, ruling out this case. 

(If $q-1$ is a $2$-power, when any element $h' \neq 1$ of odd order dividing $q^2-1$ satisfies $(h')^{q+1}=1$, and hence belong to class $B_6(i)$, with $\chi(h')=q-1$. In fact, this   leads to the conclusion that condition "$q-1$ is not a 2-power" cannot be dropped.)
 
\medskip 
(c) Suppose now that $\phi$ is not liftable. By \cite[p. 711]{Wh3}, there are three non-liftable \ir Brauer characters $\phi_1,\phi_2,\phi_3$ and they are   in the principal block. (There is a misprint in writing the decomposition number row of  the character $\xi'_{41}$ \cite[p. 710]{Wh3}, in fact, $\xi'_{41}$ is \ir modulo 2.)
 
Let $\tau\neq \tau'$ be the two \ir Brauer characters of $H$ of degree $(q^2-1)/2$. 
 Then $\tau(h)=\tau'(h)=0$ for $h$ as in (b),  $\tau(g)=\tau'(g)=-1$, and we have $\be(h)=c$, $\be(g)=c-d$, whence $d=\be(h)-\be(g)$.  In addition, $\be(1)=c+d(q^2-1)/2=\be(h)+(\be(g)-\be(h))(q^2-1)/2$. We show that this fails.
  
For this we compute and collect in Table 2, for $\be\in\{\phi_1,\phi_2,\phi_3\}$,  the values $ \be(1),\be(g), \be(h)$ using the character table of $G$ in \cite{Sr} and the decomposition matrix of $G$ modulo 2 provided in \cite{Wh1}.
  Then we have $\phi_i(1)=\phi_i(h)+(\phi_i(g)-\phi_i(h))(q^2-1)/2$. If $i=1,2$ then $\phi_i(h)=0$, whence $\phi_1(1)=\phi_i(g)(q^2-1)/2=(1-x)(q^2-1)/2$, where $x\in\{1,2\}$ by \cite{OW1}. This is a contradiction. 
Let $i=3$. Then $\phi_3(1)=\phi_i(h)+(\phi_i(g)-\phi_i(h))(q^2-1)/2=-2+\frac{q^2-1}{2}$, a contradiction again (as it is well known that $\be(1)\geq (q^2-1)/2$ unless $\be(1)=1$.)  \enp
 
\begin{remar}\label{fermat} 
{\em If $q=3$ then Lemma  \ref{2ww2} is not valid; in fact, this fails if $\be(1)=6$ or $14$.  More generally, we show that if $q-1$ is a $2$-power then 
the restriction of the character $\theta_{10}$ of degree $q(q-1)^2/2$ of $\Sp_4(q)$ 
to $2'$-elements  satisfies the assumptions of Lemma \ref{2ww2}.
Indeed, the assumption on $q$ implies that the nontrivial $2'$-elements in $H=\Sp_2(q^2)$, embedded in $\Sp_4(q)$ have order dividing $q$, $(q+1)/2$ (if $q>3$), or $(q^2+1)/2$, and hence belong to classes $A_{31}$ or $A_{32}$ in the first case, $B_6(i)$ in the second case, and $B_1(i)$ in the third case, in the notation of \cite{Sr}. Let $\eta_{1,2}$ be the two Weil characters of degree $(q^2-1)/2$ of $H$, labeled in such a way that $\eta_1(u_1) = (-1+q)/2$, $\eta_1(u_2)=(-1-q)/2$ for elements $u_1 \in H$ belonging to class $A_{31}$
and $u_2 \in H$ belonging to class $A_{32}$. Now $\theta_{10}(u_1)= 0$, $\theta_{10}(u_2)=q$, and using \cite{Sr} we can check that $\theta_{10}$ agrees with $(q-1)\cdot 1_H + (q-3)/2 \cdot \eta_1 + (q-1)/2 \cdot \eta_2$ on all $2'$-elements of $H$.}\end{remar}
 
\bl{sp63} Embed $H:= \Sp_2(27)$ in $G:=\Sp_6(3)$. Suppose that $\phi \in \Irr_2(G)$ is such that all irreducible constituents of $\phi|_H$ are of degree $1$ or $13$. Then $\phi$ is either trivial or a Weil character.\el

\bp This is proved using $\mathsf{GAP}$ to restrict 
Brauer characters of $\Sp_6(3)$ to a maximal subgroup $\SL_2(27) \rtimes C_3$.\enp

\vskip10pt 
 \begin{center}{Table 2: Some Brauer character values for $\ell=2$  }

 \vspace{10pt}
\begin{tabular}{|c|c|c|c|c| }
\hline
character& $\ell$ & degree    & $g$ & $h$    \\
\hline
 $\Phi_3$& $0$ & $q (q-1)(q^2+1)/2$   & $0$ & $0$  \\
\hline $\Phi_4$& $0$ & $q (q-1)(q^2+1)/2 $   & $0$ & $0$  
\\ \hline
$\Phi_{5}$& $0$ & $ (q^2+1)(q+1)/2 $   & $ 0$ &$ 0$    \\ 
\hline
$\psi_1,\psi_2$&$2$   & $\frac{(q-1)^2(q^2+q(1-x)+1)}{2}$&$1-x$ & $0$  
\\ \hline
%$$& $2$   &$\frac{(q-1)^2(q^2+q(1-x)+1)}{2}$ & $1-x $   & $0$   \\
 %\hline
$\psi_{3}$& $2$   &$\frac{q(q^2-1)}{2}+q-1$ &  $-1$ & $-2$  \\
  \hline
$\theta_7,\psi_4$& $0,2$   &$ (q^2-1)/2$ &  $-1$ &$0$   \\
  \hline
$\theta_8,\psi_5$& $0,2$   &$(q^2-1)/2$ &  $-1$ &$0$   \\
  \hline
$\theta_{10},\psi_6$& $0,2$   &$q(q-1)^2/2$ &  $1$ &$0$   \\
  \hline
\end{tabular}
\end{center}

 \medskip
\bl{w4w}   Let  $G=\Sp_{2n}(q)$, $n=2^m \geq 2$, and  $L \cong\Sp_4(q^{n/2})\leq G$. Let $\phi\in\Irr_\ell(G)$ with $\ell\nmid q$. 
Suppose  that all non-trivial irreducible constituents of $\phi|_L$ are Weil. 
If $\ell=2$, assume in addition that $n \neq 4$ if $q=3$, and that $q-1$ is not a $2$-power if $q>3$. Then $\phi$ is Weil. \el

\bp   We argue by induction on $m\geq 1$, with the induction base $m=1$ being obvious.

\smallskip
(a) Here we consider the case where either $\ell \neq 2$, or $\ell=2$ but $q-1$ is not a $2$-power. For the induction step $m \geq  2$, set 
$$H:=\Sp_n(q^2)\subset G,~H_1:=\diag(\Sp_2(q^2), \Id_{n-2})\subset H,~H_2:=\diag(\Sp_4(q),\Id_{2n-4})\subset G.$$ 
Note that $L = \Sp_4(q^{n/2})$ can be embedded in $H$.
By the induction hypothesis  applied to $H$, all non-trivial \ir constituents $\tau$ of $\phi|_H$ are Weil. It is  well known that, in this case,  all non-trivial \ir constituents of $ \tau|_{H_1}$ are Weil too, see for instance \cite[Theorem 2]{Z85} or \cite[Proposition 2.2(vi)]{TZ2}.  
As $H_1\subset H_2$, we can apply Lemmas \ref{ww2} and \ref{2ww2} to deduce that  all non trivial \ir constituents of $\phi|_{H_2}$ are Weil. Since $H_2$ is a standard subgroup of $G$, applying \cite[Theorem 2.3]{GMST} we  conclude that $\phi$ is Weil.

\smallskip
(b) Now we consider the case $(\ell,q) = (2,3)$, and first work on the new induction base $m = 3$, so that $G = \Sp_{16}(3)$.  Set
$$H:=\Sp_8(9)\subset G,~H_1:=\diag(\Sp_6(9), \Id_{2})\subset H,~H_2:=\diag(\Sp_2(3^6),\Id_{2})\subset H_1.$$ As in (a), we may assume that $L=\Sp_4(3^4) \subset  H$,
and so all  non-trivial \ir  constituents of $\phi|_L$ are Weil. By part (a) applied to $H$, all  non-trivial \ir constituents $\tau$ of $\phi|_H$ are Weil. In this case,  
all \ir constituents of $\phi|_{H_1}$ are Weil too \cite{TZ2}. In turn, this implies  all \ir constituents  of $\phi|_{H_2}$ to be Weil. As a subgroup  of $G$, we can write $H_2=\diag(\Sp_2(3^6),\Id_4)$, which then embeds into $H_3:=\diag(\Sp_4(27),\Id_4)$. 
Now we can apply Lemma \ref{2ww2} to deduce that  all  non-trivial \ir constituents of $\phi|_{H_3}$ are Weil. Embedding $H_4:= \diag(\Sp_2(27),\Id_{10})$ into $H_3$, we see 
that all  non-trivial  \ir constituents of $\phi|_{H_4}$ are Weil. Now we can embed $H_4$ in a standard subgroup $H_5:= \diag(\Sp_6(3),\Id_{10})$ of
$G$. By Lemma \ref{sp63}, all  non-trivial \ir constituents of $\phi|_{H_5}$ are Weil. Applying \cite[Theorem 2.3]{GMST} to $H_5 < G$, we conclude that $\phi$ is Weil.

For the induction step $m \geq 4$, set
$$H:=\Sp_n(9)\subset G,~H_1:=\diag(\Sp_8(9), \Id_{n-8})\subset H,~H_2:=\diag(\Sp_{16}(3),\Id_{2n-16})\subset G.$$ 
Again, $L = \Sp_4(3^{n/2})$ can be embedded into $H$, and so all \ir constituents of $\phi|_L$ are Weil. By part (a) applied to $H$, all  non-trivial  \ir constituents $\tau$ of $\phi|_H$ are Weil. In this case,  all  non-trivial \ir constituents of $\tau|_{H_1}$ are Weil too. This implies that all  non-trivial \ir constituents of $\phi$ restricted to  a standard subgroup $\Sp_4(3^4)$ of $H_1$ are Weil.
As $H_1\subset  H_2$, we can apply the aforementioned induction base $m=3$ to $H_2$ to deduce that  all  non-trivial  \ir constituents of $\phi|_{H_2}$ are Weil. Since  $H_2$ is a standard subgroup of $G$, we conclude by \cite[Theorem 2.3]{GMST} that $\phi$ is Weil.\enp

\bp[Proof of Proposition {\rm \ref{gu1}}] Here $q$ is odd, $p>2$, $n=2^m$ for some integer $m>0$ and $|g|=(q^n+1)/2$. It is well known that the minimal dimension of a non-trivial \irr of $G$ equals $(q^n-1)/2$. The case $m=1$ follows from Lemma \ref{p48}. In general, we may assume that $g$ lies in a standard subgroup $L:= \Sp_4(q^{n/2})$, and so 
$\deg\tau(g) < |g|$ for all irreducible constituents $\tau$ of $\phi|_L$. By Lemma \ref{p48}, any such $\tau$ is Weil or $1_L$. By Lemma \ref{w4w}, $\phi$ is Weil.
\enp

\bl{rgs} For $q$ odd let $G= \GL_{2n}(q)$, $n>1$, and $H=\Sp_{2n}(q) \subset \SL_{2n}(q)$. Let $1_G\neq \phi\in\Irr_\ell(G)$. Then $\phi|_H$ contains a non-trivial \ir constituent which is not Weil.\el

\bp Note that  \ir Weil \reps of $H_n=\Sp_{2n}(q)$ are of two types, $A,B$, say, and those of the same degree are obtained from each other by a twist
by an outer \au of  $\Sp_{2n}(q)$, in fact, this is an inner \au of $\mathrm{CSp}_{2n}(q)$, the conformal symplectic group.  The restriction  of a Weil \rep to the natural subgroups  regards the types, that is, there are embeddings 
$H_1\ra H_2\ra \cdots\ra H_n$ such that the \ir constituents of the restriction of a Weil \rep $\rho_n$ of type $A$ to $H_i$ are of type $A$ for $i=1\ld n-1$. (See also \cite{TZ2}).
Note that the restriction of 
$\rho_n$ to $H_k\times H_{n-k}$ are known \cite[Theorem 2]{Z85}, in particular, for $\ell\neq 2$ these are tensor products of \ir Weil \reps of    $H_k$ and $ H_{n-k}$, both of type $A$,  with some refinement for $\ell=2$. This implies that  the restriction of 
$\rho_n$ to $H_k\times H_{n-k}$ has no \ir constituent that is a tensor product 
of Weil \reps of $H_k,H_{n-k}$ of different types.  Let $\tau=\tau_k\otimes \tau_{n-k}$
be an \ir constituent in question, where $\tau_k, \tau_{n-k}$ are non-trivial 
 \ir Weil \reps of of $H_k,H_{n-k}$, respectively. Note that  $\GL_{2n}(q)$ contains a subgroup $X \cong \mathrm{CSp}_{2k}(q)\times \mathrm{CSp}_{2(n-k)}(q)$,
so one can twist the \rep $\tau_k\otimes \tau_{n-k}$ by an element of $x\in X$
which acts trivially on $H_k$ and as an outer \au on of $H_{n-k}$. Then
$\tau^x=\tau_k\otimes \tau^x_{n-k}$, where the type of $\tau^x_{n-k}$ is distinct from the type of $\tau_{k}$, which is a contradiction.
(One can also work with $\rho_n$ to $Y=Y_1 \times Y_2$, where $Y_1\in H_k, Y_2\in H_{n-k}$ are long root subgroups, see \cite[Lemma 5.4]{GMST}.)\enp

\begin{corol}\label{ss2} In Proposition {\rm \ref{zu3}} case {\rm (iii)} can be omitted. That is, in the assumptions of Proposition {\rm \ref{zu3}}, 
one of {\rm (i)}, {\rm (ii)}, or {\rm (iv)} holds.\end{corol}
 
\bp  Recall that $q$ is odd here. Suppose first that $n\geq 8$. Then we may choose $g\in H\cong \SO^-_n(q)$ and  Proposition \ref{zo1} implies  $2 \nmid n/2$. This is a contradiction as $n> 3$ is a 2-power by assumption.

  So we are left with the case where $n=4$.  Then $g$ is contained in a subgroup $H \cong \Sp_4(q)$. By Lemma \ref{p48}, every non-trivial \ir constituent of $\phi|_H$ is of dimension $(q^2-1)/2$, and $|g|=(q^2+1)/2$. As observed in the proof of Lemma \ref{p48}, 1 is not an \ei of $\phi(g)$ and $\deg\phi(g)=|g|-1$. \itf $1_H$ is not a constituent of $\phi|_H$. Hence we arrive at a contradiction by Lemma \ref{rgs}.\enp

%\newpage
\section{Exceptional groups}

In view of our paper \cite{TZ20} we do not need to consider
the groups 
$$G\in\{ {}^2 B_2(q),{}^2 G_2(q), {}^3 D_4(q),\\ {}^2 F_2(q), G _2(q)\}.$$

For any positive integer $n$  let $\Phi_n(q)$ denote the $n^{\mathrm {th}}$ cyclotomic polynomial evaluated at $q$.

\bl{ee4} Let $G\in\{F_4(q),E_6(q),{}^2E_6(q),E_7(q),E_8(q)\}.$ Let $p$ be a prime such that 
Sylow $p$-subgroups of $G$ are cyclic. Then $p|\Phi_i(q)$, where $i$ is as in the second line of the \f table:\el
 
\vspace{10pt}
\begin{center}

Table 3:  
Cyclic Sylow $p$-subgroups in some exceptional groups of Lie type 

%\small{

\bigskip
\begin{tabular}{|l|c| c|c| c|c|}

        \hline
         $~~~~~~~~~~G$&$F_4(q)$&$E_6(q)$&${}^2E_6(q)$&$E_7(q)$&$E_8(q)$  \\

  \hline
$~~i$  &$8,12$&5,8,9,12& 8,10,12,18&5,7,8,9,10,12,14,18&7,9,14,15,18,20,24,30 \cr
        \hline
     
\end{tabular}
\end{center}

\bigskip

\begin{lemma}\label{sr1}   
Let $ G$ be the exceptional Lie-type group   $F_4(q)$. 
Let $S_p$ be a cyclic \syl of $G$. Then   every $1\neq g\in S_p$ is regular.
\end{lemma}

\bp \cite[Lemma 4.3]{Z3} claims that $C_G(s)$ is abelian if $s\in S$ is of order $p$. 
In addition, if $s$ is not regular then $C_G(s)$
%is a reductive group 
contains  non-trivial unipotent elements and hence non-abelian
(this follows from \cite[Ch. II, Theorem 4.1]{SS}). Therefore, $s$ is regular, and so is every non-identity element $g\in S_p$.\enp

\bl{u33} Let $G=\SU_3(q)$ and let $\phi$ be an \ir $\ell$-\rep of degree $q^2-q$. Let $Z=Z(U)$, where $U$ is a maximal unipotent subgroup of G. Then $1_Z$ is not a constituent of $\phi|_Z$.  \el

\bp Suppose first that $\ell=0$, and let $\chi$ be the character of $\phi$. Then $\chi(z)=-q$ for $1\neq z\in Z$, see for instance \cite[p. 565]{Ge}.  Then $(1_Z,\chi|_Z)=0$, as required.
Let $\ell>0$. Then $\phi$ is liftable, see e.g. \cite[p. 755]{HM}. As $(\ell,q)=1$,
the result follows. \enp

\bl{mt6} The group $G=E_8(q)$ has a cyclic maximal torus T of order $\Phi_{15}(q)$
and $\Phi_{30}(q)$, 
and every non-trivial element of T is regular. If $(5,q^2+1)=1$ then 
G has a cyclic maximal torus T of order $\Phi_{20}(q)$ and every non-trivial element of it is regular. In particular,   the assumption of Lemma {\rm \ref{20tz}} and {\rm  \cite[Lemma 3.2]{TZ20}} is satisfied for T.  \el 

\bp Observe that $G$ has  maximal tori of orders $\Phi_{15}(q),\Phi_{20}(q),\Phi_{30}(q)$, see \cite[Table C, p.308]{LSS}. Let $T$ be such a torus, and $1\neq t\in T$. Let $\GG$ be a simple algebrac group such that $G=\GG^\F$. Then $C_{\mathbf{G}}(t)$ is connected (as $Z(\mathbf{G})=1)$.
If $t$ is not regular then $C_{\mathbf{G}}(t)$  contains a unipotent element. In addition,
$C_{\GG}(t)$ is reductive and contains a maximal torus of $\GG$.    The maximal subgroups $M$, say, of $G$ containing  $C_G(t)$ are listed in \cite[Table 5.1, p.322]{LSS}. By inspection of their orders we conclude that none of them is divisible by $\Phi_{15}(q)$ or $\Phi_{30}(q)$.

Let $|T|=\Phi_{20}(q)=q^8-q^6+q^4-q^2+1$.
The only maximal subgroups $M$ containing $T$ and a non-trivial unipotent element are  $ SU_5(q^2).4$ and $PSU_5(q^2).4$ \cite[Table 5.1]{LSS}. As $\Phi_{20}(q)$ divides
$ |SU_5(q^2)|=|PSU_5(q^2)|$, it follows that $T$ is conjugate to a maximal torus of
$ SU_5(q^2)$ and $PSU_5(q^2)$. These groups are of the form  $X^\F$,
where $X\subset \GG$ is a simple algebraic group of type $A_4$.
So we can assume that $T$ is a maximal torus in $X^\F$. As  
$C_{\GG}(t)$ is connected, we have  
 $t\in C_G(t)\subset X^\F$, and hence $C_{X^F}(t)$  has a non-trivial unipotent element. 
This implies $t\in Z(X^\F)$,  hence $X^\F=SU_5(q^2)$  and $|t|=5$ (as $t\neq 1$).
This is a contradiction.   \enp

\begin{remar}\label{rem75}
{\em Let $G=E_6(q)$, $(3,q-1)=1$ or ${}^2E_6(q)$, $(3,q+1)=1$. Then $Z(G)=1$ and $G$ has a maximal torus $T$, say, of order  $\Phi_9(q)$ or  $\Phi_{18}(q)$, respectively.
As in the proof of \ref{mt6}, using the fact that no proper parabolic of $G$ is of order
multiple to  $\Phi_9(q)$, respectively,   $\Phi_{18}(q)$, we observe that every non-identity elements of $T$ is regular. It is well known that $G$ is simple.  The pair $(G,T)$ satisfies the assumption of \cite[Lemmas 3.2 and 3.6]{TZ20}. Let $\phi$ be an \ir Brauer character of $G$. Then, by  \cite[Lemma  3.6]{TZ20}, either $\phi$ is liftable, or $\phi|_T=m\cdot \rho_T^{reg}$ for some integer $m>0$ or  $\phi$ is unipotent.  (One can also use 
\cite[Lemma 2.3]{MT}.)}
\end{remar}

% \begin{comment}
\bl{1w1}  Let $G={}^2E_6(q)$ and let $U$
be the root subgroup of $G$ corresponding to the maximal root. Let $P=N_G(U)$ and $L=(P/Q)'$, where $Q$  is  the unipotent radical of $P$. Then $|Q|=q^{21}$ and $L\cong
\SU_6(q)$. Moreover, the conjugation action of  P on Q defines on  $Q/U$ a structure of   an $\FF_qL$-module isomorphic  to  the  third  exterior
power of the natural $\SU_6(q)$-module, and the restriction of $Q/U$ to $\SU_5(q)$ is a direct sum of two   \ir constituents of dimension $10$.  \el

 \bp By  \cite[Propositions 4.6]{CKS} $Q/U$ under the conjugation action of $L$ is an \ir $\FF_{q^2}$-module of dimension 20. Let $\overline{L}$ be the simple algebraic group  of type $A_5$. Then  this extends to an  $\overline{L}$-module. By inspection in \cite{Lu},
$\overline{L}$ has a unique \ir module of dimension 20, and the highest weight of it is
the third fundamental weight $\om_3$. So this module  is selfdual, and   is isomorphic to the third exterior power of the natural 
$SL_6(F)$, where $F$ is the algebraic closure of $\FF_q$. The restriction of this module to
$SL_5(F)$ has a composition factor of dimension 10, the exterior square of the natural $SL_5(F)$-module (see for instance, \cite[Lemma 5.10]{z23}). As $Q/U$ is self-dual, so is its  restriction
to   $SL_5(F)$. As the above factor is not self-dual,   there is another factor of dimension 10.
The restriction of each of them to $SU_5(q)$ is well known to be \irt 
\enp
%\end{comment}

\bl{ee2} Let $G\in\{F_4(q),E_6(q),{}^2E_6(q),E_7(q),E_8(q)\}$ be  of simply connected type, and let $g$ be  a p-element. Let $1_G\neq \phi\in\Irr_\ell(G)$, $\ell\nmid pq$. Suppose that Sylow p-subgroups of G are cyclic. Then $\deg\phi(g)= |g|$ if one the \f holds:

\begin{enumerate}[\rm(i)]

\item $p|\Phi_5(q) $ and $G\in\{E_6(q),E_7(q)\};$

\item   $p|\Phi_8(q) $ and $G\in\{F_4(q),E_6(q),{}^2E_6(q),E_7(q)\};$ 
% unless possibly q is odd even and $\ell=3;$

\item  $p|\Phi_{7}(q)$ and $G\in\{E_7(q),E_8(q)\}$;

\item  $p|\Phi_{9}(q)$,   $G=E_7(q),E_8(q)$, or $G=E_6(q)$ with $3|(q-1);$

\item  $p|\Phi_{10}(q)$ and $G\in\{ E_7(q),{}^2E_6(q)\};$

\item   $p|\Phi_{12}(q)$ and $G\in\{F_4(q),E_6(q),{}^2E_6(q),E_7(q)\};$ 

\item  $p|\Phi_{14}(q)$ and $G\in\{E_7(q),E_8(q)\};$

\item  $p|\Phi_{18}(q)=q^6-q^3+1$. Moreover, either $3|(q+1)$ and $G\in \{{}^2E_6(q),E_7(q)\}$, or $G=E_8(q);$

\item $p|\Phi_{20}(q)$, $G=E_8(q)$, and $5|(q^2+1)$ but $\ell \neq 5$;

\item   $p|\Phi_{24}(q)$ and $G=E_8(q).$ 
\end{enumerate}
\el

\bp    Suppose that $p|\Phi_{5}(q)$. The subgroup
$E_6(q)\subset E_7(q)$ contains a subgroup $H$ isomorphic to $\SL_6(q) $. 
 So we can assume that $g\in H$,   and  the result follows by  % \cite[Lemma 2.3]{TZ20} or 
\cite[Theorem 1.1]{DZ1} as $g$ is contained in a parabolic subgroup of $H$. 

\smallskip
 Let  $p|\Phi_{7}(q)$. Then    $G\in \{ E_7(q), E_8(q)\}$ by Table 3. As $G$ contains a subgroup $H\cong \SL_8(q)$ and $|H|$ is a multiple of $q^7-1$, we can assume $g\in H$.
In fact, $g$ is contained in a subgroup $X\cong \SL_7(q)$.  Then the result follows  from \cite{DZ1} as in case (1).   

\smallskip
 Suppose that $p|\Phi_{8}(q)$. As Sylow $p$-subgroups are cyclic, by Table 3 we have  $G\in \{F_4(q), E_6(q)$, ${}^2E_6(q), E_7(q)\}$. Note that $G=F_4(q)$ contains a subgroup isomorphic to $\Spin_9(q)$, and also $\Spin^-_8(q)$ for $q$ odd; if $q$ is even then $G$ contains $\Sp_8(q)$, see \cite[Table 5.1]{LSS}.  
As $p|\Phi_8(q)=q^4+1$, by Sylow's theorem,  $g$ is contained in a subgroup $H$ of $G$ isomorphic to $\Spin^-_8(q)$, $q$ odd, or $\mathrm{Sp}_8(q)$, $q$ even. So the result follows from Lemma \ref{dn8} for $q$ odd and from Proposition \ref{tt4} for $q$ even, applied to $\phi|_H$.

\smallskip
 Let  $p|\Phi_{9}(q)$ and  $G= E_6(q)$. Then $g$ lies in a subgroup $X$ of $G$ isomorphic to $\SL_3(q^3)$. If $3|(q^3-1)$ then the result follows by \cite[Lemma 2.6]{TZ20}.
Note that $3|(q^3-1)$ \ii   $3|(q-1)$. 
 
Let $p|\Phi_{9}(q)$ and $G=E_7(q)$. Note that $E_6(q)$ lies in a parabolic subgroup $P_7$ of  $E_7(q)$ whose unipotent radical $U$ is abelian (of order $q^{27}$), see for instance \cite[Table 4]{Ko}. As  
 $  C_{\lan g \ran}(U)=1$
% -- what is meant by this? Perhaps $C_{\langle g \rangle}(U)=1$???} 
(see \cite[\S 13.2]{GL}),  we have $\deg\phi(g)=|g|$ by Higman's lemma (or see \cite[Lemma 2.3]{TZ20}).  This also rules out the case with $p|\Phi_{9}(q)$ and $G=E_8(q)$.
 
\smallskip
 Let $p|\Phi_{10}(q)$. Then $G\in\{ {}^2E_6(q), E_7(q)\}$. Let $P,Q,L$ be subgroups of 
${}^2E_6(q) \subset  E_7(q)$ introduced in Lemma \ref{1w1}, and $H\cong \SU_5(q)\subset L\cong \SU_6(q)$. Observe that $|H|$ is a multiple of $\Phi_{10}(q)$. As Sylow $p$-subgroups of $G$ are cyclic (see Table 3), it follows that $g$ is conjugate to an element of $H$. So we can assume that $g\in H$. 
  
The unipotent radical $Q$ of $P$ is of extraspecial type,   
and the conjugation action of $P$ on $Q$  yields on $V:=Q/Z(Q)$, viewed as an $\FF_qL$-module,  a structure of an (absolutely irreducible self-dual) $\FF_{q}L$-module of dimension 20, the 3rd exterior power of the natural $\FF_{q^2}H$-module (Lemma \ref{1w1}). 
 
Let $\tau$ be an \ir constituent of $\phi|_{QL}$ non-trivial on $Z(Q)$.     
 By  Theorem \ref{hhs}, if $\deg\tau(g)<|g|$ then % has $|g|$ distinct \eis unless 
$g$ acts irreducibly on $V/C_V(g)$.  This is  false  
as, by Lemma \ref{1w1}, the restriction of $V$ to $H$ has two \ir constituents  dual to each other. 
  
%\med
 If $p|\Phi_{12}(q)$ then $G\in\{F_4(q),E_6(q),{}^2E_6(q),E_7(q)\}$ by Table 3. Each of these groups has a subgroup isomorphic to   $H={}^3D_4(q)$, see \cite[Table 5.1]{LSS}. As  $\Phi_{12}(q) $ divides the order of $H$, we conclude   that $g$ is conjugate to an element of $ H$.  Then the result  follows from \cite{TZ20}.

\smallskip 
  Suppose that $p|\Phi_{14}(q)$ and $G= E_7(q) \subset E_8(q)$. It is known that $ \SL_2(q^7) \subset E_7(q) $ (see for instance \cite[p. 927]{HSTZ}). 
As $\Phi_{14}(q)$ divides $|\SL_2(q^7)|$, it follows 
$g$ is conjugate to an element of $H\cong  \SL_2(q^7)$. If $\deg\phi(g)<|g|$
then $\deg\tau(g)<|g|$ for every \ir \ccc $\tau$ of $\phi|_H$. By \cite[Lemma 3.3]{TZ8},
$|g|=(q^7+1)/(2,q+1)$. However,  $(q^7+1)/(2,q+1)$ is not a prime power. (Indeed, 
$q^7+1=(q+1)\Phi_{14}(q)$ and $\Phi_{14}(q)=q^6-q^5+q^4-q^3+q^2-q+1\equiv 7\pmod{q+1}$, and then $q+1$ would be a 7-power, contrary to Lemma \ref{zgm}.)

\smallskip
%Let $G=E_7(q)$.  
% or $E_8(q)$. ????????????????? ${}^3D_4(q^2)\subset E_8(q)$

 Let $p|\Phi_{18}(q)=q^6-q^3+1$ and
  $G\in\{ {}^2E_6(q), E_7(q),E_8\}$.

Suppose first that $G={}^2E_6(q)$. Note that $G$ contains a group $H\cong \SU_3(q^3)$
\cite[Table 10]{C23}. Then 
$g$ is conjugate to an element in $H$. Note that $p>3$ and $p \nmid (q^3+1)$ as Sylow $p$-subgroups of $G$ are cyclic.  Let $\tau$ be a non-trivial \ir \ccc of $\phi|_H$. Suppose the contrary, that $\deg\phi(g)<|g|$. Then $\deg\tau(g)<|g|$, and,
 by  \cite[Proposition 6.1]{TZ8},  $(3,q^3+1)=1$, $|g|=q^6-q^3+1$ 
and  % $\dim\tau
$\deg\tau=|g|-1$. This implies the result for $G={}^2E_6(q)$ and $E_7(q)$.
  
Note that % As $\Phi_{18}(q)$ divides $|{}^2E_6(q)|$ and ${}^2E_6(q)\subset 
%$E_7(q)$ lies in a parabolic subgroup of $G$, we can   use Lemma \ref{1w1}. Indeed, 
$E_7(q)$ is subgroup of a parabolic subgroup $P_8$ of $G$ whose unipotent radical $Q$ is a group of extraspecial type of order $q^{57}$, see \cite[Proposition 4.4]{CKS}. Let $L$ be a Levi subgroup of $P_8$ and $L'$ the derived subgroup of $L$. Then  $L'\cong E_7(q)$. As $Q/Z(Q)=q^{56}$,
and the minimum dimension of a non-trivial \irr of $E_7(q)$ in characteristic $r$ is known to be $56$, it follows that $E_7(q)$ acts on $Q/Z(Q)$
irreducibly. (Alternatively see \cite[Table 5]{Ko}.) %Note that $L'\subset P'_8$, the de
Let  $\si$ be an \ir constituent of $\phi|_{L'Q}$ non-trivial on $Z(Q)$. As a \syl is cyclic,
$g$ is conjugate to an element of $L'$, so we can assume that $g\in L'$. 
 By  Theorem \ref{hhs}, $\si(g)$ has $|g|$ distinct \eis unless 
$g$ acts irredicibly on $W:=Q/C_Q(g)$, when viewing $W$ as a vector space over $\FF_r$. 
Note that $W$ is in fact an $\FF_q L'$-module (see \cite[Lemma 7]{ABS}) and $g$ remains \ir  on this module. This property is studied in \cite{TZ21} and \cite{z24}. Recall that $W$ extends to a module $\mathbf{W}$ over the algebraic group of type $E_7$. By  \cite[Theorem 2]{TZ21},
$g$ is regular in  $E_7$.   By  \cite[Theorem 1.2]{z24}, $\dim C_{\mathbf{W}}(g)\leq 2$
and hence $\dim C_{W}(g)\leq 2$ (over $\FF_q$). %This implies $|C_Q(g)/Z(Q)|\leq q^2$. 
By assumption, $|g|$ divides $\Phi_{18}(q)|(q^{9}+1)$, and hence the dimension of a non-trivial \ir $\FF_q\lan g\ran$-module equals 18.  This is a contradiction as $\dim_{\FF_q} W=56$. 

\smallskip
Suppose that $p|\Phi_{20}(q)$ and $G=E_8(q)$.  As $  \SU_{5}(q^2)\subset   E_8(q)$ \cite[p. 322]{LSS} and $\Phi_{20}(q)$ divides $|\SU_{5}(q^2)|$, we can assume that  $g\in H\cong \SU_{5}(q^2)$.
Then the result follows from  Proposition \ref{z00} applied to a non-trivial \ir constituent of $\phi|_H$. 
%Note that Proposition \ref{z00} implies $Z(H)=1$, whence $(5,q^2+1)=1$, equivalently, $q\equiv 0,\pm 1\pmod 5$.  
%\itf every non-identity element of the maximal torus $S$ that contains $g$ is regular. 

%As ${}^2E_6(q) \subset  G$,  we have the same conclusion. 

\smallskip  Finally let  $p|\Phi_{24}(q)$.  As $\Phi_{24}(q)$ divides $|{}^3D_4(q^2)|$ and 
${}^3D_4(q^2) \subset  E_8(q)$, by Sylow's theorem, $g$ is conjugate to an element of ${}^3D_4(q^2)$, so the result follows from \cite{TZ20}.\enp

\begin{comment}
Let $G=E_8(q)$.   Then $\Phi_{18}(q)$ divides $|{}^2E_6(q)|$ and ${}^2E_6(q)\subset E_7(q)$ lies in a parabolic subgroup of $G$. More precisely, by \cite[Proposition 4.4]{CKS},  $E_7(q)$ lies in a parabolic subgroup $P_8$ of $G$ whose unipotent radical $Q$ is a group of extraspecial type of order $q^{57}$. Let $L$ be a Levi subgroup of $P_8$ and $L'$ the derived subgroup of $L$. Then  $L'\cong E_7(q)$. As $Q/Z(Q)=q^{56}$, and the minimal dimension of a non-trivial \irr of $E_7(q)$ is known to be $56$, it follows that $E_7(q)$ acts on $V:=Q/Z(Q)$
irreducibly (we view $Q/Z(Q)$ as a vector space over $\FF_q$). (Alternatively, see \cite[Table ]{Ko})Let  $\si$ be an \ir constituent of $\phi|_{P'_8}$ non-trivial on $Z(Q)$. Then $g\in P_8'$.  By  Theorem \ref{hhs}, $\si(g)$ has $|g|$ distinct \eis unless 
$g$ acts irreducibly on $V/C_V(g)$. This yields a contradiction as the restriction of $V$ to ${}^2E_6(q)$ contains two  \ir constituents of dimension 27.\end{comment}

\bl{aa2} Let $G\in\{E_6(q),{}^2E_6(q),E_7(q),E_8(q)\}$ be a group of simply connected type, and let $g \in G$ be  a p-element. Let $1_G\neq \phi\in\Irr_\ell(G)$
with $\ell \nmid pq$. Suppose that Sylow p-subgroups of G are cyclic. Then $\deg\phi(g)\geq |g|-1$. In addition, $\deg\phi(g)= |g|$ unless one of the \f holds:

\begin{enumerate}[\rm(i)]
\item $G=E_6(q)$, $(3,q-1)=1$ and $|g|=q^6+q^3+1;$

\item $G={}^2E_6(q)$ or $E_7(q)$, $(3,q+1)=1$ and $|g|=q^6-q^3+1;$

\item $G=E_8(q)$, $|g|=\Phi_i(q)$ with $i\in\{15,20,30\}$; moreover, if $i=20$ then either $5 \nmid (q^2+1)$ or $\ell=5|(q^2+1)$.
\end{enumerate}
In addition, the Brauer character $\be$ of $\phi$ is in a unipotent block.
% unless, possibly,  $G=E_8(q)$ and $i=20$. 
\el
 
\bp By Lemma \ref{ee4}, the triple $(G,p,i)$ are as in Table 3, and the order of a  \syl of $G$ divides $\Phi_i(q)$.  We only need to examine the cases that have not been treated in Lemma \ref{ee2}.  We proceed case-by-case. 
  
\smallskip
(a) Suppose that $p|\Phi_{9}(q)$,   $G= E_6(q)$, $3|(q-1)$ or $p|\Phi_{18}(q)$, $G= {}^2E_6(q)$ and $3|(q+1)$.  By \cite[p. 322]{LSS}, $G$ contains a subgroup  $H\cong \SL_3(q^3) $, respectively,  $\SU_3(q^3)$. As $\SL_3(q^3)$ is a multiple of $\Phi_{9}(q)$, respectively, $\SU_3(q^3)$ is a multiple of $\Phi_{18}(q)$, 
it follows that we can assume that $g\in H$.   By \cite[Proposition 6.1(ii)]{TZ8}, if $1_H\neq \tau\in\Irr_\ell (H)$, $g\in H$ and $\deg\tau(g)<|g|$ then   $|g|=(q^9-1)/(q^3-1)$, respectively, $(q^9+1)/(q^3+1)$   and $\deg\tau(g)=\dim\tau=|g|-1$ and $1$ is not an \ei of $\tau(g)$.

%If $(3,q-1)=1$ then $G$ is simple. 
As we explained in Remark \ref{rem75},   every non-identity element of maximal torus of order $\Phi_9(q)$ in $E_6(q)$, respectively, $\Phi_{18}(q)$ in ${}^2E_6(q)$ is regular and  $\be$ is either  unipotent or liftable. In the latter case $\deg\phi(g)=|g|$ by \cite{Z3}. 
%by \cite[Lemma 3.6(i)]{TZ20}. 
So the result follows.

% Suppose that $p|\Phi_{18}(q)$ and $G= {}^2E_6(q)$ with $(3,q+1)=1$, so $G$ is simple. Then $G$ contains a subgroup $H\cong  SU_3(q^3)$ and we can assume that $g\in H$. Then the statement follows from \cite[Proposition 6.1(ii)]{TZ8}.  

 \smallskip
(b) Suppose that $p|\Phi_{20}(q)$.  As $  \SU_{5}(q^2)\subset   E_8(q)$ \cite[p. 322]{LSS} and $\Phi_{20}(q)$ divides $|\SU_{5}(q^2)|$, we can assume that  $g\in H\cong \SU_{5}(q^2)$.
Then the result follows from  Proposition \ref{z00}   applied to a non-trivial \ir constituent of $\phi|_H$. 
%Note that Proposition \ref{z00} implies $Z(H)=1$, whence $(5,q^2+1)=1$, equivalently, $q\equiv 0,\pm 1\pmod 5$. 
By Lemma \ref{mt6}, every non-identity element of the maximal torus $S$ of $G$ with $g\in S$  is regular.  By Lemma \ref{20tz}, %\cite[Lemma 3.6(i)]{TZ20}, 
$\be$ is either   liftable or unipotent and $\deg\phi(g)=|g|-1$. In the former case $\deg\phi(g)=|g|$ by \cite{Z3}.  

\smallskip
(c) We are left with the \f cases: $E_8(q)$, $i=15,30$. By Lemma \ref{mt6},  $G$ has cyclic maximal tori   $\Phi_{15}(q)$ and $\Phi_{30}(q)$, both satisfy the assumption of Lemma \ref{20tz}. By this lemma, we are done unless $\phi $ is liftable. In the latter case $\deg\phi(g)=|g|$ by \cite{Z3}. \enp
 
%\med
\bp [Proof of Theorem {\rm \ref{mm1}}]  
 This follows from Propositions \ref{zu3}, \ref{tt4} and \ref{zo1} for classical groups  and from Lemmas \ref{ee2} and  \ref{aa2} for exceptional groups.  
\enp

\begin{remar} {\em Observe that, in the situation of Theorem \ref{mm1}, $\deg\phi(g)= |g|-1$ implies that $\phi(g)$ does not  have \ei 1. Note that $p$  is odd. If $P$ is a cyclic Sylow $p$-subgroup of $G$ then $N_G(P)\neq C_G(P)$, see \cite[Theorem 39.1]{Asch}. As usual, we may assume that $P =\langle g \rangle$.
Then we can find some $x \in N_G(P)$ such that  $xgx\up=g^i$ for some $i>1$, and   $(|x|,p)=1$. As $|{\rm Aut}(P)|=(p-1)m$, where $m$ is a $p$-power, $x^{p-1}\in C_G(P)$,
so we can assume that  $i|(p-1)$. 
\itf  $\lam^i$ is an \ei of   $\phi(g)$ \ii so is $\lam$, and $\lam^i\neq \lam$ for $\lam\neq 1$.  %(Indeed, $\lam^i=\lam$ implies $\lam^{i-1}=1$ and $i\equiv 1\pmo%d p$. ) %Note that $\phi(g)$ has an \ei $\lam$ which is a primitive $|g|$-root of unity. 
Now, if $\lam \neq 1$ is not an \ei of $\phi(g)$ then so
is $\lam^i$, and thus $\deg\phi(g)\leq |g|-2$, a contradiction. } 
\end{remar}

\bp [Proof of Theorem {\rm \ref{mm2}}] 
Suppose first that $G=\SL^\ep_n(q)$. By Proposition \ref{zu3}, either (i) or (ii) of Theorem 
\ref{mm2} holds or $G=\SL_n(q)$, $q$  odd and $n$ is a 2-power. The latter case is ruled out by Corollary \ref{ss2}.   

Suppose that $G\in \{\Spin^-_{2n}(q),\Spin_{2n+1}(q),\Spin^+_{2n+2}(q)\}$, $n\geq 3$.  Then the result follows from Proposition \ref{zo1}. %, for    $n=3$  see Lemma \ref{o37}.  

For exceptional groups of Lie type the result  follows, as above,  from Lemmas \ref{ee2} and  \ref{aa2}. \enp %, together with \cite{TZ20}.

 \section{Universal central extensions of finite simple groups of Lie type}

The above results do not cover the case where a simple group $G$ of Lie type has
a non-standard universal central extension $\tilde G$. It is known that this happens exactly when $|Z(\tilde G)|$ is a multiple of $r$, the defining characteristic of $G$. These cases are well known, see for instance \cite[p. 313]{GLS}.  We list these simple groups below,  indicating the structure of $Z(\tilde G)$ and the primes $p$ such that a \syl of $G$ is cyclic.

\med\noindent
$(\SL_2(4),  C_2), p=3,5$; $(\PSL_2(9),C_6),p=5$;  
  $(\SL_3(2), C_2),p=3,7$;  $(\SL_4(2), C_2),p=5,7$; $(\PSL_3(4)$, $C_{12} \times C_{4}),p=5,7$;  $(\PSU_4(3)$, $C_{12}\times C_3), p=5,7$; $(\PSU_6(2), C_6\times C_2),p=5,7,11$; $(\Sp_6(2),C_2),p=5,7$;
$(\Omega_8^+(2),C_2\times C_2),p=7$; $(G_2(4),C_2),p=7,13$; $( F_4(2), C_2),p=13,17$; $({}^2E_6(2),C_6\times C_2),p=11,13,17,19$; 
 $(\Omega_7(3),C_6),p=5,7,13$; $(G_2(3), C_3),p=7,13$; $ ({}^2 B_2(8), C_2\times C_2), p=13$.  (\itf  $p\leq 19$.)

\med

For completeness we provide information on these cases. Observe that the Brauer characters are available in \cite{JLPW}
for all the above groups except the non-trivial central extensions of the simple groups $ F_4(2), \Omega_7(3)$ and ${}^2E_6(2)$.  For these groups the decomposition matrices are provided in \cite{Lu} except for $\ell=3,5,7$ for $ {}^2E_6(2)$.

\bp[Proof of Proposition {\rm \ref{m9t}}]  Let $p=19$.  Then $G/Z(G)={}^2E_6(2)$. Then $G$ contains a subgroup $H$ with $H/(H\cap Z(G))\cong \PSU_3(8)$, and the order of $H$ is a multiple of 19. So we can assume that $g\in H$. By \cite[Proposition 6.1(ii)]{TZ8}, an \ir  $p$-element $x\in X=\SU_3(q)$ has $|x|$ distinct \eis in every non-trivial \ir $F$-\rep of $X$ (for $q>2$) whenever 3 divides $Z(X)$. As $X\cong \SU_3(8)$, this is the case,
and hence $\deg\phi(g)=|g|$. 

%(Note that  As $\phi$ is faithful, we can assume that$Z(G)$ is of  order 2 and $|g|=19$. 
%\itf $\deg\phi(g)=19$ for $g\in G$ of order 19.????Let $P$ be a parabolic subgroup of $G/Z(G)$ with Levi isomorphic to $SU_6(2)$,and $U=O_2(P)$ the unipotent radical of $P$.  Note that $[P,U]=U$,  U is extraspecial, and  $|Z(U)|=2$.Let $P_1,U_1,Z$ be the preimage of $P,U,Z(U)$ in $G$, resp. We can assume $Z(G)=\phi(Z(G)$, and then $|Z|\leq 4$. Therefore, $Z$ is abelian. We claim that $Z=Z(U_1)$.Suppose the contrary.Then $|U_1:C_(U_1)|=2$ and hence $[P_1,U_1]\neq U_1$.

Let $p=17$.  Then $G/Z(G)\in\{F_4(2),{}^2E_6(2)\}$. Then $\Sp_4(4) \subset  \Sp_8(2) \subset F_4(2) \subset {}^2E_6(2)$, see \cite{Atlas}, and we can assume that $g\in H=\Sp_4(4)$.
By Lemma \ref{tt4}, $\deg\tau(g)=|g|$ for $g\in \Sp_4(4)$ with $|g|=17$ and $\tau\in\Irr_\ell H$. So the result follows by applying this fact to a non-trivial \ir \ccc of $\phi|_H$. 

Let $p=13$.  Then $G/Z(G)\in\{{}^2B_2(8),G_2(4), F_4(2)$, ${}^2E_6(2), G_2(3)$, $\Omega_7(3)\}$. 
If $G/Z(G)\in\{{}^2B_2(8), G_2(4)$, $G_2(3)$, $\Omega_7(3)\}$ then the result follows by \cite{JLPW}. Observe that $H={}^3D_4(2) \subset   F_4(2) \subset   {}^2E_6(2)$ and $H$ has an element of order 13. Therefore, we can assume that $g\in H$. By \cite{JLPW}, $\deg\tau(g)=|g|$ for $g\in H$ with $|g|=13$. So the result follows by applying this fact to a non-trivial \ir \ccc of $\phi|_H$. 
 
Let $p=11$.  Then $G/Z(G)\in\{\PSU_6(2), {}^2E_6(2)\}$.  Using the decomposition matrix in \cite{Lu}, one observes that   $\deg\phi(g)=|g|$ for $g\in G$ with $|g|=5,7,11$.  
As ${}^2E_6(2)$ contains $\PSU_6(2),$ the result follows in this case. So we can ignore this group also for $p=5,7$.

Let $p=7$.  Then 
$$G/Z(G)\in\{\SL_3(2),  \SL_4(2), \PSL_3(4),\Sp_6(2),\PSU_4(3),\PSU_6(2), \Omega_8^+(2),G_2(3),G_2(4)\}.$$  
 For the groups $G/Z(G)$ other than $PSU_6(2),$ one can check the proposition by the Brauer character tables in \cite{JLPW}.

 Let $p=5$.  Then $G/Z(G)\in\{\SL_2(4),   \SL_4(2),\Sp_6(2),\PSL_3(4),\PSU_4(3),
\PSU_6(2)\}$. The case with $G/Z(G)=\PSU_6(2)$ has been already ruled out. For the other groups one can use the Brauer character tables in \cite{JLPW}.

Let $p=3$. Then $G/Z(G)\in\{\SL_2(4),\SL_3(2)\}$ and the result follows by inspection in \cite{JLPW}. \enp

\end{document}